\documentclass[twoside, 11pt]{article}
\usepackage{amsfonts}

\usepackage{geometry}
\usepackage{bm}
\usepackage{cite}
\usepackage{latexsym}
\usepackage{enumerate,verbatim}
\usepackage{amsfonts}
\usepackage[english]{babel}
\usepackage[ansinew]{inputenc}
\usepackage{amsmath,amsthm,amsfonts,amssymb}
\usepackage{graphics}
\usepackage{mathrsfs}

\usepackage{graphicx}
\usepackage[font=small,labelfont=bf]{caption}

\usepackage{color}

\usepackage{algorithm}
\usepackage{algorithmic}

\usepackage{appendix}

\usepackage{tikz}
\usepackage{pgfplots}

\usepackage{tablefootnote}

\usepackage{subcaption}

\usepackage{multirow}
\usepackage{booktabs}
\usepackage{caption}
\usepackage{siunitx}
\usepackage{cases}
\usepackage{amsmath}

\usepackage{color}
\usepackage{bm}
\usepackage{url}
\usepackage{hyperref}
\hypersetup{
  colorlinks=true,
  linkcolor=blue,
  filecolor=blue,
  anchorcolor=blue,
  urlcolor=blue,
  citecolor=purple
}

 \newtheorem{theorem}{Theorem}[section]
 
 \newtheorem{lemma}[theorem]{Lemma}
 \newtheorem{prop}[theorem]{Proposition}
 \theoremstyle{definition}
 \newtheorem{definition}[theorem]{Definition}
 \theoremstyle{remark}
 \newtheorem{remark}[theorem]{Remark}
 \theoremstyle{eg}

 \theoremstyle{fact}
 
\numberwithin{equation}{section}

\def\be{\begin{equation}}
\def\ee{\end{equation}}
\def\bea{\begin{eqnarray}}
\def\eea{\end{eqnarray}}

\newcommand{\N}{\mathbb N}

\DeclareMathOperator*{\argmin}{argmin}

\DeclareMathOperator{\C}{\mathcal{C}}

\begin{document}
\title{{\bf
\Large
Dynamic Proximal Gradient Algorithms for\\ Schatten-$p$ Quasi-Norm Regularized Problems
}
}
\author{Weiping Shen\thanks{Weiping Shen and Linglingzhi Zhu contributed equally to this work and are listed alphabetically.}\ \thanks{School of Mathematical Sciences, Zhejiang Normal University, Jinhua 321004, P. R. China (shenweiping@zjnu.cn). This author's work was supported in part by the National Natural Science Foundation of China (grant 12071441).}
 \and Linglingzhi Zhu\footnotemark[1]\ \thanks{H. Milton Stewart School of Industrial and Systems Engineering, Georgia Institute of Technology, Atlanta, GA 30332, USA (llzzhu@gatech.edu).}
 \and Yaohua Hu\thanks{School of Mathematical Sciences, Shenzhen University, Shenzhen 518060, P. R. China (mayhhu@szu.edu.cn). This author's work was supported in part by the National Key Research and Development Program of China (2025YFA1016903), Shenzhen Medical Research Fund (B2502001), National Natural Science Foundation of China (12571327, 12426311), Project of Educational Commission of Guangdong Province (2025ZDZX2059, 2023ZDZX1017), and Shenzhen Science and Technology Program (RCJC20221008092753082, JCYJ20241202124209011).}
\and Chong Li\thanks{Corresponding author. School of Mathematical Sciences, Zhejiang University, Hangzhou 310058, P. R. China (cli@zju.edu.cn).
This author's work was supported in part by the National Natural Science Foundation of China (grant 11971429)  and   Zhejiang
Provincial Natural Science Foundation of China (grant LY18A010004).}
\and Xiaoqi Yang\thanks{Department of Applied Mathematics, The Hong Kong Polytechnic University, Kowloon, Hong Kong (mayangxq@polyu.edu.hk). This author's work was supported in part by the Research Grants Council of Hong Kong (PolyU 15205223 and $\mbox{N}\underbar{\hskip.15cm}$PolyU507/22).}
}

\date{}

\maketitle

\begin{abstract}

This paper investigates numerical solution methods for the Schatten-$p$ quasi-norm regularized problem with $p \in [0,1]$, which has been widely studied for finding low-rank solutions of linear inverse problems and gained successful applications in various mathematics and applied science fields. We propose a dynamic proximal gradient algorithm that, through the use of the Cayley transformation, avoids computationally expensive singular value decompositions at each iteration, thereby significantly reducing the computational complexity. The algorithm incorporates two step size selection strategies: an adaptive backtracking search and an explicit step size rule. We establish the sublinear convergence of the proposed algorithm for all $p \in [0,1]$ within the framework of the Kurdyka-{\L}ojasiewicz property. Notably, under mild assumptions, we show that the generated sequence converges to a stationary point of the objective function of the problem. For the special case when $p=1$, the linear convergence is further proved under the strict complementarity-type regularity condition commonly used in the linear convergence analysis of the forward-backward splitting algorithms. Preliminary numerical results validate the superior computational efficiency  of the proposed algorithm.

\end{abstract}


\section{Introduction}

Let $m, n, l $ be positive integers with $m\ge n$.  Let $\bm{b}\in\mathbb{R}^{l}$ and $\mathcal{A}:\mathbb{R}^{m\times n}\rightarrow\mathbb{R}^{l}$ be a linear map. Let $p\in[0,\;1]$ and the  Schatten-$p$ quasi-norm $\Vert \bm{X}\Vert_{S_p}$ for $\bm{X}\in\mathbb{R}^{m\times n}$ be defined by
$$\Vert \bm{X}\Vert_{S_p}=\left\{\begin{array}{ll}{\rm rank}(\bm{X}),& p=0;\\
\left(\sum\limits_{i=1}^{n} \sigma_i^p(\bm{X})\right)^{\frac{1}{p}},& p\in(0,1],\end{array}\right.$$where $\{\sigma_i(\bm{X})\}_{i=1}^n$ are the singular values of $\bm{X}\in\mathbb{R}^{m\times n}$.
The Schatten-$p$ quasi-norm regularized problem considered in the present paper  is the following optimization problem:
\begin{equation}\label{Problem-Lp}
\min\limits_{\bm{X}\in\mathbb{R}^{m\times n}}  \mathcal{F}(\bm{X}):=
\frac{1}{2}\Vert \mathcal{A}(\bm{X})-\bm{b}\Vert^{2}+\lambda \Vert \bm{X}\Vert_{S_p}^p,
\end{equation}
where $\Vert \bm{X}\Vert_{S_0}^0:=\Vert \bm{X}\Vert_{S_0}$.
This  regularized problem has been studied extensively  (cf. \cite{Fazel2002, CandesRecht2009,Candes2010,Recht2010,Mohan2012, Ma2011, peng2017s,Peng2018,Lu2014,Lu2017,zeng2023proximal}) and  demonstrated its efficient applications in various areas such as
  low-rank matrix completion \cite{CandesRecht2009}, robust principal component analysis \cite{CandesLiMaWright2011}, low-degree statistical models for random processes \cite{Rohde2011estimation}, the low-order realization of linear control systems \cite{fazel2001rank}, system identification \cite{liu2010interior}, image processing \cite{xie2016weighted}, various machine learning tasks \cite{nie2012low,fan2019factor,lefkimmiatis2023learning} and so on.

Our main interest in the present paper is focused on the issue of numerically solving the  Schatten-$p$ quasi-norm regularized problem
 \eqref{Problem-Lp}. To the best of our knowledge, most of the existing and important   algorithms for solving \eqref{Problem-Lp} are counterparts of the ones for solving  the vector $\ell_p$-regularized problem, including the fixed point  iterative algorithm (or forward-backward splitting algorithm, or proximal gradient algorithm) \cite{HaleYin2007,Hale2007,Li2015a,Li2015b}, the smoothing majorization method \cite{Chen2010, Chen2012} and the iterative reweighted  minimization algorithm \cite{Daubechies2010,LaiXuYin13,Lu2014MP}. More precisely, in the case when $p=1$, the fixed point  iterative algorithm was proposed by Ma et al.  in \cite{Ma2011} to solve the problem \eqref{Problem-Lp}  where they proved that any sequence generated by the algorithm converges to a solution of \eqref{Problem-Lp}. In the case when $p\in(0,1)$, the corresponding  fixed point  iterative algorithm  for solving   \eqref{Problem-Lp}  was proposed   in  \cite{peng2017s} for $p=\frac12$ and further extended  in \cite{Peng2018} to the general case $p\in(0,1)$;  while
 the smoothing majorization method and the iterative reweighted singular value minimization algorithm  were
proposed for solving   \eqref{Problem-Lp} in  \cite{Lu2014} and \cite{Lu2017}, respectively.
Unlike the case of the vector  $\ell_p$-regularized problem, in the case when $p\in(0,1)$, only the  following convergence property is proven for   the above mentioned   algorithms:
any
limit point of the generated sequence  is a stationary point $\bm{X}^*$ of the problem \eqref{Problem-Lp}, i.e., $0\in \partial  \mathcal{F}(\bm{X}^*)$  and no convergence rates are obtained.

 Recently, to overcome the non-Lipschitzian of $\Vert \bm{X}\Vert_{S_p}^p$ in the case when $p\in(0,1)$,  a proximal linearization method (PLM in short) via adding a rank constraint was proposed by Zeng \cite{zeng2023proximal} to solve problem \eqref{Problem-Lp}. The PLM involves at each iteration   solving the following constrained  optimization subproblem:
  \begin{equation*}\label{Problem-Lp-wang2}
\min\limits_{\scriptsize{\begin{array}{c}
               \bm{X}\in\mathbb{R}^{m\times n} \\
               {\rm rank}(\bm{X})\le r_k
             \end{array}} }
             \mathcal{H}(\bm{X}):=\frac{1}{2}\Vert \mathcal{A}(\bm{X})-\bm{b}\Vert^{2}+\lambda\sum_{j=1}^{r_k}w_{k,j}\sigma_j(\bm{X})
+\frac\tau2\|\bm{X}-\bm{X}_k\|_F^2
\end{equation*} (see    \cite[Algorithm 1]{zeng2023proximal}) where $r_k:={\rm rank}(\bm{X}_k)$ and $w_{k,j}:=p\sigma_j^{p-1}(\bm{X}_k)$ ($1\le j\le r_k$).
 It was claimed in \cite[Theorem 4.16]{zeng2023proximal} that a sequence $\{\bm{X}_k\}$ generated by the PLM converges to a stationary point of \eqref{Problem-Lp}
 under the assumption that there exist $\widetilde{r}$ and $\widetilde{k}$   such that   $\{\bm{X}_k\}$ satisfies   \begin{equation*}\label{Problem-Lp-wang3} {\rm rank}(\bm{X}_k)=\widetilde{r}, \;  \mbox{ and }\;
\mbox{$\sigma_1(\bm{X}_k)>\cdots>\sigma_{\widetilde{r}}(\bm{X}_k)$, \;  $\forall k\ge \widetilde{k}$}.
\end{equation*}

Note further that all algorithms mentioned above for solving \eqref{Problem-Lp} (including the cases $p=1$ and $p\in(0,1)$), including the PLM proposed in \cite{zeng2023proximal},
  require to compute  a singular value decomposition (SVD in short) for the generated  matrix $\bm{X}_k$ at each iteration. As is well known (see for e.g., \cite{shang2016scalable,zhang2018lrr,fan2019factor,Golub1990,giampouras2020novel,Halko2011, Tao2022,Onuki2017}),  the SVD suffers from a high computational cost and restricts their scalability to large-scale problems. Thus, how to design an efficient algorithm without  SVD  at each iteration is a challenging task. The main purpose of the present paper is to propose an algorithm avoiding the computation of SVD. For this purpose, we note that if $\hat{ \bm{X}} $ with SVD, i.e., $\hat{\bm{X}}=\hat{\bm{U}}^\top \bm{D}(\hat{\bm{\sigma}})\hat{\bm{V}}$ is an approximating solution to problem \eqref{Problem-Lp}, then  \eqref{Problem-Lp} is equivalent to the following minimization problem:
\begin{equation}\label{Problem-eq}\min\limits_{(\bm{\sigma},\bm{E},\bm{F})\in\mathbb{R}^{n}\times\mathcal{S}(m)\times\mathcal{S}(n) } \mathcal{F}_{\hat{\bm{\Omega}}}(\bm{\sigma},\bm{E},\bm{F}):=\frac{1}{2}\Vert \mathcal{A}((\hat{\bm{U}}e^{\bm{E}})^\top \bm{D}(\bm{\sigma}) \hat{\bm{V}}e^{\bm{F}})-\bm{b}\Vert^{2}+\lambda\Vert \bm{\sigma} \Vert_{p}^p \end{equation} (in the sense that their optimal values are equal), where
 $\mathcal{S}(n)$ is the space of all $n\times n$ real skew-symmetric matrices, $\bm{D}(\bm{\sigma}) \in \mathbb{R}^{m\times n}$ represents a diagonal matrix with elements taken from the vector $\bm{\sigma}$,
 $\Vert \bm{\sigma} \Vert_{p}$ is the $\ell_p$ norm of $ \bm{\sigma}$ and $\Vert \bm{\sigma} \Vert_{0}^0:=\Vert \bm{\sigma} \Vert_{0}$ denotes
the number of non-zero elements of $\bm{\sigma}$.
 Based on this equivalence and making use of the Cayley transform (cf. \cite{Cayley1846}), we design, for all $p\in[0,1]$, a completely new algorithm which simultaneously generates $ \bm{X}_k $ and its SVD (so the algorithm is SVD-free), and then analyze the convergence property.
The main  contributions of the present paper are as follows:

 (i). A dynamic proximal gradient algorithm is proposed with two kinds of  strategies of selecting   stepsizes:
 one is chosen by the ``backtracking search", and the other is given in an explicit form.

 (ii). The Kurdyka-{\L}ojasiewicz (K{\L} in short) property for  the dynamic objective  function $\mathcal{F}_{\hat{\bm{\Omega}}}$ involved in  the   problem \eqref{Problem-eq} is established, and then
 sublinear convergence of the proposed algorithm is established for all $p\in[0,1]$ (see Theorem \ref{thm2}).
 Under mild assumptions, we further show that any sequence generated by the  proposed algorithm converges to a stationary point of the function $\mathcal{F}$ in the original problem \eqref{Problem-Lp} (see Theorem \ref{remark}).

 (iii). In the special case when $p=1$,      a linear convergence rate (see Theorem \ref{thm2p=1} in Section 4) for the proposed algorithm is established  under the standard assumption which is commonly used in  the linear convergence analysis of forward-backward splitting algorithms for solving  \eqref{Problem-Lp}; see \cite{liang2014local,zhou2017unified,drusvyatskiy2018error} for details.

(iv). Preliminary numerical experiments demonstrate the effectiveness of the proposed algorithm for solving \eqref{Problem-Lp}. Compared with the  fixed point iterative methods~\cite{Ma2011,peng2017s,Peng2018}, our algorithm attains comparable solution accuracy while requiring only about one-tenth (or less) of the CPU time, and consistently exhibits a much higher convergence success rate and improved numerical stability.

The paper is organized as follows. In Section \ref{sec:alg}, we present the necessary notation  and propose the dynamic  proximal gradient algorithm, including the algorithm framework and the strategies of selecting the stepsizes. Sections \ref{sec:descent} and \ref{sec:convergence} present the descent property and the convergence analysis of the proposed algorithm, including the sublinear convergence property for general $p\in [0,1]$ and the linear convergence property for $p=1$ under mild assumptions, respectively. Preliminary numerical results are provided in Section \ref{sec:numerical}. Some useful lemmas and their
technical proofs are deferred to the Appendix.

\section{Dynamic Proximal Gradient Algorithm}\label{sec:alg}

\subsection{Notation}
The notation used here is standard; see for example \cite{peng2017s,rockafellar2009variational}. Let $\N$ be the set of all
nonnegative integers. Let $\mathbb{R}^n$ denote the $n$-dimensional Euclidean space endowed with the Euclidean norm $\|\cdot\|$. Let $\mathbb{R}^{m\times n}$ be the space of all real  $m\times n$  matrices endowed with the inner product $\langle\cdot,\cdot\rangle$ defined by $$\langle \bm{X},\bm{Y}\rangle:={\rm tr}(\bm{X}^\top \bm{Y})\quad \text{for any} \quad\bm{X},\;\bm{Y}\in\mathbb{R}^{m\times n},$$
and so the corresponding induced norm on  $\mathbb{R}^{m\times n}$  is the Frobenius norm denoted by $\Vert\cdot\Vert_F$.  The spaces of all real skew-symmetric matrices  and  orthogonal matrices in $\mathbb{R}^{n\times n}$
are denoted by $\mathcal{S}(n)$ and $\mathcal{O}(n)$, respectively.
 $\bm{I}$ denotes the  identity matrix.
 $\mathcal{D}(\bm{x})\in\mathbb{R}^{m\times n}$ denotes the diagonal matrix with $\bm{x}\in\mathbb{R}^{n}$ being its diagonal elements, and ${\rm diag}({\bm A})\in\mathbb{R}^{n}$ represents the vector formed by extracting the  diagonal elements from the matrix  $\bm{ A}$. For a linear operator $T$  between two Euclidean spaces, we denote by $T^*$ and $\Vert T\Vert_{{\rm op}}$ the adjoint of $T$ and its operator norm, respectively.
For $\bm{X}\in\mathbb{R}^{m\times n}$,  the singular values of $\bm{X}$ are denoted by  $\{\sigma_i(\bm{X})\}_{i=1}^n$  being ordered with $\sigma_1(\bm{X})\ge\sigma_2(\bm{X})\ge\cdots\ge\sigma_n(\bm{X})\ge0$. Moreover,  we write
$\bm{\sigma}(\bm{X}):=(\sigma_1(\bm{X}),\sigma_2(\bm{X}),\ldots,\sigma_n(\bm{X}))^\top$
and define
\begin{equation*}\label{Xi}
\Xi(\bm{X}):=\Big\{(\bm{U},\bm{V})\in \mathcal{O}(m)\times\mathcal{O}(n)\ |\ \bm{X}=\bm{U}^\top \mathcal{D}(\bm{\sigma}(\bm{X}))\bm{V}\Big\}.
\end{equation*}
For simplicity,  we use $\mathcal{M}$ to denote the  Cartesian product space $\mathbb{R}^{n}\times\mathcal{S}(m)\times \mathcal{S}(n)$   equipped with  the inner product
  $$\langle(\cdot,\cdot,\cdot),(\cdot,\cdot,\cdot)\rangle:=\langle \cdot,\cdot\rangle+\langle \cdot,\cdot\rangle+\langle \cdot,\cdot\rangle$$ and the induced product norm  $\Vert(\cdot,\cdot,\cdot)\Vert$.

 For convenience, we define the function $f:\mathbb{R}^{m\times n}\rightarrow\mathbb{R}$  by
\begin{equation}\label{fXdef}
f({\bm X}):=\frac{1}{2}\Vert \mathcal{A}(\bm{X})-\bm{b}\Vert^{2}\quad{\rm for \;each}\;\bm{X}\in\mathbb{R}^{m\times n}.
\end{equation}
Obviously, the gradient of $f$, denoted by $\nabla f$,  is given by
\begin{equation}\label{fXnabla}\nabla f({\bm X})=\mathcal{A}^*(\mathcal{A}\left(\bm{X}\right)-\bm{b})\quad\mbox{for each}\; {\bm X}\in\mathbb{R}^{m\times n}.\end{equation}
Associated to the function $f$ defined in \eqref{fXdef}, we define, for each $\bm{\Omega}:=(\bm{U},\bm{V})\in\mathcal{O}(m)\times\mathcal{O}(n)$,  the auxiliary function
$f_{\bm{\Omega}}:\mathcal{M}\rightarrow \mathbb{R}$ by
\begin{equation}\label{fomega}
f_{\bm{\Omega}}(\bm{\sigma},\bm{E},\bm{F}):=f((\bm{U}e^{\bm{E}})^\top \mathcal{D}(\bm{\sigma}) \bm{V}e^{\bm{F}})\quad{\rm for \;each}\;(\bm{\sigma},\bm{E},\bm{F})\in\mathcal{M}.
\end{equation}
Clearly,  $f_{\bm{\Omega}}$ is analytic on $\mathcal{M}$, and its gradient at $(\bm{\sigma},\bm{0},\bm{0})$, $\nabla f_{\bm{\Omega}}(\bm{\sigma},\bm{0},\bm{0}):=(\nabla_{\bm{\sigma}}f_{\bm{\Omega}}(\bm{\sigma},\bm{0},\bm{0}), \nabla_{\bm{E}}f_{\bm{\Omega}}(\bm{\sigma},\bm{0},\bm{0}),$ $\nabla_{\bm{F}}f_{\bm{\Omega}}(\bm{\sigma},\bm{0},\bm{0}))$ with the partial derivatives with respect to $\bm{\sigma}$, $\bm{E}$  and $\bm{F}$  being given respectively as
\begin{equation}\label{fparsigma}
\begin{cases}
    \nabla_{\bm{\sigma}} f_{\bm{\Omega}}(\bm{\sigma}, \bm{0}, \bm{0}) = \text{diag}\left( \bm{U} \nabla f(\bm{X}) \bm{V}^\top \right), \\
    \nabla_{\bm{E}} f_{\bm{\Omega}}(\bm{\sigma}, \bm{0}, \bm{0}) = \frac{1}{2} \left[ \bm{X} \nabla f(\bm{X})^\top - \nabla f(\bm{X}) \bm{X}^\top \right], \\
    \nabla_{\bm{F}} f_{\bm{\Omega}}(\bm{\sigma}, \bm{0}, \bm{0}) = \frac{1}{2} \left[ \bm{X}^\top \nabla f(\bm{X}) - \nabla f(\bm{X})^\top \bm{X} \right],
\end{cases}
\end{equation}
where $\bm{X}:=\bm{U}^\top\mathcal{D}(\bm{\sigma})\bm{V}$ (as shown  in Proposition \ref{prop1}).

We end this subsection with the symbols  ${\bf{B}}(\bm{x},r)$ and ${\rm d}(\bm{x},\mathcal{Z})$, in a metric space $\mathcal{X}$,   to denote the open ball centered at $\bm{x}\in\mathcal{X}$ with radius $r>0$, and  the distance function from $\bm{x}\in\mathcal{X}$ to the set $\mathcal{Z}\subseteq\mathcal{X}$ respectively.

\subsection{Algorithm Framework}
As described in the introduction, the  original problem \eqref{Problem-Lp} is equivalent  to
the  minimization problem \eqref{Problem-eq} on $\mathcal{M}$, i.e.,
\begin{equation*}\min\limits_{(\bm{\sigma},\bm{E},\bm{F})\in\mathbb{R}^{n}\times\mathcal{S}(m)\times\mathcal{S}(n) } \mathcal{F}_{\hat{\bm{\Omega}}}(\bm{\sigma},\bm{E},\bm{F}):=f_{\hat{\bm{\Omega}}}(\bm{\sigma},\bm{E},\bm{F})+\lambda\Vert \bm{\sigma} \Vert_{p}^p,\end{equation*}
assuming that $\hat{\bm{X}}$ is an approximate solution of \eqref{Problem-Lp} with SVD: $\hat{\bm{X}}=\hat{\bm{U}}^\top \mathcal{D}(\hat{\bm{\sigma}})\hat{\bm{V}}$ and $\hat{\bm{\Omega}}=(\hat{\bm{U}},\hat{\bm{V}})$.
 Note that the  objective function  $\mathcal{F}_{\hat{\bm{\Omega}}}$ possesses the special structure: $\mathcal{F}_{\hat{\bm{\Omega}}}(\hat{\bm{\sigma}},\cdot,\cdot)$  is analytic on $\mathcal{S}(m)\times \mathcal{S}(n)$, while $\mathcal{F}_{\hat{\bm{\Omega}}}(\cdot,\bm{0},\bm{0})$ is the sum of a quadratic function and the $p$-norm of vectors. Then, to get a better approximation of the solution, we apply
 the (one step) gradient algorithm and the
 proximal gradient algorithm  to $\mathcal{F}_{\hat{\bm{\Omega}}}(\hat{\bm{\sigma}},\cdot,\cdot)$ and $\mathcal{F}_{\hat{\bm{\Omega}}}(\cdot,\bm{0},\bm{0})$, respectively.
Thus, incorporating the Cayley transformation, we propose the following dynamic  proximal gradient algorithm (DPGA) for solving  the Problem \eqref{Problem-Lp}.

\begin{algorithm}[h]
	\caption{Dynamic Proximal Gradient Algorithm (DPGA)}
    \label{alg1}
	\begin{algorithmic}[1]
		\STATE {\bf Input:} $(\bm{\sigma}_{0},\bm{U}_{0},\bm{V}_{0})\in\mathbb{R}^n\times\mathcal{O}(m)\times\mathcal{O}(n)$,  $k:=0$. Set $\bm{X}_{0}:=\bm{U}_{0}^\top \mathcal{D}(\bm{\sigma}_{0})\bm{V}_{0}$. \\
    \WHILE {stopping criterion isn't met}		
    \STATE Set  $\bm{\Omega}_k:=(\bm{U}_k,\bm{V}_k)$ and compute  $\nabla f_{\bm{\Omega}_k}(\bm{\sigma}_k,\bm{0},\bm{0})$ according to \eqref{fparsigma}.
    	\STATE Choose stepsizes $t_k,\;s_k>0$, and set \begin{equation}\label{ExpE}
\bm{E}_{k}:=- s_k \nabla_{\bm{E}} f_{\bm{\Omega}_k}(\bm{\sigma}_k,\bm{0},\bm{0}),\quad\bm{F}_{k}:=-s_k\nabla_{\bm{F}} f_{\bm{\Omega}_k}(\bm{\sigma}_k,\bm{0},\bm{0}) .
    \end{equation}
		\STATE Generate $(\bm{\sigma}_{k+1},\; \bm{U}_{k+1},\; \bm{V}_{k+1})$ according to
    \begin{numcases}{}
    \bm{\sigma}_{k+1} = \argmin_{\bm{\sigma}}\left\{\frac{1}{2t_{k}} \Vert \bm{\sigma}-\bm{\sigma}_{k}+t_{k}\nabla_{\bm{\sigma}} f_{\bm{\Omega}_k}(\bm{\sigma}_k,\bm{0},\bm{0})\Vert^{2} +\lambda\Vert \bm{\sigma} \Vert_{p}^p \right\} \label{eq:sigma} \\
    \left(\bm{I}+\frac{1}{2}\bm{E}_{k}\right)\bm{U}_{k+1}^\top = \left(\bm{I}-\frac{1}{2}\bm{E}_{k}\right)\bm{U}_{k}^\top \label{eq:U} \\
    \left(\bm{I}+\frac{1}{2}\bm{F}_{k}\right)\bm{V}_{k+1}^\top = \left(\bm{I}-\frac{1}{2}\bm{F}_{k}\right)\bm{V}_{k}^\top \label{eq:V}
\end{numcases}
    \STATE  Compute $\bm{X}_{k+1}:=\bm{U}_{k+1}^\top \mathcal{D}(\bm{\sigma}_{k+1})\bm{V}_{k+1}$.
    \STATE Update $k:=k+1$.
    \ENDWHILE
	\end{algorithmic}
\end{algorithm}

\begin{remark}
The main computation of implementing Algorithm \ref{alg1} is solving
the  optimization subproblem  \eqref{eq:sigma}  and the linear systems \eqref{eq:U} and \eqref{eq:V}. Note that solving the optimization subproblem \eqref{eq:sigma} amounts to computing the proximity operator
of the function $\lambda t_{k}\Vert \cdot \Vert_{p}^p$ at the point $\bm{\sigma}_{k}-t_{k}\nabla_{\bm{\sigma}} f_{\bm{\Omega}_k}(\bm{\sigma}_k,\bm{0},\bm{0})$
which  has been well discussed (cf. \cite{Beck2009,Blumensath08,Bredies2008,Daubechies04,XuZB12,XuZB23,hu2017group}), and  particularly, in the case  when $p= 0, 1/2, 2/3, 1$, the analytical
solutions of the proximity operator  were  provided  in \cite{Daubechies04,Blumensath08,XuZB12, XuZB23}.
The linear systems \eqref{eq:U} and \eqref{eq:V} are widely known as the Cayley transformation, which has
been long used in matrix computations such as inverse eigenvalue problems and inverse singular values problems \cite{Friedland1987,Chu1992}.
For large-scale problems, as remarked in \cite{Bai2004},  solving \eqref{eq:U} and \eqref{eq:V} iteratively achieves computational efficiency with
only a few iterations  because, the coefficient matrices on the left hand side of  both linear systems approach  the identity
matrix in the limit due to the fact that  $\{\bm{E}_{k}\}$ and $\{\bm{F}_{k}\}$ converge to zero as explained in Remark \ref{rmk:cluster} below.
\end{remark}

\subsection{Strategies of Selecting Stepsizes} \label{stepsizes}
The strategy for selecting the stepsizes $\{t_k\}$ and $\{s_k\}$    plays an important role in the convergence analysis of the proposed algorithm. In our approach, we expect to
 select  $\{t_k\}$ and $\{s_k\}$ such that the  following ``sufficient descent property'' of the generated  sequence $\{\mathcal{F}(\bm{X}_{k})\}$ holds  for some $\alpha>0$ and all $k\in \N$:
\begin{equation}\label{H_1}
\mathcal{F}(\bm{X}_{k+1})+\alpha\Vert (\bm{\sigma}_{k+1}-\bm{\sigma}_{k},\bm{E}_{k},\bm{F}_{k})\Vert^{2}\leq \mathcal{F}(\bm{X}_{k}).
\end{equation}
We provide the following two  strategies at the $k$-th iteration:

\begin{itemize}
  \item [\rm (i)](Backtracking search strategy) Choose $\rho_t\in(0,1)$ and $\rho_s\in(0,\;(\max\{\| \nabla_{\bm{E}}^k\|,\;\| \nabla_{\bm{F}}^k\|\})^{-1})$.  Set
  $$(t_k,\;s_k):=\max\{(\rho_t^{i},\;\rho_s^{j})\;:\;i,\;j\in\N \;{\rm and}\; i\le j\le i+1 \}$$
  such that the sufficient descent property  \eqref{H_1} holds, in the sense that  $(a,\;b)\le(c,\;d)$ if and only if $a\le c$ and $b\le d$.
  \item [\rm (ii)](Explicit stepsize strategy) Choose $\mu\in(0,\frac12]$. The stepsizes $t_k$ and $s_k$ are directly chosen  to satisfy the following two assumptions:
\begin{itemize}
  \item [($A_1$)]$\frac{\mu}{l_{\sigma}^k}\leq t_k \leq \frac{1-\mu}{l_{\sigma}^k}$;
  \item [($A_2$)]$\min\left\{1,\;\mu\bar{s}_k,\;\frac{1}{\| \nabla_{\bm{E}}^k\|_F},\;\frac{1}{\|\nabla_{\bm{F}}^k\|_F}\right\}\le s_k\leq \min\left\{\frac{1}{\mu},\;(1-\mu)\bar{s}_k,\;\frac{1}{\| \nabla_{\bm{E}}^k\|_F},\;\frac{1}{\| \nabla_{\bm{F}}^k\|_F}\right\}$,
\end{itemize}
\end{itemize}
where\begin{equation}\label{xfbound}
\bar{s}_k:=\frac{2}{\sqrt{(l_\Omega^k)^2+2c_k}+l_\Omega^k}.
\end{equation}
\begin{equation}\label{nablaekfkdefi}\nabla_{\bm{E}}^k:=\nabla_{\bm{E}} f_{\bm{\Omega}_k}(\bm{\sigma}_k,\bm{0},\bm{0}),\quad \nabla_{\bm{F}}^k:=\nabla_{\bm{F}} f_{\bm{\Omega}_k}(\bm{\sigma}_k,\bm{0},\bm{0}),\end{equation}
\begin{equation}\label{Ak}
l_{\sigma}^k:=\|\mathcal{A}^*\mathcal{A}\|_{{\rm op}}+2\Vert\nabla f(\bm{X}_k) \Vert_{F}+2\|\mathcal{A}^*\mathcal{A}\|_{{\rm op}}\Vert\bm{\sigma}_k\Vert,
\end{equation}
\begin{equation}\label{Bk}
l_{\Omega}^k:=\|\mathcal{A}^*\mathcal{A}\|_{{\rm op}}\cdot\|\bm{\sigma}_k\|^2+(\frac12+\|\bm{\sigma}_k\|)\|\nabla f(\bm{X}_k)\|_F+\frac12\|\bm{\sigma}_k\|,
\end{equation}
\begin{equation}\label{Bk2}
c_k:=\left(\|\mathcal{A}^*\mathcal{A}\|_{{\rm op}}\cdot\|\bm{\sigma}_{k+1}\|^2+\|\mathcal{A}^*(\bm{b})\|_F\cdot\|\bm{\sigma}_{k+1}\|\right)(\| \nabla_{\bm{E}}^k\|_F+\|\nabla_{\bm{F}}^k\|_F),
\end{equation}

Compared with strategy (ii), strategy (i) can find the almost maximum stepsizes  possessing  the sufficient descent property, which can be implemented within  finite iterations as remarked before Proposition \ref{prop4}.

\section{Sufficient Descent Property of the Proposed Algorithm}\label{sec:descent}

We begin with the  following lemma which is  a trivial extension of \cite[3.2.12]{Ortega1970} and so, the proof is omitted.

\begin{lemma}\label{lemma1}
Let $\mathbb{V}$ be a finite dimensional normed linear space and $h:\mathbb{V}\rightarrow\mathbb{R}$ be continuously differentiable. Let $\bm{x},\bm{y}\in \mathbb{V}$ and suppose that there exists $L>0$ such that
\begin{equation*}\label{grad-lemma}
\langle \nabla h(\bm{x}+t\bm{y})-\nabla h(\bm{x}),\; \bm{y}  \rangle\leq Lt\Vert \bm{y}\Vert^2,\quad \forall t\in[0,1].
\end{equation*}
Then
  \begin{displaymath}
  h(\bm{x}+\bm{y})- h(\bm{x})-\left\langle \nabla h(\bm{x}),\; \bm{y} \right\rangle\leq\frac{L}{2}\Vert \bm{y}\Vert^{2}.
  \end{displaymath}
\end{lemma}

Recall that the auxiliary function $f_{\bm{\Omega}}$  is defined by \eqref{fomega}. The following proposition  provides the partial derivative formulas for  $f_{\bm{\Omega}}$.

\begin{prop}\label{prop1}
Let  $\bm{\Omega}:=(\bm{U},\bm{V})\in\mathcal{O}(m)\times\mathcal{O}(n)$ and $(\bm{\sigma},\bm{E},\bm{F})\in \mathcal{M}$. Set $\bm{X}:=\bm{U}^\top\mathcal{D}(\bm{\sigma})\bm{V}$ and $\bm{Y}:=(\bm{U}e^{\bm{E}})^\top\mathcal{D}(\bm{\sigma})\bm{V}e^{\bm F}$. Let $\bm{S}\in\mathcal{S}(m),\bm{T}\in\mathcal{S}(n)$ which commute with $\bm{E}$ and $\bm{F}$ respectively, i.e., $\bm{E}\bm{S}=\bm{S}\bm{E}$ and $\bm{F}\bm{T}=\bm{T}\bm{F}$. Then the following assertions hold:
\begin{itemize}
\item [\rm(i)]$\nabla_{\bm{\sigma}}f_{\bm{\Omega}}(\bm{\sigma},\bm{E},\bm{F})
={\rm diag}\left(\bm{U}e^{\bm{E}} \nabla f(\bm{Y})(\bm{V}e^{\bm{F}})^\top\right)={\rm diag}\left(\bm{U}e^{\bm{E}} \mathcal{A}^*(\mathcal{A}\left(\bm{Y}\right)-\bm{b})(\bm{V}e^{\bm{F}})^\top\right)$;
  \item [\rm(ii)]$\langle\nabla_{\bm{E}}f_{\bm{\Omega}}(\bm{\sigma},\bm{E},\bm{F}),\bm{S}\rangle=\langle \bm{Y} \nabla f(\bm{Y})^\top, \bm{S}\rangle=\langle \bm{Y} \left(\mathcal{A}^*(\mathcal{A} \left(\bm{Y}\right)-\bm{b})\right)^\top, \bm{S}\rangle;$
  \item [\rm(iii)]$\langle\nabla_{\bm{F}}f_{\bm{\Omega}}(\bm{\sigma},\bm{E},\bm{F}),\bm{T}\rangle=\langle\bm{Y}^\top\nabla f(\bm{Y}), \bm{T}\rangle=\langle\bm{Y}^\top\mathcal{A}^*(\mathcal{A}\left(\bm{Y}-\bm{b})\right), \bm{T}\rangle.$
\end{itemize}
In particular, one has that
\begin{itemize}
  \item[\rm(iv)] $\nabla_{\bm{E}}f_{\bm{\Omega}}(\bm{\sigma},\bm{0},\bm{0})=\frac{1}{2}[\bm{X}\nabla f(\bm{X})^\top-\nabla f(\bm{X}){\bm{X}}^\top]=\frac{1}{2}[\bm{X}(\mathcal{A}^*\left(\mathcal{A} (\bm{X})-\bm{b}\right))^\top-\mathcal{A}^*\left(\mathcal{A}(\bm{X})-\bm{b}\right)\bm{X}^\top ],$
  \item[\rm(v)] $\nabla_{\bm{F}}f_{\bm{\Omega}}(\bm{\sigma},\bm{0},\bm{0})=\frac{1}{2}[\bm{X}^\top\nabla f(\bm{X})-\nabla f(\bm{X})^\top\bm{X}]=\frac{1}{2}[\bm{X}^\top\mathcal{A}^*\left(\mathcal{A} (\bm{X})-\bm{b}\right)-(\mathcal{A}^*\left(\mathcal{A} (\bm{X})-\bm{b}\right))^\top\bm{X}]$.
\end{itemize}

\end{prop}

\begin{proof} Let $\bm {\omega}\in\mathbb{R}^n$. Since $f_{\bm{\Omega}}$ is analytic on $\mathcal{M}$ as noted early,  one checks  by the definition of the directional derivative  that
\begin{align}
\langle \nabla_{\bm{\sigma}}f_{\bm{\Omega}}(\bm{\sigma},\bm{E},\bm{F}),
\bm{\omega}\rangle
&={\rm tr}\left(\bm{V}e^{\bm{F}} \left(\mathcal{A}^*\mathcal{A}\left((\bm{\bm{U}}e^{\bm{E}})^\top \mathcal{D}(\bm{\sigma})\bm{V}e^{\bm{F}}\right)^\top-\mathcal{A}^*(\bm{b})^\top\right)(\bm{U}e^{\bm{E}})^\top \mathcal{D}(\bm{\omega})\right)\notag\\
&=\Big\langle{\rm diag}\left(\bm{U}e^{\bm{E}}\left(\mathcal{A}^*\mathcal{A}\left((\bm{U}e^{\bm{E}})^\top \mathcal{D}(\bm{\sigma})\bm{V}e^{\bm{F}}\right)-\mathcal{A}^*(\bm{b})\right)(\bm{V}e^{\bm{F}} )^\top\right),\bm{\omega}\Big\rangle;\notag
\end{align}
hence assertion (i) is seen to hold as  $\bm {\omega}$ is arbitrary.

 Since the proof of assertion (iii) (resp. (v)) is similar to that of assertion (ii) (resp. (iv)), below we only prove assertions (ii) and (iv) for brevity.  Note by Lemma \ref{lemmaexplog} (iii) that
\begin{equation*}\label{Zassenhaus}
e^{\bm{E}+t\bm{S}}=e^{\bm{E}}e^{t \bm{S}}=e^{\bm{E}}\left(\bm{I}+t\bm{S}\right)+o(t).
\end{equation*}
Then, as we do for $\langle \nabla_{\bm{\sigma}}f_{\bm{\Omega}}(\bm{\sigma},\bm{E},\bm{F}),
\bm{\omega}\rangle$, we have
\begin{equation}\label{diff-E-20}
\begin{aligned}
\langle \nabla_{\bm{E}}f_{\bm{\Omega}}(\bm{\sigma},\bm{E},\bm{F}), \bm{S}\rangle
&=\Big\langle \mathcal{A} \left((\bm{U}e^{\bm{E}})^\top \mathcal{D}(\bm{\sigma})\bm{V}e^{\bm{F}}\right)-\bm{b}, \mathcal{A}\left(\bm{S}^\top(\bm{U}e^{\bm{E}})^\top \mathcal{D}(\bm{\sigma})\bm{V}e^{\bm{F}}\right)  \Big\rangle \\
&=\Big\langle (\bm{U}e^{\bm{E}})^\top \mathcal{D}(\bm{\sigma})\bm{V} e^{\bm{F}} \left(\mathcal{A}^*\mathcal{A} \left((\bm{U}e^{\bm{E}})^\top \mathcal{D}(\bm{\sigma})\bm{V}e^{\bm{F}}\right)-\mathcal{A}^*(\bm{b})\right)^\top, \bm{S}\Big\rangle.
\end{aligned}
\end{equation}
Hence,  assertion (ii) is proved.
In particular, letting $\bm{E}=\bm{F}={\bm 0}$ in \eqref{diff-E-20}, one has that,  for any $\bm{S}\in\mathcal{S}(m)$,
\begin{align*}\label{diff-E-20a}
\langle \nabla_{\bm{E}}f_{\bm{\Omega}}(\bm{\sigma},{\bm 0},{\bm 0}), \bm{S}\rangle
=\Big\langle \bm{X}\nabla f(\bm{X})^\top, \bm{S}\Big\rangle=\Big\langle\frac{1}{2}(\bm{X}\nabla f(\bm{X})^\top-\nabla f(\bm{X}){\bm{X}}^\top),\;\bm{\bm{S}}\Big\rangle.
\end{align*}
This implies that assertion (iv) holds (noting that $\frac{1}{2}(\bm{X}\nabla f(\bm{X})^\top-\nabla f(\bm{X}){\bm{X}}^\top)\in\mathcal{S}(m)$) and the proof is completed.
\end{proof}

The following  proposition  shows some Lipschitz-like properties related to the partial derivatives of  $f_{\bm{\Omega}}$.
For convenience, we write for any $\bm{X}\in\mathbb{R}^{m\times n}$ and $\bm{\sigma}\in\mathbb{R}^{n}$ that
\begin{equation}\label{fdef}
\Gamma_{\bm{X}}(\bm{\sigma}):=\Vert\nabla f(\bm{X})\Vert_F+\Vert\mathcal{A}^*\mathcal{A}\Vert_{{\rm op}}\Vert \bm{\sigma}\Vert.
\end{equation}

\begin{lemma}\label{prop2}
Let   $\bm{\omega}\in\mathbb{R}^{n}$, $\bm{\Omega}:=(\bm{U},\bm{V})\in\mathcal{O}(m)\times\mathcal{O}(n),\;\bm{\Theta}:=(\bm{P},\bm{Q})\in\mathcal{O}(m)\times\mathcal{O}(n)$ and $(\bm{\sigma},\;\bm{E},\;\bm{F})\in \mathcal{M}$. Set $\bm{X}:=\bm{U}^\top \mathcal{D}(\bm{\sigma})\bm{V}$. Then the following assertions hold:
\begin{itemize}
  \item [\rm (i)]$\Vert\nabla_{\bm{\sigma}} f_{\bm{\Omega}}(\bm{\sigma},\bm{0},\bm{0})-\nabla_{\bm{\sigma}} f_{\bm{\Theta}}(\bm{\omega},\bm{0},\bm{0}) \Vert\le \Vert \mathcal{A}^* \mathcal{A}\Vert_{\rm{op}}\Vert \bm{\sigma}-\bm{\omega}\Vert+\Gamma_{\bm{X}}(\bm{\sigma})\|\bm{U}-\bm{P}\|_F+\Gamma_{\bm{X}}(\bm{\omega})\|\bm{V}-\bm{Q}\|_F$;
  \item [\rm (ii)]$\Vert\nabla_{\bm{E}} f_{\bm{\Omega}}(\bm{\sigma},\bm{0},\bm{0})-\nabla_{\bm{E}} f_{\bm{\Theta}}(\bm{\omega},\bm{0},\bm{0})\Vert_{F} \leq \Gamma_{\bm{X}}(\bm{\omega})(\|\bm{\sigma}-\bm{\omega}\|+\|\bm{\sigma}\|\cdot\| \bm{U}-\bm{P}\|_F+\|\bm{\omega}\|\cdot\| \bm{V}-\bm{Q}\|_F);$
  \item [\rm (iii)]$\Vert\nabla_{\bm{F}} f_{\bm{\Omega}}(\bm{\sigma},\bm{0},\bm{0})-\nabla_{\bm{F}} f_{\bm{\Theta}}(\bm{\omega},\bm{0},\bm{0}) \Vert_{F} \leq \Gamma_{\bm{X}}(\bm{\omega}) (\|\bm{\sigma}-\bm{\omega}\|+\|\bm{\sigma}\|\cdot\| \bm{U}-\bm{P}\|_F+\|\bm{\omega}\|\cdot\| \bm{V}-\bm{Q}\|_F).$
 \item [\rm (iv)]$\Vert \nabla_{\bm{\sigma}} f_{\bm{\Omega}}(\bm{\sigma},\bm{E},\bm{F})- \nabla_{\bm{\sigma}} f_{\bm{\Omega}}(\bm{\sigma},\bm{0},\bm{0})\Vert\leq \Gamma_{\bm{X}}(\bm{\sigma}) (\Vert \bm{E}\Vert_{F} +\Vert \bm{F}\Vert_{F})$;
  \item [\rm (v)]$\Vert\nabla_{\bm{\sigma}}f_{\bm{\Omega}}(\bm{\sigma},\bm{E},\bm{F})-\nabla_{\bm{\sigma}}f_{\bm{\Omega}}(\bm{\omega},\bm{E},\bm{F})\Vert \leq \Vert \mathcal{A}^*\mathcal{A}\Vert_{\rm{op}}\cdot\Vert\bm{\sigma}-\bm{\omega}\Vert$;
  \item [\rm (vi)]$\langle\nabla_{\bm{E}}f_{\bm{\Omega}}(\bm{\sigma},\bm{E},\bm{F})-\nabla_{\bm{E}}f_{\bm{\Omega}}(\bm{\sigma},\bm{0},\bm{0}),\bm{E}\rangle
\leq \Gamma_{\bm{X}}(\bm{\sigma})\Vert \bm{\sigma} \Vert(\Vert\bm{E}\Vert_F^2+\Vert\bm{E}\Vert_F\Vert\bm{F}\Vert_F)$;
  \item[\rm(vii)]$\langle\nabla_{\bm{F}}f_{\bm{\Omega}}(\bm{\sigma},\bm{E},\bm{F})-\nabla_{\bm{F}}f_{\bm{\Omega}}(\bm{\sigma},\bm{0},\bm{0}),\bm{F}\rangle
\leq \Gamma_{\bm{X}}(\bm{\sigma})\Vert \bm{\sigma} \Vert (\Vert\bm{F}\Vert_F^2+\Vert\bm{E}\Vert_F\Vert\bm{F}\Vert_F)$.
\end{itemize}
\end{lemma}

\begin{proof}(i).  Set $\bm{Z}:=\bm{P}^\top \mathcal{D}(\bm{\omega})\bm{Q}$.  Then, using a standard technique, we have
\begin{equation}\label{X-Z}
\Vert\bm{X}-\bm{Z}\Vert_F\le\|\bm{\sigma}-\bm{\omega}\|+\|\bm{\sigma}\|\cdot\| \bm{U}-\bm{P}\|_F+\|\bm{\omega}\|\cdot\| \bm{V}-\bm{Q}\|_F.
\end{equation}
By  \eqref{fparsigma} and using a standard argument, one has
\begin{equation}\label{nabsig}
\begin{aligned}
&\Vert\nabla_{\bm{\sigma}} f_{\bm{\Omega}}(\bm{\sigma},\bm{0},\bm{0})-\nabla_{\bm{\sigma}} f_{\bm{\Theta}}(\bm{\omega},\bm{0},\bm{0}) \Vert\\
\le\ & \Vert \bm{U}\nabla f(\bm{X})\bm{V}^\top-\bm{P}\nabla f(\bm{X})\bm{Q}^\top\Vert_F+\Vert\bm{P}[\nabla f(\bm{X})-\nabla f(\bm{Z})]\bm{Q}^\top\Vert_F\\
\leq\ & \Vert\nabla f(\bm{X})\Vert_F\| \bm{U}-\bm{P}\|_F+\Vert\nabla f(\bm{X})\Vert_F\| \bm{V}-\bm{Q}\|_F+\Vert\mathcal{A}^*\mathcal{A}\Vert_{{\rm op}}\Vert \bm{X}-\bm{Z}\Vert_F.
\end{aligned}
\end{equation}
Substituting \eqref{X-Z} into \eqref{nabsig} and using \eqref{fdef}, one checks that assertion (i) holds.

(ii) and (iii). We only need to prove assertion (ii) as the proof of (iii) is similar and is omitted. By  \eqref{fparsigma}, we deduce
\begin{align*}\label{lipnablaE}
\Vert\nabla_{\bm{E}} f_{\bm{\Omega}}(\bm{\sigma},\bm{0},\bm{0})-\nabla_{\bm{E}} f_{\bm{\Theta}}(\bm{\omega},\bm{0},\bm{0}) \Vert_F
&\leq \Vert\bm{X}\nabla f(\bm{X})^\top-\bm{Z}\nabla f(\bm{Z})^\top\Vert_F\notag\\
&\leq (\Vert\nabla f(\bm{X})\Vert_F+\Vert\mathcal{A}^*\mathcal{A}\Vert_{{\rm op}}\Vert \bm{Z}\Vert_F)\Vert \bm{X}-\bm{Z}\Vert_F\notag\\
&\leq \Gamma_{\bm{X}}(\bm{\omega}) \Vert \bm{X}-\bm{Z}\Vert_F,
\end{align*}
where the last inequality holds because of \eqref{fdef} and the fact that $\Vert\bm{Z}\Vert_F=\Vert\bm{\omega}\Vert$. Thus, assertion (ii) is seen to hold by \eqref{X-Z}.

(iv) and (v). Set $\bm{\Theta}_1:=(\bm{P}_1, \bm{Q}_1)=(\bm{U}e^{\bm{E}},\bm{V}e^{\bm{F}})$.   Then,
by Proposition \ref{prop1} (i), one checks that
\begin{equation}\label{Theta1}\nabla_{\bm{\sigma}}f_{\bm{\Omega}}(\bm{\sigma},\bm{E},\bm{F})
=\nabla_{\bm{\sigma}}f_{\bm{\Theta}_1}(\bm{\sigma},\bm{0},\bm{0})
\quad{\rm and }\quad
\nabla_{\bm{\sigma}}f_{\bm{\Omega}}(\bm{\omega},\bm{E},\bm{F})
=\nabla_{\bm{\sigma}}f_{\bm{\Theta}_1}(\bm{\omega},\bm{0},\bm{0}).\end{equation}
Thus, applying assertion (i) (with $\bm{\sigma}$, $\bm{\Theta}_1$ in place of $\bm{\omega}$ and $\bm{\Theta}$), we have
$$\Vert \nabla_{\bm{\sigma}} f_{\bm{\Omega}}(\bm{\sigma},\bm{E},\bm{F})- \nabla_{\bm{\sigma}} f_{\bm{\Omega}}(\bm{\sigma},\bm{0},\bm{0})\Vert
\le\Gamma_{\bm{X}}(\bm{\sigma})(\|\bm{U}-\bm{U}e^{\bm{E}}\|_F+\|\bm{V}-\bm{V}e^{\bm{F}}\|_F)\le\Gamma_{\bm{X}}(\bm{\sigma})(\|\bm{E}\|_F+\|\bm{F}\|_F)$$
where the last inequality holds because of  Lemma \ref{lemmaexplog} (iv). That is, assertion (iv) is shown. To show (v),  applying assertion (i) (with $\bm{\Theta}_1$ in place of $\bm{\Omega}$ and $\bm{\Theta}$), we have
$$\Vert\nabla_{\bm{\sigma}}f_{\bm{\Theta}_1}(\bm{\sigma},\bm{0},\bm{0})- \nabla_{\bm{\sigma}} f_{\bm{\Theta}_1}(\bm{\omega},\bm{0},\bm{0})\Vert\leq \Vert \mathcal{A}^*\mathcal{A}\Vert_{\rm{op}}\cdot\Vert\bm{\sigma}-\bm{\omega}\Vert.$$
Hence, assertion (v) holds by \eqref{Theta1}.

(vi) and (vii). We only need to prove assertion (vi) as the proof of (vii) is similar and is omitted.   By Proposition \ref{prop1} (ii) and the definition of  $\bm{\Theta}_1$, one checks
 $$\langle\nabla_{\bm{E}}f_{\bm{\Omega}}(\bm{\sigma},\bm{E},\bm{F}),\bm{E}\rangle=\langle \nabla_{\bm{E}}f_{\bm{\Theta}_1}(\bm{\sigma},\bm{0},\bm{0}),\bm{E}\rangle.$$
Then, thanks to assertion (ii), we deduce
$$\begin{array}{lll}|\langle\nabla_{\bm{E}}f_{\bm{\Omega}}(\bm{\sigma},\bm{E},\bm{F})-\nabla_{\bm{E}}f_{\bm{\Omega}}(\bm{\sigma},\bm{0},\bm{0}),\bm{E}\rangle|
&\le\Vert\nabla_{\bm{E}}f_{\bm{\Theta}_1}(\bm{\sigma},\bm{0},\bm{0})-\nabla_{\bm{E}}f_{\bm{\Omega}}(\bm{\sigma},\bm{0},\bm{0})\Vert_F\cdot\Vert\bm{E}\Vert_F\\
&\le\Gamma_{\bm{X}}(\bm{\sigma})\|\bm{\sigma}\|\cdot\Vert\bm{E}\Vert_F(\| \bm{U}-\bm{U}e^{\bm{E}}\|_F+\| \bm{V}-\bm{V}e^{\bm{F}}\|_F).\end{array}$$
Hence, by Lemma \ref{lemmaexplog} (iv), one has that assertion (vi) holds.
Therefore, the proof is complete.
\end{proof}

Recall  that $l_{\sigma}^k$,  $l_{\Omega}^k$ and $c_k$ are defined by \eqref{Ak}, \eqref{Bk} and \eqref{Bk2} respectively.  Define \begin{equation}\label{alphakdef}\alpha_k:=\min\left\{\frac{1}{2}\left(\frac{1}{t_{k}}-l_{\sigma}^k\right),\;\frac{1}{s_k}-\frac{1}{2}c_ks_k- l_{\Omega}^k\right\}\quad{\rm for\;  each }\quad k\in\N.\end{equation}
The following lemma provides  a lower bound for $\{\alpha_k\}$ which will be used for establishing the sufficient descent property of  $\{\mathcal{F}(\bm{X}_{k})\}$ in Proposition \ref{prop4} below.

\begin{lemma}\label{alphabound}Assume that the stepsizes $\{t_k\}$ and $\{s_k\}$ are selected by strategy ${\rm (ii)}$. Then, it holds that
\begin{equation*}\label{alphalater}
\alpha_k\ge\min\left\{\frac{\mu}{2(1-\mu)}\|\mathcal{A}^*\mathcal{A}\|_{{\rm op}},\quad\frac{\mu^3(2-\mu)}{2(1-\mu)}\right\}\quad{\rm for\; each}\quad k\in\N.
\end{equation*}
\end{lemma}
\begin{proof}For convenience, fix $k\in\N$ and write
\begin{equation*}\label{alphade}
\alpha_1^k:=\frac12\left(\frac{1}{t_{k}}-l_{\sigma}^k\right),\quad\alpha_2^k:=\frac{1}{s_k}-\frac{1}{2}c_ks_k -l_{\Omega}^k.
\end{equation*}
Then, thanks to the definition of $\alpha_k$, it suffices to prove
\begin{equation}\label{ls}\alpha_1^k\ge\frac{\mu}{2(1-\mu)}\|\mathcal{A}^*\mathcal{A}\|_{{\rm op}}\quad{\rm and}\quad \alpha_2^k\ge\frac{\mu^3(2-\mu)}{2(1-\mu)}.\end{equation}
By assumption ($A_1$), one has
\begin{equation*}\label{lsineq}
\alpha_1^k=\frac12\left(\frac{1}{t_{k}}-l_{\sigma}^k\right)\ge  \frac{\mu}{2(1-\mu)}l_{\sigma}^k.
\end{equation*}
Then the first inequality in  \eqref{ls} is seen to hold as $l_{\sigma}^k\ge\|\mathcal{A}^*\mathcal{A}\|_{{\rm op}}$ by  \eqref{Ak}.
To prove the second inequality in \eqref{ls}, we need to consider the following two cases: (a) $l_{\Omega}^k\le\frac{\mu(2-\mu)}{2s_k}$;
(b) $l_{\Omega}^k>\frac{\mu(2-\mu)}{2s_k}$.

Case (a). Note by \eqref{xfbound} and  assumption ($A_2$) that
\begin{equation*}\label{ckskes}c_ks_k^2\le(1-\mu)^2 c_k\bar{s}_k^2\leq2(1-\mu)^2.\end{equation*}
Then, by the assumption  $l_{\Omega}^k\le\frac{\mu(2-\mu)}{2s_k}$, we deduce
\begin{equation*}l_{\Omega}^k+\frac{1}{2}c_ks_k-\frac{1}{s_k}=l_{\Omega}^k+\frac{c_ks_k^2-2}{2s_k}\le\frac{\mu(\mu-2)}{2s_k}\le\frac{\mu^2(\mu-2)}{2},\end{equation*}
where the last inequality holds as $s_k\le\frac{1}{\mu}$ by  assumption ($A_2$). Then, thanks to the definition of $\alpha_2^k$ and the fact that $\mu\in(0,\;\frac12]$, one sees that the second inequality in \eqref{ls} holds.

Case (b).  Since $s_k\le(1-\mu)\bar{s}_k$ and $\mu\in(0,\frac12]$, one has
\begin{equation}\label{ckskii}\frac{1}{s_k}-\frac12c_ks_k\ge\frac{1}{(1-\mu)\bar{s}_k}-\frac12(1-\mu)c_k\bar{s}_k\ge\frac{1}{1-\mu}\left(\frac{1}{\bar{s}_k}-\frac12c_k\bar{s}_k\right).\end{equation}
On the other hand, by the definition of $\bar{s}_k$, it is easy to check that $l_{\Omega}^k+\frac{1}{2}c_k\bar{s}_k-\frac{1}{\bar{s}_k}=0$ or equivalently
\begin{equation}\label{ckskif}\frac{1}{\bar{s}_k}-\frac{1}{2}c_k\bar{s}_k=l_{\Omega}^k.\end{equation} Then, from \eqref{ckskii}, \eqref{ckskif} and the definition of $\alpha_2^k$, we derive
\begin{equation}\label{ckski}\alpha_2^k=\frac{1}{s_k}-\frac{1}{2}c_ks_k -l_{\Omega}^k\ge\frac{1}{1-\mu}\left(\frac{1}{\bar{s}_k}-\frac12c_k\bar{s}_k\right) - l_{\Omega}^k=\frac{\mu}{1-\mu}l_{\Omega}^k.\end{equation}
Then, noting $l_{\Omega}^k>\frac{\mu(2-\mu)}{2s_k}$ and $s_k\le\frac{1}{\mu}$, we obtain by \eqref{ckski} that the second inequality  in \eqref{ls} holds. Therefore, the proof is complete.
\end{proof}

For the remainder of this paper, let $\{\bm{X}_{k}\}$, $\{\bm{\Omega}_k=(\bm{U}_{k},\bm{V}_k)\}$ and $\{(\bm{\sigma}_k,\bm{E}_{k},\bm{F}_k)\}$ be the sequences generated by Algorithm \ref{alg1} with the stepsizes $\{t_k\}$, $\{s_k\}$ selected by strategies (i) or (ii).
We now present  the following proposition. In particular,  assertion ${\rm (i)}$ illustrates that  the stepsizes $\{t_k\}$ and $\{s_k\}$ selected by strategy ${\rm (i)}$ are attainable within  finite steps;  while assertion ${\rm (ii)}$  shows the sufficient descent property of  $\{\mathcal{F}(\bm{X}_{k})\}$ if the stepsizes $\{t_k\}$ and $\{s_k\}$ are selected by strategy ${\rm (ii)}$.

\begin{prop}\label{prop4} We have the following two assertions:
\begin{itemize}
                           \item [\rm (i)]Given $\alpha>0$, if the stepsizes $\{t_k\}$ and $\{s_k\}$ are selected by Strategy ${\rm (i)}$, then   $$t_k\ge\frac{\rho_t}{\alpha +l_{\sigma}^k}
\quad \text{ and}\quad s_k\ge\frac{2\rho_s}{\alpha+l_{\Omega}^k+\sqrt{(\alpha+l_{\Omega}^k)^2+2c_k}},\quad\text{ for each}\quad k\in \N;$$

                           \item [\rm (ii)]If  the stepsizes $\{t_k\}$ and $\{s_k\}$ are selected by Strategy ${\rm (ii)}$, then
\eqref{H_1} is satisfied with $\alpha:=\inf\limits_{k\ge0}\alpha_k>0$.
                         \end{itemize}

\end{prop}
\begin{proof}
Fix $k\in \N$  and define the map $\mathcal{C}:\mathcal{S}(n)\rightarrow\mathbb{R}^{n\times n}$ by
\begin{align}\label{mapC}\mathcal{C}(\bm{X}):=\left(\bm{I}+\frac{1}{2}\bm{X}\right)\left(\bm{I}-\frac{1}{2}\bm{X}\right)^{-1}\quad{\rm for \;each}\quad\bm{X}\in\mathcal{S}(n).\end{align}
Then it is clear that  $\bm{U}_{k+1}=\bm{U}_{k}\C(\bm{E}_{k})$. Moreover, write
$\hat{\bm{X}}_{k+1}:=(\bm{U}_{k}e^{\bm{E}_{k}})^\top \mathcal{D}(\bm{\sigma}_{k+1})\bm{V}_{k}e^{\bm{F}_{k}}$.
We first give the estimation of $\mathcal{F}(\bm{X}_{k+1})-\mathcal{F}(\hat{\bm{X}}_{k+1})$.  Note by the convexity of $\Vert \mathcal{A}(\cdot)-\bm{b}\Vert^{2}$ that
\begin{align}\label{convex}
\Vert \mathcal{A}(\bm{X}_{k+1})-\bm{b}\Vert^{2}-\Vert \mathcal{A}(\hat{\bm{X}}_{k+1})-\bm{b}\Vert^{2}
&\leq\langle\mathcal{A}^*(\mathcal{A}(\bm{X}_{k+1})-\bm{b}),\; \bm{X}_{k+1}-\hat{\bm{X}}_{k+1}\rangle\notag\\
&\leq(\|\mathcal{A}^*\mathcal{A}\|_{\rm{op}}\cdot\|\bm{X}_{k+1}\|_F+\|\mathcal{A}^*(\bm{b})\|)\|\bm{X}_{k+1}-\hat{\bm{X}}_{k+1}\|_F.
\end{align}
Moreover, by the definitions of $\bm{X}_{k+1}$ and $\hat{\bm{X}}_{k+1}$, one has
\begin{align}\label{convex2}
\bm{X}_{k+1}-\hat{\bm{X}}_{k+1}=(\bm{U}_{k}\C(\bm{E}_{k}))^\top\mathcal{D}(\bm{\sigma}_{k+1})\bm{V}_{k}(\C(\bm{F}_{k})-e^{\bm{F}_{k}})
+(\C(\bm{E}_{k})-e^{\bm{E}_{k}})^\top\bm{U}_{k}^\top\mathcal{D}(\bm{\sigma}_{k+1})\bm{V}_{k}e^{\bm{F}_{k}}.
\end{align}
Since $\|\bm{E}_{k}\|_F\le1$ and $\|\bm{F}_{k}\|_F\le1$ (noting the selection strategies of $t_k$, $s_k$ and the definitions of $\bm{E}_{k}$,  $\bm{F}_{k}$ in Algorithm \ref{alg1}), Lemma \ref{fact3} is applicable to concluding that
\begin{align}\label{factapp}
\|\mathcal{C}(\bm{E}_{k})-e^{\bm{E}_{k}}\|_F\le\|\bm{E}_{k}\|_F^3\quad{\rm and} \quad\|\mathcal{C}(\bm{F}_{k})-e^{\bm{F}_{k}}\|_F\le\|\bm{F}_{k}\|_F^3.
\end{align}
Then, thanks to the orthogonality of $\bm{U}_{k}$, $\bm{V}_{k}$, $e^{\bm{F}_{k}}$ and $\C(\bm{E}_{k})$, we deduce by \eqref{convex2} and \eqref{factapp} that
\begin{align}\label{convexes}
\|\bm{X}_{k+1}-\hat{\bm{X}}_{k+1}\|_F\leq\|\bm{\sigma}_{k+1}\|(\|\bm{E}_{k}\|_F^3+\|\bm{F}_{k}\|_F^3).
\end{align}
On the other hand, using again the definition of $\bm{X}_{k+1}$ and the orthogonality of $\bm{U}_{k+1}$, $\bm{V}_{k+1}$, we have
\begin{align}\label{xes}
\|\bm{X}_{k+1}\|_F=\|\bm{U}_{k+1}^\top \mathcal{D}(\bm{\sigma}_{k+1})\bm{V}_{k+1}\|_F=\|\bm{\sigma}_{k+1}\|.
\end{align}
Substituting \eqref{convexes} and \eqref{xes} into \eqref{convex}, we obtain
\begin{align}\label{ax}&\Vert \mathcal{A}(\bm{X}_{k+1})-\bm{b}\Vert^{2}-\Vert \mathcal{A}(\hat{\bm{X}}_{k+1})-\bm{b}\Vert^{2}\notag\\
&\leq(\|\mathcal{A}^*\mathcal{A}\|_{\rm{op}}\cdot\|\bm{\sigma}_{k+1}\|^2+\|\mathcal{A}^*(\bm{b})\|\cdot\|\bm{\sigma}_{k+1}\|)(\|\bm{E}_{k}\|_F^3+\|\bm{F}_{k}\|_F^3)\notag\\
&\leq(\|\mathcal{A}^*\mathcal{A}\|_{\rm{op}}\cdot\|\bm{\sigma}_{k+1}\|^2+\|\mathcal{A}^*(\bm{b})\|\cdot\|\bm{\sigma}_{k+1}\|)(\|\bm{E}_{k}\|_F+\|\bm{F}_{k}\|_F)(\|\bm{E}_{k}\|_F^2+\|\bm{F}_{k}\|_F^2).
\end{align}
Hence, thanks to the definitions of $\bm{E}_{k}$,  $\bm{F}_{k}$ and $c_k$, we derive from \eqref{ax} that
\begin{align}\label{fx+1fh}
\mathcal{F}(\bm{X}_{k+1})-\mathcal{F}(\hat{\bm{X}}_{k+1})
&=\frac12(\Vert \mathcal{A}(\bm{X}_{k+1})-\bm{b}\Vert^{2}-\Vert \mathcal{A}(\hat{\bm{X}}_{k+1})-\bm{b}\Vert^{2})\leq\frac{1}{2}c_k s_k(\|\bm{E}_{k}\|_F^2+\|\bm{F}_{k}\|_F^2).
\end{align}
Now we will estimate $\mathcal{F}(\hat{\bm{X}}_{k+1})-\mathcal{F}(\bm{X}_k)$. For this purpose, rewrite
\begin{align}\label{Fx+1}
\mathcal{F}(\hat{\bm{X}}_{k+1})&=f_{\bm{\Omega}_k}(\bm{\sigma}_{k+1},\bm{E}_{k},\bm{F}_{k})+\lambda\Vert \bm{\sigma}_{k+1}\Vert_{p}^p=R_L+R_{\bm{\sigma}}+R_{\bm{(E,F)}}+R_M,
\end{align}
where
\begin{align*}
R_L:=\ &\langle \nabla_{\bm{\sigma}} f_{\bm{\Omega}_k}(\bm{\sigma}_{k},\bm{E}_{k},\bm{F}_{k})-\nabla_{\bm{\sigma}} f_{\bm{\Omega}_k}(\bm{\sigma}_{k},\bm{0},\bm{0}),\bm{\sigma}_{k+1}-\bm{\sigma}_{k}\rangle,\\
R_{\bm{\sigma}}:=\ &f_{\bm{\Omega}_k}(\bm{\sigma}_{k+1},\bm{E}_{k},\bm{F}_{k})
-f_{\bm{\Omega}_k}(\bm{\sigma}_{k},\bm{E}_{k},\bm{F}_{k})-\langle \nabla_{\bm{\sigma}} f_{\bm{\Omega}_k}(\bm{\sigma}_{k},\bm{E}_{k},\bm{F}_{k}),\bm{\sigma}_{k+1}-\bm{\sigma}_{k}\rangle,\\
R_{\bm{(E,F)}}:=\ &f_{\bm{\Omega}_k}(\bm{\sigma}_{k},\bm{E}_{k},\bm{F}_{k})-f_{\bm{\Omega}_k}(\bm{\sigma}_{k},\bm{0},\bm{0})-\langle\nabla_{\bm{E}} f_{\bm{\Omega}_k}(\bm{\sigma}_k,\bm{0},\bm{0}),\bm{E}_{k}\rangle-\langle\nabla_{\bm{F}} f_{\bm{\Omega}_k}(\bm{\sigma}_k,\bm{0},\bm{0}),\bm{F}_{k}\rangle,\\
R_M:=\ &f_{\bm{\Omega}_k}(\bm{\sigma}_{k},\bm{0},\bm{0})+\lambda\Vert \bm{\sigma}_{k+1}\Vert_{p}^p+\langle \nabla_{\bm{\sigma}} f_{\bm{\Omega}_k}(\bm{\sigma}_{k},\bm{0},\bm{0}),\bm{\sigma}_{k+1}-\bm{\sigma}_{k}\rangle+\langle\nabla_{\bm{E}} f_{\bm{\Omega}_k}(\bm{\sigma}_k,\bm{0},\bm{0}),\bm{E}_{k}\rangle\\
&+\langle\nabla_{\bm{F}} f_{\bm{\Omega}_k}(\bm{\sigma}_k,\bm{0},\bm{0}),\bm{F}_{k}\rangle.
\end{align*}
Below we will  estimate $R_L$, $R_{\bm{\sigma}}$, $R_{\bm{(E,F)}}$ and $R_M$. We first note by Lemma \ref{prop2} (iv) and the definition of $R_L$ that
\begin{align}\label{RL}
R_L
&\leq (\| \nabla f(\bm{X}_k)\Vert_F+\Vert\mathcal{A}^*\mathcal{A}\Vert_{{\rm op}}\|\bm{\sigma}_k\|) (\Vert \bm{E}_{k}\Vert_{F} +\Vert \bm{F}_{k}\Vert_{F})\Vert\bm{\sigma}_{k+1}-\bm{\sigma}_{k}\Vert\notag\\
&\leq \frac12(\| \nabla f(\bm{X}_k)\Vert_F+\Vert\mathcal{A}^*\mathcal{A}\Vert_{{\rm op}}\|\bm{\sigma}_k\|)( \Vert \bm{E}_{k}\Vert_{F} ^2+\Vert \bm{F}_{k}\Vert_{F}^2+2\Vert\bm{\sigma}_{k+1}-\bm{\sigma}_{k}\Vert^2).
\end{align}
For the estimation of $R_{\bm{\sigma}}$ and $R_{\bm{(E,F)}}$, let $t\in[0,1]$. Then, we have by
Lemma \ref{prop2} (v) that
\begin{align*}
\big<\nabla_{\bm{\sigma}}f_{\bm{\Omega}_{k}}(\bm{\sigma}_{k}+t(\bm{\sigma}_{k+1}-\bm{\sigma}_{k}),\bm{E}_{k},\bm{F}_{k})-\nabla_{\bm{\sigma}}f_{\bm{\Omega}_{k}}(\bm{\sigma}_k,\bm{E}_{k},\bm{F}_{k}),\;\bm{\sigma}_{k+1}-\bm{\sigma}_{k}\big>
\leq t\Vert\mathcal{A}^*\mathcal{A}\Vert_{{\rm op}}\Vert \bm{\sigma}_{k+1}-\bm{\sigma}_k\Vert^2.
\end{align*}
Thus, thanks to Lemma \ref{lemma1} and the definition of $R_{\bm{\sigma}}$, one sees
\begin{align}\label{RS}R_{\bm{\sigma}}\leq\frac12\Vert\mathcal{A}^*\mathcal{A}\Vert_{{\rm op}}\Vert \bm{\sigma}_{k+1}-\bm{\sigma}_k\Vert^2.
\end{align}
Using Lemma \ref{prop2} (vi) and (vii), we deduce
for each $t\in[0,1]$,
\begin{align*}
&\langle\nabla_{\bm{E}}f_{\bm{\Omega}_k}(\bm{\sigma}_k,t\bm{E}_{k},t\bm{F}_{k})-\nabla_{\bm{E}} f_{\bm{\Omega}_k}(\bm{\sigma}_k,\bm{0},\bm{0}),\bm{E}_{k}\rangle\notag\\
\leq\ & t(\|\nabla f(\bm{X}_k)\|_F\|\bm{\sigma}_k\|+\|\mathcal{A}^*\mathcal{A}\|_{\rm{op}}\|\bm{\sigma}_k\|^2)
\left(\Vert \bm{E}_k\Vert_{F}^2+\Vert \bm{E}_k\Vert_{F}\Vert \bm{F}_k\Vert_{F}\right)
\end{align*}
and
\begin{align*}
&\langle\nabla_{\bm{F}}f_{\bm{\Omega}_k}(\bm{\sigma}_k,t\bm{E}_{k},t\bm{F}_{k})-\nabla_{\bm{F}}f_{\bm{\Omega}_k}(\bm{\sigma}_k,\bm{0},\bm{0}), \bm{F}_{k}\rangle\notag\\
\leq\ & t(\|\mathcal{A}^*\mathcal{A}\|_{\rm{op}}\cdot\|\bm{\sigma}_k\|^2+\|\nabla f(\bm{X}_k)\|_F\cdot\|\bm{\sigma}_k\|)
\left(\Vert \bm{F}_k\Vert_{F}^2+\Vert \bm{E}_k\Vert_{F}\Vert \bm{F}_k\Vert_{F}\right).
\end{align*}
Then, applying Lemma \ref{lemma1}  and using the definition of $R_{\bm{(E,F)}}$, one has
\begin{align}\label{REF}
R_{\bm{(E,F)}}\leq(\|\mathcal{A}^*\mathcal{A}\|_{\rm{op}}\cdot\|\bm{\sigma}_k\|^2+\|\nabla f(\bm{X}_k)\|_F\cdot\|\bm{\sigma}_k\|)
 (\Vert \bm{E}_{k}\Vert_F^2+\Vert \bm{F}_{k}\Vert_F^2).
\end{align}
As for the estimation of $R_M$, we note by \eqref{eq:sigma} that
$$\frac{1}{2t_{k}} \Vert \bm{\sigma}_{k+1}-\bm{\sigma}_{k}+t_{k}\nabla_{\bm{\sigma}} f_{\bm{\Omega}_k}(\bm{\sigma}_k,\bm{0},\bm{0})\Vert^{2} +\lambda\Vert \bm{\sigma}_{k+1}\Vert_{p}^p\le\frac{1}{2t_{k}} \Vert t_{k}\nabla_{\bm{\sigma}} f_{\bm{\Omega}_k}(\bm{\sigma}_k,\bm{0},\bm{0})\Vert^{2} +\lambda\Vert \bm{\sigma}_{k}\Vert_{p}^p,$$
which gives
\begin{align}\label{eqk0}
\langle \nabla _{\bm{\sigma}} f_{\bm{\Omega}_k}(\bm{\sigma}_{k},\bm{0},\bm{0}),\bm{\sigma}_{k+1}-\bm{\sigma}_{k}\rangle+\lambda\Vert \bm{\sigma}_{k+1}\Vert_{p}^p\le\lambda\Vert \bm{\sigma}_{k}\Vert_{p}^p
-\frac{1}{2t_{k}}\Vert \bm{\sigma}_{k+1}-\bm{\sigma}_k\Vert^{2}.
\end{align}
On the other hand, we recall from \eqref{ExpE} that
\begin{equation}\label{ekdef}\langle\nabla_{\bm{E}} f_{\bm{\Omega}_k}(\bm{\sigma}_k,\bm{0},\bm{0}),\bm{E}_{k}\rangle=-\frac{1}{s_{k}}\Vert\bm{E}_{k}\Vert_F^2\quad{\rm and}\quad\langle\nabla_{\bm{F}} f_{\bm{\Omega}_k}(\bm{\sigma}_k,\bm{0},\bm{0}),\bm{F}_{k}\rangle=-\frac{1}{s_{k}}\Vert\bm{F}_{k}\Vert_F^2.\end{equation}
Then, in view of the definition of $R_M$, we obtain  by \eqref{eqk0} and \eqref{ekdef} that
\begin{equation}\label{RM}R_M\leq \mathcal{F}(\bm{X}_k)-\frac{1}{2t_{k}}\Vert \bm{\sigma}_{k+1}-\bm{\sigma}_k\Vert_2^{2}-\frac{1}{s_k}(\Vert \bm{E}_{k}\Vert_{F}^{2}+\Vert \bm{F}_{k}\Vert_{F}^{2})\end{equation}
(noting that $f_{\bm{\Omega}_k}(\bm{\sigma}_{k},\bm{0},\bm{0})+\lambda\Vert \bm{\sigma}_{k}\Vert_{p}^p=\mathcal{F}(\bm{X}_k)$).
Substituting \eqref{RS}, \eqref{REF}, \eqref{RL} and \eqref{RM} into \eqref{Fx+1}, we can conclude
\begin{align*}
\mathcal{F}(\hat{\bm{X}}_{k+1})&\leq \mathcal{F}(\bm{X}_k)+\left(\frac{1}{2}\|\mathcal{A}^*\mathcal{A}\|_{\rm{op}}+\Vert\nabla f(\bm{X}_k) \Vert_{F}+\|\mathcal{A}^*\mathcal{A}\|_{\rm{op}}\Vert\bm{\sigma}_k\Vert-\frac{1}{2t_k}\right)\|\bm{\sigma}_{k+1}-\bm{\sigma}_{k}\|^2\notag\\
&\quad +\left(\|\mathcal{A}^*\mathcal{A}\|_{\rm{op}}\cdot\|\bm{\sigma}_k\|^2+(\frac12+\|\bm{\sigma}_k\|)\|\nabla f(\bm{X}_k)\|_F+\frac12\|\bm{\sigma}_k\|-\frac{1}{s_k}\right) (\Vert \bm{E}_{k}\Vert_F^2+\Vert \bm{F}_{k}\Vert_F^2).
\end{align*}
Thus, by the definitions of  $l_{\sigma}^k$ and $l_{\Omega}^k$, we have
\begin{align}\label{xhat}
&\mathcal{F}(\hat{\bm{X}}_{k+1})-\mathcal{F}(\bm{X}_k)
\leq \frac12\left(l_{\sigma}^k-\frac{1}{t_{k}}\right)\Vert \bm{\sigma}_{k+1}-\bm{\sigma}_k\Vert^{2}+\left(l_{\Omega}^k-\frac{1}{s_k}\right)(\Vert \bm{E}_{k}\Vert_{F}^{2}+\Vert \bm{F}_{k}\Vert_{F}^{2}).
\end{align}
Adding \eqref{xhat}  to \eqref{fx+1fh}, one checks
\begin{align}\label{fxk}
\mathcal{F}(\bm{X}_{k+1})-\mathcal{F}(\bm{X}_{k})
&\le\frac12\left(l_{\sigma}^k-\frac{1}{t_{k}}\right) \Vert\bm{\sigma}_{k+1}-\bm{\sigma}_k\Vert^{2}+\left(l_{\Omega}^k+\frac{1}{2}c_ks_k-\frac{1}{s_k}\right)(\Vert \bm{E}_{k}\Vert_{F}^{2}+\Vert \bm{F}_{k}\Vert_{F}^{2}).
\end{align}
Therefore, for a given $\alpha$, $\frac12\left(l_{\sigma}^k-\frac{1}{t_{k}}\right)\le-\alpha$ and $\left(l_{\Omega}^k+\frac{1}{2}c_ks_k-\frac{1}{s_k}\right)\le-\alpha$
are satisfied respectively for $t_{k}\le\frac{1}{2\alpha+l_{\sigma}^k}$ and $s_k\le\frac{2}{\alpha+l_{\Omega}^k+\sqrt{(\alpha+l_{\Omega}^k)^2+2c_k}}$.
Consequently, assertion (i) is seen to hold.
On the other hand, using \eqref{alphakdef} and \eqref{fxk}, one has
\begin{align*}\label{fxk2}
\mathcal{F}(\bm{X}_{k+1})-\mathcal{F}(\bm{X}_{k})\le-\alpha_k(\Vert\bm{\sigma}_{k+1}-\bm{\sigma}_k\Vert^{2}+\Vert \bm{E}_{k}\Vert_{F}^{2}+\Vert \bm{F}_{k}\Vert_{F}^{2}).
\end{align*}
Thus, assertion (ii) is also seen to hold by Lemma \ref{alphabound} and the fact that $\alpha=\inf\limits_{k\ge0}\alpha_k$. The proof is complete.
\end{proof}

\begin{remark}\label{rmk:cluster}
The following assertions follow directly from Proposition \ref{prop4}.
\begin{enumerate}[(i)]
  \item The sequence $\{(\bm{\sigma}_k,\bm{U}_k,\bm{V}_k)\}$ (resp. $\{\bm{X}_k\}$) is bounded; hence it has a cluster point which, in what follows, will be denoted by $(\bm{\bar{\sigma}},\bm{\bar{U}},\bm{\bar{V}})$ (resp. $\bm{\bar{X}}$).
  \item $\{\mathcal{F}(\bm{X}_{k})\}$ is monotonically decreasing and
  $ \lim\limits_{k\rightarrow\infty}\mathcal{F}(\bm{X}_{k})=\mathcal{F}(\bm{\bar{X}})$.
  \item $\sum\limits_{k=0}^{\infty}\Vert (\bm{\sigma}_{k+1}-\bm{\sigma}_{k},\bm{E}_{k},\bm{F}_{k})\Vert^{2}<+\infty$ and so $\lim\limits_{k\rightarrow\infty} (\bm{\sigma}_{k+1}-\bm{\sigma}_{k},\bm{E}_{k},\bm{F}_{k})=(\bm{0},\bm{0},\bm{0})$.
\end{enumerate}
\end{remark}

\begin{remark}\label{sktkbound}Let $ \{s_k\}$ and $ \{t_k\}$ be selected by Strategy (i) or (ii). We assert that there exist positive constants $\underline{s}$ and $\underline{t}$ such that
\begin{equation}\label{tklowbound}
s_k\ge\underline{s}\quad {\rm and}\quad t_k\ge\underline{t}\quad {\rm for\; each}\; k\in\N.
\end{equation}
In fact, we note  by Remark \ref{rmk:cluster} (i) and  \eqref{fXnabla} that
 $\{\bm{\sigma}_{k}\}$, $\{\bm{X}_k\}$ and $\{\nabla f(\bm{X}_k)\}$ are bounded. Moreover, by Proposition \ref{prop1} (iv), (v) and the definitions of $\{\nabla_{\bm{E}}^k\}$, $\{\nabla_{\bm{F}}^k\}$ in \eqref{nablaekfkdefi},
 one sees that $\{\|\nabla_{\bm{E}}^k\|_F\}$ and $\{\|\nabla_{\bm{F}}^k\|_F\}$ are bounded. Then, by definitions, one checks that $\{c_k\}$,  $\{l_{\sigma}^k\}$ and $\{l_{\Omega}^k\}$ are  bounded.  Thus, thanks to Proposition \ref{prop4} (i), the assertion \eqref{tklowbound}  holds when $ \{t_k\}$ and $ \{s_k\}$ are selected by Strategy (i).
On the other  hand, by \eqref{xfbound}, one has that $\{\bar{s}_k\}$  has a lower  bound and so, the assertion \eqref{tklowbound} is also seen to  hold if $ \{t_k\}$ and $ \{s_k\}$ are selected by Strategy (ii).
\end{remark}

\section{Convergence Analysis}\label{sec:convergence}
\subsection{K{\L} Property}
We begin this subsection with the notions of the Fr${\rm\acute{e}}$chet and limiting subdifferential. For this purpose, let  $h: \mathbb{R}^l\rightarrow \overline{\mathbb{R}}:=\mathbb{R}\cup\{+\infty\}$ be a proper  lower semi-continuous  function, the domain of which is defined by
${\rm dom}\ h:=\{x\in \mathbb{R}^l:h(x)<+\infty\}$. The Fr${\rm\acute{e}}$chet and limiting subdifferential of  $h$ at $\bm{x}\in \mathbb{R}^l$ (cf. \cite[Definition 8.3]{rockafellar2009variational}) are defined respectively by
 $$\hat{\partial} h(\bm{x}):=
 \left\{
      \begin {array}{lr}
      \left\{\bm{v}\in\mathbb{R}^{l}:\liminf\limits_{\bm{y}\neq\bm{x}, \bm{y}\rightarrow\bm{x}}
\frac{h(\bm{y})-h(\bm{x})-\langle \bm{v}, \bm{y}-\bm{x}\rangle}{\Vert \bm{y}-\bm{x}\Vert}\ge0\right\}, & x\in {\rm dom}\ h ,\\\\
      \emptyset,& x\notin{\rm dom}\ h,
      \end {array}
 \right.
 $$
and
$$\partial h(\bm{x}):=\{\bm{v}\in\mathbb{R}^l: \exists \bm{x}_k\rightarrow \bm{x}, h(\bm{x}_k)\rightarrow h(\bm{x}), \bm{v}_k\in\hat{\partial} h(\bm{x}_k)\rightarrow \bm{v} \}.$$
In particular, if $h$ is an analytic function, then $\hat{\partial} h(\cdot)=\partial h(\cdot)=\{\nabla h(\cdot)\}$.

In our convergence analysis, the K{\L} property (cf. \cite{Attouch2010, Attouch2013Convergence, li2018calculus}) defined in Definition \ref{defexp} plays a key role. The original and pioneering work of {\L}ojasiewicz \cite{Lojasiewicz1963} and Kurdyka \cite{Kurdyka1998}  on differentiable functions
laid the foundation of the K{\L}  property, which was later extended to nonsmooth functions by Bolte
et al. \cite{Bolte2007The, Bolte2007}.

\begin{definition}\label{defexp} Let $\bar{\bm{x}}\in\mathbb{R}^l$ and $\theta\in[0,1)$.   We say that the function $h$
 has the  K{\L} property at $\bar{\bm{x}}$ with the exponent $\theta$ if
 there exist a neighborhood $\mathcal{N}$ of $\bar{\bm{x}}$, $\kappa>0$ and $\epsilon\in(0, +\infty]$ such that, for any $\bm{x}\in \mathcal{N}$ with
$h(\bar{\bm{x}})<h(\bm{x})<h(\bar{\bm{x}})+\epsilon$, $$
\kappa\operatorname{d}(0, \partial h(\bm{x})) \geq (h(\bm{x})-h(\bar{\bm{x}}))^\theta.
$$
\end{definition}

To provide  a sufficient condition for proving the K{\L} property,  we also recall the notation of  a semianalytic function; see for instance \cite[Definition 2.1, Definition 2.3]{Bierstone} and \cite[Definition 2.1]{Bolte2007}.

\begin{definition}\label{defsubana}
(a) A subset $\mathcal{S}$ of $\mathbb{R}^l$ is said to be semi-analytic if for each  $\bm{z}\in\mathbb{R}^l$ there exists a neighborhood $\mathcal{N}$ of $\bm{z}$ and
real-analytic functions $g_{ij}$ and $h_{ij}$ ($1\le i \le p, 1\le j \le q$) such that
$$
\mathcal{S}\cap\mathcal{N}=\bigcup_{i=1}^p \bigcap_{j=1}^q\left\{\bm{x} \in \mathcal{N}: g_{i j}(\bm{x})=0 \text { and } h_{i j}(\bm{x})>0\right\}.
$$
(b) A function $h:\mathbb{R}^l \rightarrow \bar{\mathbb{R}}$ is said to be semi-analytic if its graph
$$
{\rm Graph}\ h:=\{(\bm{x},t)\in\mathbb{R}^l \times \mathbb{R}:h(\bm{x})=t\}
$$
is  semi-analytic.

\end{definition}

The following proposition is a direct consequence of \cite[Theorem 3.1, Remark 3.2]{Bolte2007The} (noting that a semi-analytic function is subanalytic (cf. \cite[Definition 2.1]{Bolte2007}) as remarked by \cite[Section 6.6]{Facchinei2007Finite}).

\begin{prop}\label{semilemma}Assume that $h:\mathbb{R}^l \rightarrow \bar{\mathbb{R}}$ is a semi-analytic function with a closed domain, and assume that $h|_{{\rm dom} h}$ is continuous. Then, for each $\bar{\bm{x}}\in\mathbb{R}^l$, there exists $\theta\in[0,1)$ such that $h$ has the K{\L} property at $\bar{\bm{x}}$ with the exponent $\theta$.
\end{prop}

To prove the K{\L} property of $\mathcal{F}_{\bm{\bm{\Omega}}}$, we need to
present below the closed form of $\partial \mathcal{F}_{\bm{\bm{\Omega}}}$ which will also be used in the later convergence analysis. For this purpose, we write 
\begin{equation}\label{gammadef}\gamma_p:=\left\{
                                                                                             \begin{array}{ll}
                                                                                               1, & \hbox{$p=1$;} \\
                                                                                               +\infty, & \hbox{$p\in[0,1)$,}
                                                                                             \end{array}
                                                                                           \right.\end{equation}
and use
 ${\rm supp}(\bm{\sigma})$ to denote the support index set of $\bm{\sigma}\in\mathbb{R}^{n}$, i.e.,
${\rm supp}(\bm{\sigma}):= \{1\le i\le n |\;[\bm{\sigma}]_i\neq 0\}$.

\begin{prop}\label{lemma-par}We have the following assertions:
\begin{itemize}
 \item [{\rm (i)}]For any  $\bm{\sigma}\in\mathbb{R}^{n}$, it holds that $$\partial \Vert \bm{\sigma} \Vert_{0}^0=\left\{\bm{v}:[\bm{v}]_i= 0\; {\rm if}\; i\in{\rm supp}(\bm{\sigma})\right\}$$ and that, for $0<p\le 1$,
$$\partial \Vert \bm{\sigma} \Vert_{p}^p
= \left\{p\bm{v}:[\bm{v}]_i= {\rm sign}([\bm{\sigma}]_i)|[\bm{\sigma}]_i|^{p-1}\; {\rm if}\; i\in{\rm supp}(\bm{\sigma}),\;  |[\bm{v}]_i|\le\gamma_p\; {\rm if}\;  i\notin{\rm supp}(\bm{\sigma})\right\}.
$$

   \item [{\rm (ii)}] If $\bm{X}\in\mathbb{R}^{m\times n}$ admits the singular value decomposition $\bm{X}=\bm{U}^\top\mathcal{D}(\bm{\sigma})\bm{V}$, then
   \begin{equation*}
\partial \Vert \bm{X} \Vert_{S_p}^p
=\left\{\bm{U}^\top \left(
                      \begin{array}{cc}
                        \mathcal{D}(\partial\Vert \bm{\sigma}_{+}(\bm{X})\Vert_{p}^p) & \bm{0} \\
                        \bm{0}  & \bm W^{(22)} \\
                      \end{array}
                    \right)
\bm{V}:\bm W^{(22)}\in\mathbb{R}^{(m-r)\times(n-r)},\;\|\bm W^{(22)}\|_2\le\gamma_p\right\},
\end{equation*}
where $\partial\Vert \bm{\sigma}_{+}(\bm{X})\Vert_p^p $ is the limiting subdifferential of $\Vert\cdot\Vert_p^p$ at $\bm{\sigma}_{+}(\bm{X})$ which is the vector of the positive singular values of $\bm{X}$.
  \item [{\rm (iii)}]$\partial \mathcal{F}_{\bm{\bm{\Omega}}}(\cdot,\;\cdot,\;\cdot)=\nabla f_{\bm{\bm{\Omega}}}(\cdot,\;\cdot,\;\cdot)+\lambda\partial\Vert \cdot\Vert_{p}^p$ \; for\; each \; $\bm{\Omega}:=(\bm{U},\bm{V})\in\mathcal{O}(m)\times\mathcal{O}(n)$;
  \item [{\rm (iv)}]$\partial \mathcal{F}(\cdot)=\nabla f(\cdot)+\lambda\partial \Vert \cdot \Vert_{S_p}^p$.

\end{itemize}

\end{prop}
\begin{proof}By definitions, one checks that assertion (i) holds  (see e.g., \cite[Theorem 1]{Le2013} for $p=0$,  \cite[P24]{meng} for $p=1$ and  \cite[P328]{Ma2011} for $0<p<1$). For the proof of assertion (ii), we first note by the definitions  that  $\hat{\partial}  \Vert\cdot \Vert_{S_p}^p=\partial \Vert\cdot\Vert_{S_p}^p$. Then, in view of  \cite[Lemma 1]{giampouras2020novel}, one sees that assertion (ii) holds.
Finally, since $f_{\bm{\bm{\Omega}}}$ and $f$ are analytic functions, assertions (iii) and (iv) follow from the calculus rules of subgradients \cite[Exercise 8.8]{rockafellar2009variational}.
Thus, the proof is completed.
\end{proof}

For the remainder of this subsection, let $\hat{\bm{\Omega}}:=(\bm{\hat{U}},\bm{\hat{V}})\in\mathcal{O}(m)\times\mathcal{O}(n)$. Now we can show the following proposition that illustrates the K{\L} property of $\mathcal{F}_{\hat{\bm{\Omega}}}$.

\begin{prop}\label{lemma2} Let $\hat{\bm{\sigma}}\in\mathbb{R}^{n}$.
Then $\mathcal{F}_{\hat{\bm{\Omega}}}$ has the   K{\L} property at $(\hat{\bm{\sigma}},{\bm 0},{\bm 0})$ with some  exponent $\theta\in [0,1)$, that is,
there exist $\kappa>0$ and $\delta_1>0$ such that for all $(\bm{\sigma},\bm{E},\bm{F})\in{\bf{B}}((\hat{\bm{\sigma}},{\bm 0},{\bm 0}),\delta_1)\subseteq\mathcal{M}$ with $\mathcal{F}_{\hat{\bm{\Omega}}}(\hat{\bm{\sigma}},{\bm 0},{\bm 0})<\mathcal{F}_{\hat{\bm{\Omega}}}(\bm{\sigma}, \bm{E}, \bm{F})<\mathcal{F}_{\hat{\bm{\Omega}}}(\hat{\bm{\sigma}},{\bm 0},{\bm 0})+\epsilon$, the following inequality holds:
\begin{equation}\label{KLineq}
\kappa{\rm d}(0,\partial \mathcal{F}_{\hat{\bm{\Omega}}}(\bm{\sigma},\bm{E},\bm{F}))\geq ( \mathcal{F}_{\hat{\bm{\Omega}}}(\bm{\sigma},\bm{E},\bm{F})-\mathcal{F}_{\hat{\bm{\Omega}}}(\hat{\bm{\sigma}},{\bm 0},{\bm 0}))^{\theta}.
\end{equation}
\end{prop}
\begin{proof}Note that in the case when $p=1$,  $\mathcal{F}_{\hat{\bm{\Omega}}}$ is semi-analytic because
\begin{align}
{\rm Graph}\ \mathcal{F}_{\hat{\bm{\Omega}}}=\bigcup_{(s_i)\in\{-1,0,1\}^n}\{(\bm{\sigma},\bm{E},\bm{F},t)\in\mathcal{M}\times\mathbb{R}:\;f_{\hat{\bm{\Omega}}}(\bm{\sigma},\bm{E},\bm{F})+\lambda\sum_{i=1}^ns_i[\bm{\sigma}]_i=t,\;s_i[\bm{\sigma}]_i\ge0\},\notag\end{align}
and so the assertion holds by Definition \ref{defsubana} and Proposition \ref{semilemma}. Similarly, by definitions, we can also show that the assertion holds for the case when $p=0$ as, in this case,
$$\begin{array}{lll}&{\rm Graph}\ \mathcal{F}_{\hat{\bm{\Omega}}}
=\bigcup\limits_{k=1}^n\bigcup\limits_{1\le i_1<\cdots<i_k\le n}\{(\bm{\sigma},\bm{E},\bm{F},t)\in\mathcal{M}\times\mathbb{R}:\;f_{\hat{\bm{\Omega}}}(\bm{\sigma},\bm{E},\bm{F})+\lambda k=t,\\
&\qquad\qquad\qquad[\bm{\sigma}]_{i_1}^2>0, \ldots,[\bm{\sigma}]_{i_k}^2>0,\; [\bm{\sigma}]_{j}=0, j\notin\{i_1,\ldots,i_k\}\}\end{array}$$

Below we only need to consider the case when $p\in(0,\;1)$. For this purpose, we assume without loss of generality that $[\hat{\bm{\sigma}}]_i\neq0$ for each $1\le i\le r.$ Write  $\mathfrak{A}:=(R_{\neq}^{r}\times{\bm 0})\times\mathcal{S}(m)\times \mathcal{S}(n)$ where $R_{\neq}^{r}\times{\bm 0}:=\{(\bm{x}_{1},\bm{x}_{2},\ldots,\bm{x}_{r},0,\ldots,0)^\top \in \mathbb{R}^{n}\mid \bm{x}_{i} \neq0,\; 1\le i\le r\}$. Let $\tilde{\mathcal{F}}_{\hat{\bm{\Omega}}}$
be the restriction of $\mathcal{F}_{\hat{\bm{\Omega}}}$ on $\mathfrak{A}$.
Then  $\tilde{\mathcal{F}}_{\hat{\bm{\Omega}}}$ is analytic in a neighborhood of $(\hat{\bm{\sigma}},{\bm 0},{\bm 0})$ and so is semi-analytic.
 Therefore, thanks to Proposition \ref{semilemma},  there exist $\kappa>0$, $\theta\in[0,1)$, $\delta>0$ and $\epsilon\in(0, +\infty]$ such that \eqref{KLineq} holds for all $(\bm{\sigma},\bm{E},\bm{F})\in{\bf{B}}((\hat{\bm{\sigma}},{\bm 0},{\bm 0}),\delta)\subseteq\mathfrak{A}$ with $\mathcal{F}_{\hat{\bm{\Omega}}}(\hat{\bm{\sigma}},\bm{0},\bm{0}))<\mathcal{F}_{\hat{\bm{\Omega}}}(\bm{\sigma},\bm{E},\bm{F}))<\mathcal{F}_{\hat{\bm{\Omega}}}(\hat{\bm{\sigma}},\bm{0},\bm{0})) +\epsilon$.

 Below we will prove that there exist $\kappa>0$, $\theta\in[0,1)$ and $\delta>0$  such that \eqref{KLineq}
holds for all $(\bm{\sigma},\bm{E},\bm{F})\in{\bf{B}}((\hat{\bm{\sigma}},{\bm 0},{\bm 0}),\delta)\subseteq \mathcal{M}$ with $[\bm{\sigma}]_j\neq0$ for some $ r+1\le j\le n$.  In fact, by the  continuity of $\mathcal{F}_{\hat{\bm{\Omega}}}$, there exist $\delta_1>0$ and $\kappa>0$ such that for all $(\bm{\sigma},\bm{E},\bm{F})\in{\bf{B}}((\hat{\bm{\sigma}},{\bm 0},{\bm 0}),\delta_1) \subseteq \mathcal{M}$,
\begin{equation}\label{delta1}
\vert \mathcal{F}_{\hat{\bm{\Omega}}}(\bm{\sigma},\bm{E},\bm{F})-\mathcal{F}_{\hat{\bm{\Omega}}}(\hat{\bm{\sigma}},{\bm 0},{\bm 0})\vert^{\theta}\leq \kappa (\Vert\mathcal{A}^* \mathcal{A}\Vert_{{\rm op}}(1+\Vert\hat{\bm{\sigma}}\Vert)+\Vert\mathcal{A}^*(\bm{b})\Vert_F).
 \end{equation}
Set
$$\delta:=\min\left\{1,\;\delta_1,\;\left(\frac{2(\Vert\mathcal{A}^* \mathcal{A}\Vert_{{\rm op}}(1+\Vert\hat{\bm{\sigma}}\Vert)+\Vert\mathcal{A}^*(\bm{b})\Vert_F)}{\lambda p}\right)^{\frac{1}{p-1}}\right\}$$
 and let  $(\bm{\sigma},\bm{E},\bm{F})\in{\bf{B}}((\hat{\bm{\sigma}},0,0),\delta)\subseteq \mathcal{M}$ with  $[\bm{\sigma}]_j\neq0$ for some $r+1\le j\le n$.  Then, \eqref{delta1} follows subsequently as  $\delta\leq\delta_1$. On the other hand, thanks to Proposition \ref{lemma-par}, we deduce
\begin{align}\label{surj-key2}{\rm d}(0,\partial \mathcal{F}_{\hat{\bm{\Omega}}}(\bm{\sigma},\bm{E},\bm{F}))
&\ge|[\nabla_{\bm{\sigma}} f_{\hat{\bm{\Omega}}}(\bm{\sigma},\bm{E},\bm{F})]_j
+\lambda p{\rm sign}([\bm{\sigma}]_j)[\bm{\sigma}]_j^{p-1}|\ge\lambda p|[\bm{\sigma}]_j|^{p-1}-\Vert\nabla_{\bm{\sigma}} f_{\hat{\bm{\Omega}}}(\bm{\sigma},\bm{E},\bm{F})\Vert_F.\end{align}
Since $ r+1\le j\le n$ (i.e., $[\hat{\bm{\sigma}}]_j=0$), one has by the definition of $\delta$ that
\begin{align}\label{surj-key4}
\vert[\bm{\sigma}]_j\vert\le\Vert\bm{\sigma}-\hat{\bm{\sigma}}\Vert\le\delta\le\left(\frac{2(\Vert\mathcal{A}^* \mathcal{A}\Vert_{{\rm op}}(1+\Vert\hat{\bm{\sigma}}\Vert)+\Vert\mathcal{A}^*(\bm{b})\Vert_F)}{\lambda p}\right)^{\frac{1}{p-1}}.
\end{align}
Moreover, by Proposition \ref{prop1} (i), we have
\begin{align}\label{surj-key3}\Vert\nabla_{\bm{\sigma}}f_{\hat{\bm{\Omega}}}(\bm{\sigma},\bm{E},\bm{F})\Vert
&\leq\big \Vert\left(\hat{\bm{U}}e^{\bm{E}} \left(\mathcal{A}^*\mathcal{A}\left((\hat{\bm{U}}e^{\bm E})^\top \mathcal{D}(\bm{\sigma})\hat{\bm{V}}e^{\bm F}\right)-\mathcal{A}^*(\bm{b})\right)(\hat{\bm{V}}e^{\bm{F}})^\top\right)\big\Vert_F\notag\\
&\leq \Vert\mathcal{A}^* \mathcal{A}\Vert_{{\rm op}}(1+\Vert\hat{\bm{\sigma}}\Vert)+\Vert\mathcal{A}^*(\bm{b})\Vert_F,
\end{align}
where the last inequality holds because, by the assumption and the definition of $\delta$, $\Vert\bm{\sigma}\Vert\le\Vert\bm{\sigma}-\hat{\bm{\sigma}}\Vert+\Vert\hat{\bm{\sigma}}\Vert\le1+\Vert\hat{\bm{\sigma}}\Vert.$
Substituting \eqref{surj-key4}, \eqref{surj-key3} into \eqref{surj-key2} and using \eqref{delta1}, one checks that
\eqref{KLineq} holds  and so the proof is complete.
\end{proof}

The following lemma gives an important property of $\partial F_{\hat{\bm{\Omega}}}$ which will be used  in the subsequent convergence analysis.
\begin{lemma}\label{prop1add}For any $\delta_1\in(0,1)$, there exists $\eta>0$ such that for all $(\bm{\sigma},\bm{E},\bm{F})\in\mathcal{M}$ satisfying $\Vert\bm{E}\Vert_F\le\delta_1$ and $\Vert\bm{F}\Vert_F\le\delta_1$, the following inequality holds:
\begin{equation*}\label{subdiffcomp}
{\rm d}(\bm{0},\partial \mathcal{F}_{\hat{\bm{\Omega}}}(\bm{\sigma},\bm{E},\bm{F}))\le\eta{\rm d}(\bm{0},\partial \mathcal{F}_{\bm{{\Omega}}}(\bm{\sigma},\bm{0},\bm{0})),
\end{equation*}
where $\bm{\Omega}:=(\hat{\bm{U}}e^{\bm{E}},\hat{\bm{V}}e^{\bm{F}})$.
\end{lemma}

\begin{proof}Let $\delta_1\in(0,1)$ and $\eta:=\frac{1}{\sqrt{1-\delta_1}}$. We will prove that $\eta$ is as desired. For this end, let $(\bm{\sigma},\bm{E},\bm{F})\in\mathcal{M}$ satisfying $\Vert \bm{E}\Vert_F\le\delta_1$ and $\Vert \bm{F}\Vert_F\le\delta_1$.
Applying Proposition  \ref{prop1} (i) and using the definition of $\bm{\Omega}$, one has
$$\nabla_{\bm{\sigma}}f_{\hat{\bm{\Omega}}}(\bm{\sigma},\bm{E},\bm{F})
={\rm diag}\left(\hat{\bm{U}}e^{\bm{E}} \left(\mathcal{A}^*\mathcal{A}\left((\hat{\bm{U}}e^{\bm E})^\top \mathcal{D}(\bm{\sigma})\hat{\bm{V}}e^{\bm F}\right)-\mathcal{A}^*(\bm{b})\right)(\hat{\bm{V}}e^{\bm{F}})^\top\right)=\nabla_{\bm{\sigma}}f_{\bm{\Omega}}(\bm{\sigma}, \bm{0},\bm{0}).$$
Thus, it follows from Proposition \ref{lemma-par} (iii) that
\begin{align*}
{\rm d}^2(\bm{0},\partial \mathcal{F}_{\hat{\bm{\Omega}}}(\bm{\sigma},\bm{E},\bm{F}))= \Vert\nabla_{\bm{E}}f_{\hat{\bm{\Omega}}}(\bm{\sigma},\bm{E},\bm{F})\Vert_F^2+\Vert\nabla_{\bm{F}}f_{\hat{\bm{\Omega}}}(\bm{\sigma},\bm{E},\bm{F})\Vert_F^2+{\rm d}^2(-\nabla_{\bm{\sigma}}f_{\bm{\Omega}}(\bm{\sigma}, \bm{0},\bm{0}), \lambda\partial\Vert \bm{\sigma} \Vert_{p}^p).
\end{align*}
It suffices to prove \begin{align}\label{fomegaE}
\Vert \nabla_{\bm{E}}f_{\hat{\bm{\Omega}}}(\bm{\sigma},\bm{E},\bm{F})\Vert_F\leq\eta^2\Vert\nabla_{\bm{E}}f_{\bm{\Omega}}(\bm{\sigma},\bm{0},\bm{0})\Vert_F,\quad \Vert \nabla_{\bm{F}}f_{\hat{\bm{\Omega}}}(\bm{\sigma},\bm{E},\bm{F})\Vert_F\leq\eta^2\Vert\nabla_{\bm{F}}f_{\bm{\Omega}}(\bm{\sigma},\bm{0},\bm{0})\Vert_F.
\end{align}
Granting this, we get
\begin{align*}
{\rm d}^2(\bm{0},\partial \mathcal{F}_{\hat{\bm{\Omega}}}(\bm{\sigma},\bm{E},\bm{F}))&\leq\eta^2\Vert\nabla_{\bm{E}}f_{\bm{\Omega}}(\bm{\sigma},\bm{0},\bm{0})\Vert_F^2+\eta^2\Vert\nabla_{\bm{F}}f_{\bm{\Omega}}(\bm{\sigma},\bm{0},\bm{0})\Vert_F^2+{\rm d}^2(-\nabla_{\bm{\sigma}}f_{\bm{\Omega}}(\bm{\sigma}, \bm{0},\bm{0}), \lambda\partial\Vert \bm{\sigma} \Vert_{p}^p)\\
&\le\eta^2{\rm d}^2(\bm{0},\partial F_{\bm{\Omega}}(\bm{\sigma},\bm{0},\bm{0}))
\end{align*}
and so, the proof is completed.

We only prove the first inequality in \eqref{fomegaE} as the second one can be proved in  similar arguments.
For this end, suppose that $\bm{S}\in\mathcal{S}(m)$. Then, by definition, one has
\begin{align}\label{deltaADD}
\langle \nabla_{\bm{E}}f_{\hat{\bm{\Omega}}}(\bm{\sigma},\bm{E},\bm{F}), \bm{S}\rangle
=\lim_{t\rightarrow0}\frac{\Vert\mathcal{A} (\hat{\bm{U}}e^{\bm{E}+t\bm{S}})^\top \mathcal{D}(\bm{\sigma})\hat{\bm{V}}e^{\bm{F}})-\bm{b}\Vert^2
-\Vert\mathcal{A}((\hat{\bm{U}}e^{\bm{E}})^\top\mathcal{D}(\bm{\sigma})\hat{\bm{V}}e^{\bm{F}})-\bm{b}\Vert^2 }{2t}.
\end{align}
To proceed, let us define  $$ H_{1}(\bm{E},\bm{S}):=\bm{S},\quad H_{n}(\bm{E},\bm{S}):=H_{n-1}(\bm{E},\bm{S}){\bm E}-{\bm E}H_{n-1}(\bm{E},\bm{S}),\quad n=2,3,\ldots,$$  and  write $$\mathcal{P}(\bm{E},\bm{S}):=\sum\limits_{n=1}^{\infty}\frac{1}{n!}H_{n}(\bm{E},\bm{S}).$$
By using mathematical induction, one can easily check
 \begin{equation*}\label{Hproperty}\|H_{n}(\bm{E},\bm{S})\|_F\le n!\|\bm{E}\|_F^{n-1}\|\bm{S}\|_F\quad{\rm for\; each}\quad n=1,2,\ldots.\end{equation*}
Then, in view of the definition of $\mathcal{P}(\bm{E},\bm{S})$, we deduce
\begin{align}\label{pestimate}
\Vert \mathcal{P}(\bm{E},\bm{S})\Vert\le \|\bm{S}\|_F\sum\limits_{n=1}^{\infty}\|\bm{E}\|_F^{n-1} \leq\frac{ \|\bm{S}\|_F}{1-\delta_1}
\end{align}
(noting that  $\Vert \bm{E}\Vert_F\leq\delta_1<1$).
Thanks to \eqref{eX} and using again the definition of $\mathcal{P}(\bm{E},\bm{S})$, we have
\begin{align}\label{eE-1}
e^{\bm{E}+t\bm{S}}=e^{\bm{E}}(\bm{I}+t\mathcal{P}(\bm{E},\bm{S}))+o(t).
\end{align}
Substituting \eqref{eE-1} into \eqref{deltaADD} and using elementary  calculations, we obtain
\begin{align*}
\langle \nabla_{\bm{E}}f_{\hat{\bm{\Omega}}}(\bm{\sigma},\bm{E},\bm{F}), \bm{S}\rangle
&=\Big\langle \mathcal{A} \left((\hat{\bm{U}}e^{\bm{E}})^\top \mathcal{D}(\bm{\sigma})\hat{\bm{V}}e^{\bm{F}}\right)-\bm{b}, \;\mathcal{A}\left(\mathcal{P}(\bm{E},\bm{S})^\top(\hat{\bm{U}}e^{\bm{E}})^\top \mathcal{D}(\bm{\sigma})\hat{\bm{V}}e^{\bm{F}}\right)  \Big\rangle.
\end{align*}
Then, defining $\bm{Y}:=(\hat{\bm{U}}e^{\bm{E}})^\top \mathcal{D}(\bm{\sigma})\hat{\bm{V}}e^{\bm{F}}$, we have
\begin{equation}\label{NAB}
\begin{aligned}
\langle \nabla_{\bm{E}}f_{\hat{\bm{\Omega}}}(\bm{\sigma},\bm{E},\bm{F}), \;\bm{S}\rangle&=\Big\langle \mathcal{A}(\bm{Y})-\bm{b},\; \mathcal{A}( \mathcal{P}(\bm{E},\bm{S})^\top\bm{Y})\Big\rangle\\
&=\Big\langle\bm{Y}(\mathcal{A}^*(\mathcal{A} \left(\bm{Y}\right)-\bm{b}))^\top, \;\mathcal{P}(\bm{E},\bm{S})\Big\rangle\\
&=\Big\langle\frac{1}{2}\left[\bm{Y}(\mathcal{A}^*(\mathcal{A} \left(\bm{Y}\right)-\bm{b}))^\top-\mathcal{A}^*(\mathcal{A} \left(\bm{Y}\right)-\bm{b}) \bm{Y}^\top\right], \;\mathcal{P}(\bm{E},\bm{S})\Big\rangle,
\end{aligned}
\end{equation}
where the last equality holds because $\mathcal{P}(\bm{E},\bm{S})$ is skew-symmetric.
On the other hand,
applying  \eqref{fparsigma} (with $\bm{Y}$ in place of  $\bm{X}$), one sees
 $$\nabla_{\bm{E}}f_{\bm{\Omega}}(\bm{\sigma},\bm{0},\bm{0})=\frac{1}{2}\left(\bm{Y}\nabla f(\bm{Y})^\top-\nabla f(\bm{Y})\bm{Y}^\top \right)=\frac{1}{2}\left[\bm{Y}(\mathcal{A}^*(\mathcal{A} \left(\bm{Y}\right)-\bm{b}))^\top-\mathcal{A}^*(\mathcal{A} \left(\bm{Y}\right)-\bm{b}) \bm{Y}^\top\right].$$
This together with \eqref{NAB} yields
\begin{align*}
\langle \nabla_{\bm{E}}f_{\hat{\bm{\Omega}}}(\bm{\sigma},\bm{E},\bm{F}), \bm{S}\rangle=\Big\langle \nabla_{\bm{E}}f_{\bm{\Omega}}(\bm{\sigma},\bm{0},\bm{0}),\; \mathcal{P}(\bm{E},\bm{S})\Big\rangle.
\end{align*}
Therefore, thanks to \eqref{pestimate} and the definition of $\eta$, the first inequality in \eqref{fomegaE} follows immediately and so, the proof is complete.
\end{proof}

\subsection{Sublinear Convergence of the Proposed Algorithm}
As assumed in Section 3,  the sequences $\{\bm{X}_{k}\}$, $\{\bm{\Omega}_k=(\bm{U}_{k},\bm{V}_k)\}$ and $\{(\bm{\sigma}_k,\bm{E}_{k},\bm{F}_k)\}$ are generated by Algorithm \ref{alg1}. To analyze the convergence of the DPGA,  we provide some  auxiliary results.
\begin{lemma}\label{lemma3}Let $k\in\mathbb{N}$ and $L$ be a  positive constant. Then
the following  two assertions hold:
\begin{enumerate}
  \item  [\rm (i)]$\frac{2}{3}\Vert\bm{E}_{k}\Vert_F\le\Vert \bm{U}_{k+1}-\bm{U}_{k}\Vert_F \leq 2\Vert\bm{E}_{k}\Vert_F$,  $\frac{2}{3}\Vert\bm{F}_{k}\Vert_F\le\Vert \bm{V}_{k+1}-\bm{V}_k\Vert_F \leq 2\Vert\bm{F}_{k}\Vert_F$;
  \item [\rm (ii)]$\Vert \bm{X}_{k+1}-\bm{X}_{k}\Vert_F\leq L\Vert (\bm{\sigma}_{k+1}-\bm{\sigma}_k,\bm{E}_{k}, \bm{F}_{k})\Vert$.
\end{enumerate}

\end{lemma}

\begin{proof}(i). We only need to estimate $\Vert \bm{U}_{k+1}-\bm{U}_k\Vert_{F}$. Note by $(A_1)$ and the definition of $\bm{E}_{k}$ in Algorithm \ref{alg1}  that $\|\bm{E}_{k}\|_F\le1$. Since, by the definitions of $\bm{U}_{k+1}$ and $\C(\cdot)$ in \eqref{mapC}, $\Vert \bm{U}_{k+1}-\bm{U}_k\Vert_{F}=\Vert \C(\bm{E}_{k})-\bm{I}\Vert_{F},$ the first inequality in assertion (i) is seen to hold by applying Lemma \ref{fact3} (with $\bm{X}$ replaced by $\bm{E}_{k}$).

(ii). It suffices to prove  there exists a positive constant $L_1$  such that
\begin{equation}\label{xkL}\Vert \bm{X}_{k+1}-\bm{X}_{k}\Vert_F\leq L_1\Vert (\bm{\sigma}_{k+1},\bm{U}_{k+1},\bm{V}_{k+1})-(\bm{\sigma}_k,\bm{U}_{k},\bm{V}_k)\Vert\end{equation}
(see assertion (i)). By  the definition of $\{\bm{X}_{k}\}$ in Algorithm \ref{alg1}, one has
\begin{align*}
\Vert \bm{X}_{k+1}-\bm{X}_{k}\Vert_F
&=\Vert \bm{U}_{k+1}^\top \mathcal{D}(\bm{\sigma}_{k+1})\bm{V}_{k+1}-\bm{U}_{k}^\top \mathcal{D}(\bm{\sigma}_{k})\bm{V}_{k}\Vert_F\notag\\
&\leq \max\{\Vert\bm{\sigma}_{k+1}\Vert,\Vert\bm{\sigma}_{k}\Vert,1\}(\Vert \bm{U}_{k+1}-\bm{U}_k\Vert_F+\Vert\bm{\sigma}_{k+1}-\bm{\sigma}_k\Vert+\Vert \bm{V}_{k+1}-\bm{V}_k\Vert_F)\notag\\
&\leq \sqrt{3}\max\{\Vert\bm{\sigma}_{k+1}\Vert,\Vert\bm{\sigma}_{k}\Vert,1\}\Vert (\bm{\sigma}_{k+1},\bm{U}_{k+1},\bm{V}_{k+1})-(\bm{\sigma}_k,\bm{U}_{k},\bm{V}_k)\Vert,
\end{align*}
where the last inequality holds because of the  fact that $(a+b+c)^2\leq 3(a^2+b^2+c^2)$ for any $ a,\;b,\;c\ge0.$
Thus, thanks to Remark \ref{rmk:cluster} (i),  \eqref{xkL} follows immediately  and so, the proof is complete.
\end{proof}

\begin{lemma}\label{prop4-2}
There exists a positive number $\beta>0$  such that
\begin{equation}\label{H_2}
{\rm d}(0,\partial \mathcal{F}_{\bm{\Omega}_k}(\bm{\sigma}_k,\bm{0},\bm{0}))\leq \beta\Vert (\bm{\sigma}_{k}-\bm{\sigma}_{k-1},\bm{E}_{k-1},\bm{F}_{k-1})\Vert\quad {\rm for\; each} \quad k\ge1.
\end{equation}

\end{lemma}
\begin{proof}To begin, fix $k\ge1$. Applying Lemma \ref{prop2} (i), we have
\begin{align*}
&\Vert\nabla_{\bm{\sigma}} f_{\bm{\Omega}_{k}}(\bm{\sigma}_{k},\bm{0},\bm{0})-\nabla_{\bm{\sigma}} f_{\bm{\Omega}_{k-1}}(\bm{\sigma}_{k-1},\bm{0},\bm{0}) \Vert\\
\le\ & \Vert \mathcal{A}^* \mathcal{A}\Vert_{\rm{op}}\Vert \bm{\sigma}_{k}-\bm{\sigma}_{k-1}\Vert+(\Vert\nabla f(\bm{X}_k)\Vert_F+\Vert\mathcal{A}^*\mathcal{A}\Vert_{{\rm op}}\Vert \bm{\sigma}_k\Vert)\|\bm{U}_k-\bm{U}_{k-1}\|_F\\
&+(\Vert\nabla f(\bm{X}_k)\Vert_F+\Vert\mathcal{A}^*\mathcal{A}\Vert_{{\rm op}}\Vert \bm{\sigma}_{k-1}\Vert)\|\bm{V}_{k}-\bm{V}_{k-1}\|_F.
\end{align*}
Then, thanks to Lemma \ref{lemma3} (i) and Remark \ref{rmk:cluster} (i),  one checks that there exists a positive constant $L_1$ such that
$$\Vert\nabla_{\bm{\sigma}} f_{\bm{\Omega}_{k}}(\bm{\sigma}_{k},\bm{0},\bm{0})-\nabla_{\bm{\sigma}} f_{\bm{\Omega}_{k-1}}(\bm{\sigma}_{k-1},\bm{0},\bm{0}) \Vert\le L_1\Vert (\bm{\sigma}_{k}-\bm{\sigma}_{k-1}, \bm{E}_{k-1},\bm{F}_{k-1})\Vert.$$
Similarly, we can also prove the existence of Lipschitz constant $L_1$ for $\nabla_{\bm{E}} f$ and $\nabla_{\bm{F}} f$.
Thus, it follows that
\begin{equation}\label{L-constant}
\Vert\nabla f_{\bm{\Omega}_{k}}(\bm{\sigma}_{k},\bm{0},\bm{0})-\nabla f_{\bm{\Omega}_{k-1}}(\bm{\sigma}_{k-1},\bm{0},\bm{0}) \Vert\leq L_1\Vert (\bm{\sigma}_{k}-\bm{\sigma}_{k-1}, \bm{E}_{k-1},\bm{F}_{k-1})\Vert.
\end{equation}
Let $\underline{s},\;\underline{t}>0$  be determined by Remark \ref{sktkbound}. Set
\begin{equation}\label{deltade} \beta:=\frac{1}{\min\{\underline{s},\underline{t}\}}+L_1.
\end{equation}
Below we will prove that $\beta$ is  as desired.
Note by \eqref{eq:sigma} and \cite[Theorem 10.1]{Mordukhovich2006Variational} that
$$\bm 0\in\nabla_{\bm{\sigma}} f_{\bm{\Omega}_{k-1}}(\bm{\sigma}_{k-1},\bm{0},\bm{0})+\lambda\partial\Vert \bm{\sigma}_{k}\Vert_{p}^p+\frac{1}{t_{k-1}}(\bm{\sigma}_{k}-\bm{\sigma}_{k-1}).$$
Thus, there exists $\bm{z}_{k}\in\lambda\partial\Vert \bm{\sigma}_{k}\Vert_{p}^p$ such that
$\nabla_{\bm{\sigma}} f_{\bm{\Omega}_{k-1}}(\bm{\sigma}_{k-1},\bm{0},\bm{0})+\bm{z}_{k}=-\frac{1}{t_{k-1}}(\bm{\sigma}_{k}-\bm{\sigma}_{k-1})$. This together with \eqref{ExpE} as well as Remark \ref{sktkbound} yields
\begin{align}\label{grabd}
\Vert\nabla f_{\bm{\Omega}_{k-1}}(\bm{\sigma}_{k-1},\bm{0},\bm{0})+(\bm{z}_{k},\bm{0},\bm{0})\Vert&=\bigg\Vert\left(\frac{1}{t_{k-1}}(\bm{\sigma}_{k}-\bm{\sigma}_{k-1}),\frac{1}{s_{k-1}}\bm{E}_{k-1},\frac{1}{s_{k-1}}\bm{F}_{k-1}\right)\bigg\Vert\notag\\
&\leq\frac{1}{\min\{\underline{s},\underline{t}\}} \Vert(\bm{\sigma}_{k}-\bm{\sigma}_{k-1},\bm{E}_{k-1},\bm{F}_{k-1})\Vert.
\end{align}
To proceed, let $\bm{v}_k:=\nabla f_{\bm{\Omega}_{k}}(\bm{\sigma}_{k},\bm{0},\bm{0})+(\bm{z}_{k},\bm{0},\bm{0})$.
Obviously, $\bm{v}_{k}\in\partial \mathcal{F}_{\bm{\Omega}_k}(\bm{\sigma}_k,\bm{0},\bm{0})$. Moreover, from \eqref{grabd} and \eqref{L-constant}, we derive
\begin{align*}
\Vert \bm{v}_{k}\Vert&\leq \Vert\nabla f_{\bm{\Omega}_{k-1}}(\bm{\sigma}_{k-1},\bm{0},\bm{0})+(\bm{z}_{k},\bm{0},\bm{0})\Vert
+\Vert\nabla f_{\bm{\Omega}_{k}}(\bm{\sigma}_{k},\bm{0},\bm{0})-\nabla f_{\bm{\Omega}_{k-1}}(\bm{\sigma}_{k-1},\bm{0},\bm{0}) \Vert\\
&\leq\left(\frac{1}{\min\{\underline{s},\underline{t}\}}+L_1\right) \Vert(\bm{\sigma}_{k}-\bm{\sigma}_{k-1},\bm{E}_{k-1},\bm{F}_{k-1})\Vert.
\end{align*}
 This together with \eqref{deltade} shows \eqref{H_2} and so, the proof is completed.
\end{proof}

Below we will present an important result which  is crucial for the proof of the main theorem in this subsection. For this purpose,
recall that $(\bm{\bar{\sigma}},\bm{\bar{U}},\bm{\bar{V}})$ is a cluster point of $\{(\bm{\sigma}_k,\bm{U}_k,\bm{V}_k)\}$ (as we assumed in Remark \ref{rmk:cluster} (i)) and let $\bar{\bm{\Omega}}:=(\bar{\bm{U}},\bar{\bm{V}})\in\mathcal{O}(m)\times\mathcal{O}(n)$.
Set $\mathbb{N}_0:=\{k\in\N: \;\Vert\bm{\bar{U}}^{\top}\bm{U}_{k}-\bm{I}\Vert_F<1,\;\Vert\bm{\bar{V}}^{\top}\bm{V}_{k}-\bm{I}\Vert_F<1\}$ and fix $k\in\mathbb{N}_0$.
Write
\begin{equation}\label{Ekbarde}\bm{\bar{E}}_k:={\rm log}(\bm{\bar{U}}^{\top}\bm{U}_{k})\quad{\rm and}\quad\bm{\bar{F}}_k:={\rm log}(\bm{\bar{V}}^{\top}\bm{V}_{k}).\end{equation}
Since $(\bm{\bar{\sigma}},\bm{\bar{U}},\bm{\bar{V}})$ is a cluster point of $\{(\bm{\sigma}_k,\bm{U}_k,\bm{V}_k)\}$, one checks by Lemma \ref{lemmaexplog}  (ii) that $(\bar{\bm{\sigma}}, \bm{0}, \bm{0})$ is a cluster point of $\{(\bm{\sigma}_k, \bm{\bar{E}}_k, \bm{\bar{F}}_k)\}$.
By \eqref{Ekbarde},  we also have
\begin{equation}\label{Ekbardepro}e^{\bm{\bar{E}}_k}=\bm{\bar{U}}^{\top}\bm{U}_{k},\quad e^{\bm{\bar{F}}_k}=\bm{\bar{V}}^{\top}\bm{V}_{k}\quad {\rm and }\quad\mathcal{F}_{\bm{\bar{\Omega}}}(\bm{\sigma}_k, \bm{\bar{E}}_k, \bm{\bar{F}}_k)=\mathcal{F}(\bm{X}_k).\end{equation}
Finally, we set for simplicity \begin{equation}\label{akdef}a_k:=\mathcal{F}_{\bm{\bar{\Omega}}}(\bm{\sigma}_k, \bm{\bar{E}}_k, \bm{\bar{F}}_k)-F_{\bm{\bar{\Omega}}}(\bm{\bar{\sigma}},{\bm 0},{\bm 0})=\mathcal{F}(\bm{X}_k)-\mathcal{F}(\bar{\bm{X}}).\end{equation}

\begin{prop}\label{prop5}Let $\theta\in[0,1)$ be the   K{\L} exponent of $\mathcal{F}_{\bm{\bar{\Omega}}}$. There exist $N\in \N$ and positive numbers $\delta_2\in(0,\frac12)$, $a$, $b$ and $\mu$ (independent of $N$) such that for each $k\ge N$, if
\begin{equation}\label{condi}k\in\mathbb{N}_0\quad\mbox {and }\quad(\bm{\sigma}_{k},\bar{\bm{E}}_{k},\bar{\bm{F}}_{k})\in {\bf B}((\bm{\bar{\sigma}},\bm{0},\bm{0}),\delta_2),\end{equation}
  then $k+1\in\mathbb{N}_0$ and
the following assertions hold:

\begin{itemize}
\item[${\rm(i)}$]$\frac13\Vert \bm{E}_{k}\Vert_F\le\Vert\bm{\bar{E}}_{k+1}-\bm{\bar{E}}_{k}\Vert_F\le4\|\bm{E}_{k}\|_F$ and $\frac13\Vert \bm{F}_{k}\Vert_F\le\Vert\bm{\bar{F}}_{k+1}-\bm{\bar{F}}_{k}\Vert_F\le4\|\bm{F}_{k}\|_F$;
\item[${\rm(ii)}$]$\mathcal{F}_{\bm{\bar{\Omega}}}(\bm{\sigma}_{k+1},\bm{\bar{E}}_{k+1},\bm{\bar{F}}_{k+1})+a\Vert(\bm{\sigma}_{k+1}-\bm{\sigma}_{k},\bm{\bar{E}}_{k+1}-\bm{\bar{E}}_{k},\bm{\bar{F}}_{k+1}-\bm{\bar{F}}_{k})\Vert^2
\leq \mathcal{F}_{\bm{\bar{\Omega}}}(\bm{\sigma}_{k},\bm{\bar{E}}_{k},\bm{\bar{F}}_{k})$;
  \item [$ {\rm (iii)}$]${\rm d}(\bm{0},\partial \mathcal{F}_{\bm{\bar{\Omega}}}(\bm{\sigma}_{k+1},\bm{\bar{E}}_{k+1},\bm{\bar{F}}_{k+1}))\leq b \Vert(\bm{\sigma}_{k+1}-\bm{\sigma}_k,\bm{\bar{E}}_{k+1}-\bm{\bar{E}}_{k},\bm{\bar{F}}_{k+1}-\bm{\bar{F}}_{k})\Vert$;
  \item [${\rm(iv)}$]$6 \Vert(\bm{\sigma}_{k+1}-\bm{\sigma}_{k},\bm{\bar{E}}_{k+1}-\bm{\bar{E}}_{k},\bm{\bar{F}}_{k+1}-\bm{\bar{F}}_{k})\Vert\leq \mu(a_k^{1-\theta}-a_{k+1}^{1-\theta})+\Vert (\bm{\sigma}_{k}-\bm{\sigma}_{k-1},\bm{E}_{k-1},\bm{F}_{k-1})\Vert.$
\end{itemize}

\end{prop}
\begin{proof}Let $\kappa>0$, $\epsilon\in(0,+\infty]$ and $\delta_1\in(0,1)$ have the properties stated in Proposition \ref{lemma2} with $\hat{\bm{\sigma}}=\bar{\bm{\sigma}}$ and $\hat{\bm{\Omega}}=\bar{\bm{\Omega}}$. 
Set $$ \delta_2:=\min\left\{\frac14{\rm log}2,\;\delta_1\right\}. $$
According to Lemma \ref{prop1add},  there exists $\eta>0$ such that the following implication holds for all $k\in\N$:
\begin{equation}\label{theimp}
(\bm{\sigma}_k,\bm{\bar{E}}_{k},\bm{\bar{F}}_{k})\in{\bf{B}}\left((\bm{\bar{\sigma}},{\bm 0},{\bm 0}),\frac12{\rm log}2\right)\subseteq\mathcal{M}
\Longrightarrow{\rm d}(\bm{0},\partial \mathcal{F}_{\bar{\bm{\Omega}}}(\bm{\sigma}_k,\bm{\bar{E}}_{k},\bm{\bar{F}}_{k})\le\eta{\rm d}(\bm{0},\partial \mathcal{F}_{\bm{{\Omega}}_k}(\bm{\sigma}_k,\bm{0},\bm{0})).
\end{equation}
Let $\alpha$, $\beta$ be determined in Lemma \ref{prop4} and Lemma \ref{prop4-2}, respectively.
Write $$a:=\frac{\alpha}{16},\quad b:=3\eta\beta\quad{\rm and}\quad\mu:=\frac{9\beta\kappa\eta}{a(1-\theta)}.$$
 By  Remark \ref{rmk:cluster} and the third equality in \eqref{Ekbardepro}, there exists $N\in\N$ such that for each $k\ge N$,
\begin{equation}\label{ekfkcon}\max\{\Vert\bm{E}_{k}\Vert_F,\;\Vert\bm{F}_{k}\Vert_F\}\le\frac{1}{16}{\rm log}2\quad{\rm and}\quad\mathcal{F}_{\bm{\bar{\Omega}}}(\bm{\bar{\sigma}},{\bm 0},{\bm 0})<\mathcal{F}_{\bm{\bar{\Omega}}}(\bm{\sigma}_k, \bm{\bar{E}}_k, \bm{\bar{F}}_k)<\mathcal{F}_{\bm{\bar{\Omega}}}(\bm{\bar{\sigma}},{\bm 0},{\bm 0})+\epsilon.\end{equation}
 Below we shall show that $\delta_2$,  $a$, $b$, $\mu$ and $ N$ are as desired.
For this end, let $k\ge N$ and suppose that $(\bm{\sigma}_{k},\bar{\bm{E}}_{k},\bar{\bm{F}}_{k})\in {\bf B}((\bm{\bar{\sigma}},\bm{0},\bm{0}),\delta_2)$.

(i) and (ii). We first note that assertion (ii) follows immediately from assertion (i), Proposition \ref{prop4}  and the definition of $a$. While, to prove assertion (i), it suffices to prove that $\bm{\bar{E}}_{k+1}$ is well defined and the first conclusion in assertion (i) holds.
In fact, by  Lemma \ref{lemma3} (i) and the first inequality in \eqref{ekfkcon}, we obtain
\begin{equation}\label{ukeku}\Vert\bm{\bar{U}}^\top({\bm U}_{k+1}-{\bm U}_{k})\Vert_F=\Vert{\bm U}_{k+1}-{\bm U}_{k}\Vert_F\le2\|\bm{E}_{k}\|_F\le\frac{1}{8}{\rm log}2.\end{equation}
Using \eqref{Ekbardepro}, Lemma \ref{lemmaexplog} (iv), (vi) and the fact that $\Vert\bar{\bm{E}}_{k}\Vert_F\le\delta_2\le\frac{1}{4}{\rm log}2$, one has
\begin{equation*}
\Vert\bm{\bar{U}}^{\top}\bm{U}_{k}-\bm{I}\Vert_F=\Vert e^{\bar{\bm{E}}_{k}}-\bm{I}\Vert_F\le\Vert\bar{\bm{E}}_{k}\Vert_F\le\frac{1}{4}{\rm log}2.
\end{equation*}
Combining this with \eqref{ukeku}, we have
\begin{equation}\label{H_1imp}3\Vert \bm{\bar{U}}^\top{\bm U}_{k}-{\bm I}\Vert_F+\Vert\bm{\bar{U}}^\top({\bm U}_{k+1}-{\bm U}_{k})\Vert_F<\frac12.\end{equation}
This particularly gives $\Vert\bm{\bar{U}}^{\top}\bm{U}_{k+1}-\bm{I}\Vert_F<\frac12$.
Hence  $\bar{\bm{E}}_{k+1}$  is well defined.
Moreover, in view of \eqref{H_1imp}, \eqref{Ekbarde} and applying  Lemma \ref{explogcon} (ii) (with $\bm{\bar{U}}^\top{\bm U}_{k}$, $\bm{\bar{U}}^\top({\bm U}_{k+1}-{\bm U}_{k})$ in place of ${\bm X}$ and $\Delta {\bm X}$, respectively), we get
\begin{equation*}
\Vert\bm{\bar{E}}_{k+1}-\bm{\bar{E}}_{k}\Vert_F=\Vert {\rm log}(\bm{\bar{U}}^\top{\bm U}_{k+1})-{\rm log}(\bm{\bar{U}}^\top{\bm U}_{k})\Vert_F\le2\Vert\bm{\bar{U}}^\top({\bm U}_{k+1}-{\bm U}_{k})\Vert_F.
\end{equation*}
This together with \eqref{ukeku} yields
\begin{equation}\label{ekupp}\Vert\bm{\bar{E}}_{k+1}-\bm{\bar{E}}_{k}\Vert_F\le4\|\bm{E}_{k}\|_F.\end{equation}
Hence, by the first inequality in \eqref{ekfkcon} and the fact that $\Vert \bm{\bar{E}}_{k}\Vert_F\le\delta_2\le\frac{1}{4}{\rm log}2$, one has
\begin{equation*}
2\Vert \bm{\bar{E}}_{k}\Vert_F+\Vert\bm{\bar{E}}_{k+1}-\bm{\bar{E}}_{k}\Vert_F\le\frac34{\rm log}2<1.
\end{equation*}
Thus, Lemma \ref{explogcon} (i) is applicable (with
 $\bm{X}:=\bm{\bar{E}}_{k}$
and $\Delta\bm{X}:=\bm{\bar{E}}_{k+1}-\bm{\bar{E}}_{k}$) and so,
$$\Vert{\bm U}_{k+1}-{\bm U}_{k}\Vert_F=\Vert\bm{\bar{U}}^\top({\bm U}_{k+1}-{\bm U}_{k})\Vert_F=\Vert e^{\bm{\bar{E}}_{k+1}}-e^{\bm{\bar{E}}_{k}}\Vert_F\le2\Vert e^{\bm{\bar{E}}_{k}}\Vert_F\Vert \bm{\bar{E}}_{k+1}-\bm{\bar{E}}_{k}\Vert_F=2\Vert \bm{\bar{E}}_{k+1}-\bm{\bar{E}}_{k}\Vert_F.$$
Therefore, by Lemma  \ref{lemma3} (i),  one sees $\frac13\Vert \bm{E}_{k}\Vert_F\le\Vert\bm{\bar{E}}_{k+1}-\bm{\bar{E}}_{k}\Vert_F$.  This, together with \eqref{ekupp}, shows the first conclusion in assertion (i).

(iii). Note  by \eqref{ekupp}, \eqref{ekfkcon} and the fact that $\Vert \bm{\bar{E}}_{k}\Vert_F\le\delta_2\le\frac{1}{4}{\rm log}2$, one has
$\Vert\bm{\bar{E}}_{k+1}\Vert_F\le\frac12{\rm log}2.$ Similarly, $\Vert\bm{\bar{F}}_{k+1}\Vert_F\le\frac12{\rm log}2.$
Then, using Lemma \ref{lemmaexplog} (vi) (with $\bm{X}$ replaced by $\bm{\bar{U}}^{\top}\bm{U}_{k+1}$ and $\bm{\bar{V}}^{\top}\bm{V}_{k+1}$, respectively), we know that $(\bm{\bar{E}}_{k+1},\bm{\bar{F}}_{k+1})\in \mathcal{S}(m)\times \mathcal{S}(n)$. Thus, implication  \eqref{theimp} is applicable (with $k+1$ in place of $k$) and so
$${\rm d}(\bm{0},\partial \mathcal{F}_{\bar{\bm{\Omega}}}(\bm{\sigma}_{k+1},\bm{\bar{E}}_{k+1},\bm{\bar{F}}_{k+1})\le\eta{\rm d}(\bm{0},\partial \mathcal{F}_{\bm{{\Omega}}_{k+1}}(\bm{\sigma}_{k+1},\bm{0},\bm{0})).$$
 Consequently, using Lemma \ref{prop4-2} (with $k+1$ in place of $k$) and assertion (i), one sees that assertion (iii) holds  as $b=3\eta\beta$.

(iv). Using Lemma \ref{lemmaexplog} (vi) (with $\bm{X}$ replaced by $\bm{\bar{U}}^{\top}\bm{U}_{k}$ and $\bm{\bar{V}}^{\top}\bm{V}_{k}$, respectively), we know that $(\bm{\bar{E}}_k,\bm{\bar{F}}_k)\in \mathcal{S}(m)\times \mathcal{S}(n)$. Then, in view of the definition of $\delta_2$, one sees that \eqref{theimp} is applicable and so,
$${\rm d}(\bm{0},\partial \mathcal{F}_{\bar{\bm{\Omega}}}(\bm{\sigma}_{k},\bm{\bar{E}}_{k},\bm{\bar{F}}_{k})\le\eta{\rm d}(\bm{0},\partial \mathcal{F}_{\bm{{\Omega}}_{k}}(\bm{\sigma}_{k},\bm{0},\bm{0})).$$
Thus, applying Proposition \ref{lemma2} (with  $\bar{\bm{\sigma}}$, $\bar{\bm{\Omega}}$, $\bm{\Omega}_k$, $(\bm{\sigma}_k, \bm{\bar{E}}_k, \bm{\bar{F}}_k)$ in place of $\hat{\bm{\sigma}}$, $\hat{\bm{\Omega}}$, $\bm{\Omega}$ and $ (\bm{\sigma}, \bm{E}, \bm{F})$ respectively), we have
$$( F_{\bm{\bar{\Omega}}}(\bm{\sigma}_k, \bm{\bar{E}}_k, \bm{\bar{F}}_k)-F_{\bm{\bar{\Omega}}}(\bm{\bar{\sigma}},{\bm 0},{\bm 0}))^{\theta}\le\kappa\eta{\rm d}(\bm{0},\partial F_{\bm{\Omega}_k}(\bm{\sigma}_k,\bm{0},\bm{0})).$$
Hence, thanks to Lemma \ref{prop4-2} and the definition of $a_k$, we conclude that
\begin{equation}\label{H2con}
a_k^{\theta}\le\beta\kappa\eta
\Vert(\bm{\sigma}_{k}-\bm{\sigma}_{k-1},\bm{E}_{k-1},\bm{F}_{k-1})\Vert.
\end{equation}
Since $(\cdot)^{1-\theta}$ is convex, we have by assertion (ii) and the definition of $\{a_k\}$ that
\begin{align}\label{phiykyk1}
a_k^{1-\theta}-a_{k+1}^{1-\theta}
&\geq (1-\theta)a_k^{-\theta}[\mathcal{F}_{\bm{\bar{\Omega}}}(\bm{\sigma}_{k},\bm{\bar{E}}_{k},\bm{\bar{F}}_{k})-\mathcal{F}_{\bm{\bar{\Omega}}}(\bm{\sigma}_{k+1},\bm{\bar{E}}_{k+1},\bm{\bar{F}}_{k+1})]\notag\\
&\geq a(1-\theta)a_k^{-\theta}\Vert(\bm{\sigma}_{k+1}-\bm{\sigma}_{k},\bm{\bar{E}}_{k+1}-\bm{\bar{E}}_{k},\bm{\bar{F}}_{k+1}-\bm{\bar{F}}_{k})\Vert^2.
\end{align}
Substituting \eqref{H2con} into \eqref{phiykyk1} and using the definition of $\mu$, we get
\begin{align*}
\Vert(\bm{\sigma}_{k+1}-\bm{\sigma}_{k},\bm{\bar{E}}_{k+1}-\bm{\bar{E}}_{k},\bm{\bar{F}}_{k+1}-\bm{\bar{F}}_{k})\Vert^2\le\frac{\mu}{9}( a_k^{1-\theta}- a_{k+1}^{1-\theta} )\Vert(\bm{\sigma}_{k}-\bm{\sigma}_{k-1},\bm{E}_{k-1},\bm{F}_{k-1})\Vert.
\end{align*}
This gives $$\begin{array}{lll}6\Vert(\bm{\sigma}_{k+1}-\bm{\sigma}_{k},\bm{\bar{E}}_{k+1}-\bm{\bar{E}}_{k},\bm{\bar{F}}_{k+1}-\bm{\bar{F}}_{k})\Vert&\le2\sqrt{\mu( a_k^{1-\theta}- a_{k+1}^{1-\theta} )\Vert(\bm{\sigma}_{k}-\bm{\sigma}_{k-1},\bm{E}_{k-1},\bm{F}_{k-1})\Vert}\\
&\le \mu( a_k^{1-\theta}- a_{k+1}^{1-\theta} )+\Vert(\bm{\sigma}_{k}-\bm{\sigma}_{k-1},\bm{E}_{k-1},\bm{F}_{k-1})\Vert,\end{array}$$
where the last inequality holds because of the fact that $2\sqrt{ab}\le a+b$ for any $a,b\ge0$.
Therefore, assertion (iv) is seen to hold and the proof is completed.
\end{proof}

The main theorem of this subsection is as follows which provides the sublinear  convergence property for the  sequence $\{\bm{X}_k\}$, together with $\{ \mathcal{F}(\bm{X}_k)\}$, generated by the proposed algorithm.
\begin{theorem}\label{thm2}The sequence $\{\bm{X}_k\}$ (resp. $\{(\bm{\sigma}_k,\bm{U}_k,\bm{V}_k)\}$) converges to $\bm{\bar{X}}$  (resp. $(\bm{\bar{\sigma}},\bm{\bar{U}},\bm{\bar{V}})$) with
$ \bm{0}\in\partial \mathcal{F}_{\bm{\bar{\Omega}}}(\bm{\bar{\sigma}},{\bm 0},{\bm 0})$. Moreover, there exist $\theta_1,\;\theta_2\in(1,+\infty)$ such that
  $$\mathcal{F}(\bm{X}_k)-\mathcal{F}(\bm{\bar{X}})=O(k^{-\theta_1})\quad{\rm and}\quad
\Vert\bm{X}_k-\bm{\bar{X}}\Vert=O(k^{-\theta_2})\quad{\rm for\; each}\quad k\ge1.$$
\end{theorem}

\begin{proof}
Let $\delta_2\in(0,\;\frac12)$, $\mu$ and $N$ be determined by Proposition \ref{prop5}. To proceed, let $\rho\in(0,\;\delta_2)$.
Recall that  $(\bar{\bm{\sigma}}, \bm{0}, \bm{0})$ is a cluster point of $\{(\bm{\sigma}_k, \bm{\bar{E}}_k, \bm{\bar{F}}_k)\}$. Then, in view of Remark \ref{rmk:cluster} and \eqref{akdef},  we may assume without loss of generality that
\begin{align*}
& \frac{\mu}{3} a_{N}^{1-\theta}+\Vert(\bm{\sigma}_{N},\bar{\bm{E}}_{N},\bar{\bm{F}}_{N})-(\bm{\bar{\sigma}},\bm{0},\bm{0})\Vert+8\Vert (\bm{\sigma}_{N+1}-\bm{\sigma}_{N},\bm{E}_{N},\bm{F}_{N})\Vert<\rho.
\end{align*}
Below we will prove that \eqref{condi} holds for each $k\ge N$. Granting this, we have  by Proposition \ref{prop5} that assertions (ii) and (iii) in Proposition \ref{prop5} hold for each $k\ge N$. Thus, recalling that  $\mathcal{F}_{\bar{\bm{\Omega}}}$ has the   K{\L} property at $(\bm{\bar{\sigma}},{\bm 0},{\bm 0})$, the conclusion follows from  \cite[Theorem 3.1, Theorem 3.4]{Frankel}.
  To this end, we consider the following inequality:
 \begin{equation}\label{consid}
 \begin{aligned}
& \sum\limits_{i=N+1}^{k}\Vert (\bm{\sigma}_{i+1}-\bm{\sigma}_{i},\bar{\bm{E}}_{i+1}-\bar{\bm{E}}_{i},\bar{\bm{F}}_{i+1}-\bar{\bm{F}}_{i})\Vert+\Vert (\bm{\sigma}_{k+1}-\bm{\sigma}_{k},\bar{\bm{E}}_{k+1}-\bar{\bm{E}}_{k},\bar{\bm{F}}_{k+1}-\bar{\bm{F}}_{k})\Vert\\
  \leq\ & \frac{\mu}{3}(a_{N+1}^{1-\theta}-a_{k+1}^{1-\theta})+\Vert (\bm{\sigma}_{N+1}-\bm{\sigma}_{N},\bar{\bm{E}}_{N+1}-\bar{\bm{E}}_{N},\bar{\bm{F}}_{N+1}-\bar{\bm{F}}_{N})\Vert.
 \end{aligned}
 \end{equation}
Clearly, \eqref{condi}  and \eqref{consid} are trivial for $k=N$.

Assume that, for each $N\le k\le l$, \eqref{condi} and \eqref{consid} hold. Then,  Proposition \ref{prop5} is applicable for all $N\le k\le l$. In particular, $l+1\in\mathbb{N}_0$ and Proposition \ref{prop5} (i)  holds for $k=l$. Note that
\begin{equation}\label{enbar}
\begin{aligned}
\Vert(\bm{\sigma}_{l+1},\bar{\bm{E}}_{l+1},\bar{\bm{F}}_{l+1})-(\bm{\bar{\sigma}},\bm{0},\bm{0})\Vert&\le\Vert(\bm{\sigma}_{N},\bar{\bm{E}}_{N},\bar{\bm{F}}_{N})-(\bm{\bar{\sigma}},\bm{0},\bm{0})\Vert\\
&\quad +\sum\limits_{i=N+1}^{l}\Vert (\bm{\sigma}_{i+1}-\bm{\sigma}_{i},\bar{\bm{E}}_{i+1}-\bar{\bm{E}}_{i},\bar{\bm{F}}_{i+1}-\bar{\bm{F}}_{i})\Vert\\
&\quad +\Vert (\bm{\sigma}_{N+1}-\bm{\sigma}_{N},\bar{\bm{E}}_{N+1}-\bar{\bm{E}}_{N},\bar{\bm{F}}_{N+1}-\bar{\bm{F}}_{N})\Vert.
\end{aligned}
\end{equation}
Thanks to \eqref{consid} (with $k=l$) and the fact that $a_N\ge a_{N+1}$, we deduce
\begin{equation}\label{460k=l}\sum\limits_{i=N+1}^{l}\Vert (\bm{\sigma}_{i+1}-\bm{\sigma}_{i},\bar{\bm{E}}_{i+1}-\bar{\bm{E}}_{i},\bar{\bm{F}}_{i+1}-\bar{\bm{F}}_{i})\Vert<\frac{\mu}{3} a_{N}^{1-\theta}+\Vert (\bm{\sigma}_{N+1}-\bm{\sigma}_{N},\bar{\bm{E}}_{N+1}-\bar{\bm{E}}_{N},\bar{\bm{F}}_{N+1}-\bar{\bm{F}}_{N})\Vert.\end{equation}
Substituting \eqref{460k=l} into \eqref{enbar} and using Proposition \ref{prop5} (i) (with $k=N$), we obtain
$$
\Vert(\bm{\sigma}_{l+1},\bar{\bm{E}}_{l+1},\bar{\bm{F}}_{l+1})-(\bm{\bar{\sigma}},\bm{0},\bm{0})\Vert
< \frac{\mu}{3} a_{N}^{1-\theta}+\Vert(\bm{\sigma}_{N},\bar{\bm{E}}_{N},\bar{\bm{F}}_{N})-(\bm{\bar{\sigma}},\bm{0},\bm{0})\Vert+8\Vert (\bm{\sigma}_{N+1}-\bm{\sigma}_{N},\bm{E}_{N},\bm{F}_{N})\Vert,
$$
and so, by  the definition of $\rho$, one sees \eqref{condi} holds for $k=l+1$.
Moreover, applying Proposition \ref{prop5} (iv) (with $l+1$ in place of $k$), we have
$$ 6\Vert(\bm{\sigma}_{l+2}-\bm{\sigma}_{l+1},\bm{\bar{E}}_{l+2}-\bm{\bar{E}}_{l+1},\bm{\bar{F}}_{l+2}-\bm{\bar{F}}_{l+1})\Vert\leq \mu(a_{l+1}^{1-\theta}-a_{l+2}^{1-\theta})+\Vert (\bm{\sigma}_{l+1}-\bm{\sigma}_{l},\bm{E}_{l},\bm{F}_{l})\Vert.$$
This together with  Proposition \ref{prop5} (i) (with $k=l$) gives
$$2\Vert (\bm{\sigma}_{l+2}-\bm{\sigma}_{l+1},\bm{\bar{E}}_{l+2}-\bm{\bar{E}}_{l+1},\bm{\bar{F}}_{l+2}-\bm{\bar{F}}_{l+1})\Vert
\le\frac{\mu}{3}( a_{l+1}^{1-\theta}- a_{l+2}^{1-\theta} )+\Vert(\bm{\sigma}_{l+1}-\bm{\sigma}_{l},\bm{\bar{E}}_{l+1}-\bm{\bar{E}}_{l},\bm{\bar{F}}_{l+1}-\bm{\bar{F}}_{l})\Vert.$$
Adding this to \eqref{consid} (with $k=l$), one checks that  \eqref{consid} holds for $k=l+1$. Therefore, we  prove \eqref{condi} and \eqref{consid} hold for each $k\ge\N$ by mathematical induction and so, the proof is complete.
\end{proof}

Theorem \ref{thm2} demonstrates  that the sequence $\{(\bm{\sigma}_{k},\bm{E}_{k},\bm{F}_{k})\}$ generated by Algorithm \ref{alg1} converges to a stationary point of the auxiliary function $\mathcal{F}_{\hat{\bm{\Omega}}}$ defined in \eqref{Problem-eq}, rather than that of $\mathcal{F}$ involved in the original problem \eqref{Problem-Lp}. We will further show in Theorem \ref{remark}  that the generated sequence $\{\bm{X}_k\}$ converges to a stationary point of $\mathcal{F}_{\hat{\bm{\Omega}}}$, under the assumption that  all non-zero elements of $\bar{\bm{\sigma}}$ are distinct.  This requires  the following proposition, which will also facilitate the linear convergence analysis for the case
$p=1$ in the subsequent subsection. For convenience, let $\bm{U}^{>}_{X},\; \bm{V}^{>}_{X} $ denote the matrices composed of the left and right singular vectors, respectively,  corresponding to zero singular values of  $\bm{X}\in\mathbb{R}^{m\times n}$. For $\bm{\sigma}:=(\bm{\sigma}_1,\;\bm{\sigma}_2,\ldots,\bm{\sigma}_n)^\top\in\mathbb{R}^{n}$, we write
\begin{equation}\label{taude}\tau(\bm{\sigma}):=\sqrt{2}\max\left\{\frac{\max\limits_{i\in{\rm supp}(\bm{\sigma})}\{\bm{\sigma}_i\}}{\min\limits_{i,\;j\in {\rm supp}(\bm{\sigma})}\{\bm{\sigma}_{j}^{2}-\bm{\sigma}_{i}^{2}\}},\;\frac{1}{\min\limits_{i\in{\rm supp}(\bm{\sigma})}\{\bm{\sigma}_i\}}\right\}.\end{equation}

\begin{prop}\label{DT-v-w} Let $\bm{X}\in\mathbb{R}^{m\times n}$ be with the singular value decomposition $\bm{X}=\bm{U}^\top \mathcal{D}(\bm{\sigma})\bm{V}$. Suppose that
\begin{itemize}
  \item [{\rm(H)}] \text{non-zero  singular values  of } $\bm{X}$ \text{are  distinct, and} $\|\bm{U}^{>}_{X}\nabla f(\bm{X})(\bm{V}^{>}_{X})^\top\|_2\le\lambda$ \text{if} $p=1$.
\end{itemize}
Then it holds that
\begin{equation*}
{\rm d}(0,\partial \mathcal{F}(\bm{X}))\leq(2\tau(\bm{\sigma})+1){\rm d}(\bm{0},\partial \mathcal{F}_{\bm{\Omega}}(\bm{\sigma},\bm{0},\bm{0}));
\end{equation*}
in particular, $\bm{X}$ is a stationary point of $\mathcal{F}$ in the original problem \eqref{Problem-Lp} provided that $(\bm{\sigma},\bm{0},\bm{0})$ is a stationary point of the auxiliary function $\mathcal{F}_{\bm{\Omega}}$,  where $\bm{\Omega}=(\bm{U},\bm{V})$.
 \end{prop}
\begin{proof}Without loss of generality, we may assume that  \begin{equation}\label{sigmaassump}\bm{\sigma}_1>\cdots>\bm{\sigma}_r>\bm{\sigma}_{r+1}=\cdots=\bm{\sigma}_{n}=0.\end{equation}In fact, letting $\bm{\sigma}'$ be a  permutation of $\bm{\sigma}$,
there exist two permutation matrices $\bm{P_U}\in\mathbb{R}^{m\times m}$ and $\bm{P_V}\in\mathbb{R}^{n\times n}$ such that $\mathcal{D}(\bm{\sigma})=\bm{P_U}^\top\mathcal{D}(\bm{\sigma}')\bm{P_V}$. Then,
by \eqref{fparsigma}, one checks that $ \nabla_{\bm{\sigma}} f_{\bm{\Omega}}(\bm{\sigma}',\bm{0},\bm{0})=\bm{P_U}\nabla_{\bm{\sigma}} f_{\bm{\Omega}}(\bm{\sigma},\bm{0},\bm{0})\bm{P_V}^\top$,
$\nabla_{\bm{E}}f_{\bm{\Omega}}(\bm{\sigma}',\bm{0},\bm{0})=\nabla_{\bm{E}}f_{\bm{\Omega}}(\bm{\sigma},\bm{0},\bm{0})$ and $\nabla_{\bm{F}}f_{\bm{\Omega}}(\bm{\sigma}',\bm{0},\bm{0})=\nabla_{\bm{F}}f_{\bm{\Omega}}(\bm{\sigma},\bm{0},\bm{0})$, which implies
${\rm d}(\bm{0},\partial \mathcal{F}_{\bm{\Omega}}(\bm{\sigma}',\bm{0},\bm{0}))={\rm d}(\bm{0},\partial \mathcal{F}_{\bm{\Omega}}(\bm{\sigma},\bm{0},\bm{0}))$.
Thus, the assumption \eqref{sigmaassump}  is reasonable and so we can write 
$\bm{U}:=((\bm{U}^{=}_X)^\top, (\bm{U}^{>}_X)^\top) ^\top$ and $\bm{V}:=((\bm{V}^{=}_X)^\top, (\bm{V}^{>}_X)^\top) ^\top$ where  $\bm{U}^{>}_X\in \mathbb{R}^{(m-r)\times m}$ and $\bm{V}^{>}_X\in \mathbb{R}^{(n-r)\times n}$.

Let $\bm{G}\in\partial \mathcal{F}_{\bm{\Omega}}(\bm{\sigma},\bm{0},\bm{0})$. It suffices to show that there exists $\bm{W}\in\partial  \mathcal{F}(\bm{X})$ such that
\begin{equation}\label{eq-chainrule-con}
\|{\bm W}\|_F\leq(2\tau(\bm{\sigma})+1)\|\bm{G}\|_F.
\end{equation}
To proceed, we define the mapping $\mathcal{T}_{\bm{\Omega}}:\mathcal{M}\rightarrow \mathbb{R}^{m\times n}$ by
\begin{equation*}
\mathcal{T}_{\bm{\Omega}}(\bm{\sigma},\bm{E},\bm{F}):=(\bm{U}e^{\bm{E}})^\top \mathcal{D}(\bm{\sigma}) \bm{V}e^{\bm{F}}\quad{\rm for \;each}\;(\bm{\sigma},\bm{E},\bm{F})\in\mathcal{M}.
\end{equation*}
Then, one sees $\mathcal{F}_{\bm{\Omega}}=\mathcal{F}\circ \mathcal{T}_{\bm{\Omega}}$.
Thus, we check by the chain rule \cite[Theorem 10.6]{rockafellar2009variational} that
$\partial \mathcal{F}_{\bm{\Omega}}(\bm{\sigma},\bm{0},\bm{0})=({\rm D}\mathcal{T}_{\bm{\Omega}}(\bm{\sigma},\bm{0},\bm{0}))^* \partial \mathcal{F}(\bm{X})$.
That is, there exists $\overline{\bm{W}}\in \partial  \mathcal{F}(\bm{X})$ such that
\begin{equation}\label{eq-chainrule}
({\rm D}\mathcal{T}_{\bm{\Omega}}(\bm{\sigma},\bm{0},\bm{0}))^* (\overline{\bm{W}})=\bm{G}.
\end{equation}Note by Proposition \ref{lemma-par} (ii) that
\begin{equation*}
\begin{array}{lll}
\partial \mathcal{F}(\bm{X})= \left\{\bm{U}^\top\left(
                  \begin{array}{cc}
                    \bm{U}^{=}_X\nabla f(\bm{X})(\bm{V}^{=}_X)^\top+\lambda  \mathcal{D}(\partial\Vert \bm{\sigma}_{+}(\bm{X})\Vert_{p}^p) &  \bm{U}^{=}_X\nabla f(\bm{X})(\bm{V}^{>}_X)^\top \\
                    \bm{U}^{>}_X\nabla f(\bm{X})(\bm{V}^{=}_X)^\top & \bm{U}^{>}_X\nabla f(\bm{X})(\bm{V}^{>}_X)^\top+\lambda\bm W^{(22)} \\
                  \end{array}
                \right)
 \bm{V}\right\},
\end{array}
\end{equation*}
where  $\|\bm W^{(22)}\|_2\le\gamma_p$.
Thus, we may assume that
$\overline{\bm{W}}:=\bm{U}^\top\scriptsize\left[\begin{array}{cc}
\overline{\bm{W}}^{(11)} & \overline{\bm{W}}^{(12)} \\
\overline{\bm{W}}^{(21)} & \overline{\bm{W}}^{(22)}
\end{array}\right]\bm{V}$, and then $\bm{W}:=\bm{U}^\top\scriptsize\left[\begin{array}{cc}
\overline{\bm{W}}^{(11)} & \overline{\bm{W}}^{(12)} \\
\overline{\bm{W}}^{(21)} & \bm{0}
\end{array}\right]\bm{V}\in\partial  \mathcal{F}(\bm{X})$
by the fact that $\|\bm{U}^{>}_X\nabla f(\bm{X})(\bm{V}^{>}_X)^\top\|_2\le\lambda\gamma_p$ (see assumption (H) and the definition of $\gamma_p$ in \eqref{gammadef}).

Below we will show that $\bm{W}$ is as desired. To  this end, we claim that there exists $\bm{Y}\in\mathcal{M}$  such that
  \begin{equation}\label{lemma411claim}
  \nabla \mathcal{T}_{\bm{\Omega}}(\bm{\sigma},\bm{0},\bm{0})\bm{Y}=\bm{W}\quad{\rm and}\quad\|\bm{Y}\|\le(2\tau(\bm{\sigma})+1)\|{\bm W}\|_F.
  \end{equation}
Granting this,  we deduce by \eqref{eq-chainrule} that
\begin{align*}
\langle\bm{G},\; \bm{Y}\rangle
=\langle \nabla \mathcal{T}_{\bm{\Omega}}(\bm{\sigma},\bm{0},\bm{0})^*(\overline{\bm{W}}),\; \bm{Y}\rangle
=\langle\overline{\bm{W}},\;\nabla \mathcal{T}_{\bm{\Omega}}(\bm{\sigma},\bm{0},\bm{0})\bm{Y}\rangle
=\langle\overline{\bm{W}},\;\bm{W}\rangle=\|\bm{W}\|_F^2.
\end{align*}
This gives
\begin{align*}
\|\bm{W}\|_F^2\le\|\bm{G}\|_F\|\bm{Y}\|\leq(2\tau(\bm{\sigma})+1)\|{\bm W}\|_F\|\bm{G}\|_F.
\end{align*}
Hence \eqref{eq-chainrule-con} is seen to hold and the proof is completed.
To proceed, using the component form, one checks that  there exist unique $\bm{H}\in\mathcal{S}(m)$ and $\bm{K}\in\mathcal{S}(n)$ with $\bm{H}_{i j}=0$ ($r+1 \leq i,\; j \leq m$) and $\bm{K}_{i j}=0$ ($r+1 \leq i,\; j \leq n$) such that \begin{equation}\label{wcomponent}\bm{H} \mathcal{D}(\bm{\sigma})-\mathcal{D}(\bm{\sigma})\bm{K}+\mathcal{D}({\rm diag}(\bm{U}\bm{W}\bm{V}^\top))=\bm{U}\bm{W}\bm{V}^\top.\end{equation}
This particularly implies
$\|\bm{H}\|_F\leq\tau(\bm{\sigma})\|\bm{W}\|_F$, $\|\bm{K}\|_F\leq\tau(\bm{\sigma})\|\bm{W}\|_F$.
Set $\Delta \bm{\sigma}:={\rm diag}(\bm{U}\bm{W}\bm{V}^\top)$,  $\Delta \bm{E}:=\bm{U}^\top\bm{H}^\top\bm{U}$,
$\Delta \bm{F}:=-\bm{V}^\top\bm{K}\bm{V}$ and $\bm{Y}:=(\Delta\bm{\sigma},\Delta \bm{E},\Delta \bm{F})$. Then, $\bm{Y}$ satisfies the  inequality in \eqref{lemma411claim}. Furthermore, we check by definitions that
  \begin{align*}
  \nabla \mathcal{T}_{\bm{\Omega}}(\bm{\sigma},\bm{0},\bm{0})\bm{Y}
  &=(\bm{U}\Delta \bm{E})^\top \mathcal{D}(\bm{\sigma}) \bm{V}+\bm{U}^\top \mathcal{D}(\Delta \bm{\sigma}) \bm{V}+\bm{U}^\top \mathcal{D}(\bm{\sigma})\bm{V}\Delta \bm{F}.
  \end{align*}
 Thus, by \eqref{wcomponent} and the definition of $(\Delta \bm{\sigma},\;\Delta \bm{E},\;\Delta \bm{F})$, one has that $\bm{Y}$ satisfies the  equality in \eqref{lemma411claim}. Hence,
claim \eqref{lemma411claim} holds and so the proof is completed.
\end{proof}

The following theorem is a direct consequence of  Theorem \ref{thm2} and Proposition \ref{DT-v-w},  which shows that
$\{\bm{X}_k\}$ converges to a stationary point of the function $\mathcal{F}$ involved in the original problem \eqref{Problem-Lp}.
\begin{theorem}\label{remark}  Suppose that the assumption (H) holds  with $\bar{\bm{\sigma}}$, $\bar{\bm{X}}$, $\bm{U}^{>}_{\bar{X}}$, $\bm{V}^{>}_{\bar{X}}$ in place of $\bm{\sigma}$, $\bm{X}$, $\bm{U}^{>}_{X}$,  $\bm{V}^{>}_{X}$.
Then the generated sequence $\{\bm{X}_k\}$ converges to a stationary point of $\mathcal{F}$.

 \end{theorem}

\subsection{Linear Convergence for $p=1$}
In this subsection, we analyze the convergence properties of   Algorithm \ref{alg1} for the special case $p=1$. 
Our analysis relies on the strict complementarity-type regularity condition
\[0\in {\rm ri}\partial \mathcal{F}(\bar{\bm X})\]
at the optimal solution $\bar{\bm X}$, which is also  used in the forward-backward splitting algorithms for solving nuclear norm regularized least-squares problems; see \cite{liang2014local,zhou2017unified,drusvyatskiy2018error} for details.

Recall from Theorem \ref{thm2} that
$\{\bm{X}_k\}$ (resp. $\{(\bm{\sigma}_k,\bm{U}_k,\bm{V}_k)\}$) converges to $\bm{\bar{X}}$  (resp. $(\bm{\bar{\sigma}},\bm{\bar{U}},\bm{\bar{V}})$). In this subsection, we assume that all the non-zero elements of $\bm{\bar{\sigma}}$ are distinct. As in the beginning  of the proof for Proposition \ref{DT-v-w},  we assume without loss of generality that
\begin{equation}\label{sigmabarass}\bm{\bar{\sigma}}_1>\cdots>\bm{\bar{\sigma}}_r>\bm{\bar{\sigma}}_{r+1}=\cdots=\bm{\bar{\sigma}}_{n}=0.\end{equation}
For the following lemma, we write
 $\bar{\bm{U}}:=((\bar{\bm{U}}^{(1)})^\top,\; (\bar{\bm{U}}^{(2)})^\top)^\top$, $\bar{\bm{V}}:=((\bar{\bm{V}}^{(1)})^\top,\; (\bar{\bm{V}}^{(2)})^\top)^\top$
and for each $k\in \N$,  $\bm{U}_k=(( \bm{U}^{(1)}_k)^\top, (\bm{U}^{(2)}_k)^\top )^\top$, $\bm{V}_k=(( \bm{V}^{(1)}_k)^\top, (\bm{V}^{(2)}_k)^\top)^\top$
where  $\bar{\bm{U}}^{(2)},\;\bm{U}^{(2)}_k\in \mathbb{R}^{(m-r)\times m}$ and $\bar{\bm{V}}^{(2)},\;\bm{V}^{(2)}_k\in \mathbb{R}^{(n-r)\times n}$.

\begin{lemma}\label{lemma-supp}Suppose  that $0\in {\rm ri}\partial \mathcal{F}(\bar{\bm X})$. Then
there exists  $N\in\mathbb{N}$ such that
\begin{equation}\label{uklambda}\|(\bm{U}^{(2)}_k)^\top\nabla f(\bm{X}_k)\bm{V}^{(2)}_k\|_2\le\lambda\quad{\rm and}\quad{\rm supp}(\bm{\sigma}_k)={\rm supp}(\bm{\bar{\sigma}})\quad{\rm for\; each }\quad k\ge N.\end{equation}
\end{lemma}

\begin{proof}By Proposition \ref{lemma-par} (ii), one sees
$${\rm ri}\partial \Vert \bar {\bm{X} }\Vert_{S_1}=\left\{\bar{\bm{U}}^\top \left(
                      \begin{array}{cc}
                        \bm{I} & \bm{0} \\
                        \bm{0}  & \bm W^{(22)} \\
                      \end{array}
                    \right)
\bar{\bm{V}}:\bm W^{(22)}\in\mathbb{R}^{(m-r)\times(n-r)}, \;\|\bm W^{(22)}\|_2<1\right\}.$$
Note by assumption that $-\nabla f(\bar{\bm{X}})\in\lambda{\rm ri}\partial \Vert\bar{\bm{X}} \Vert_{S_1}$ (see Proposition \ref{lemma-par} (ii)).
Therefore, one has
\begin{equation}\label{ri}
-\bar{\bm U}\nabla f(\bar{\bm{X}})\bar{\bm V}^\top\in
\lambda\left\{\left(
                      \begin{array}{cc}
                        \bm {I} & \bm{0} \\
                        \bm{0}  & \bm W^{(22)} \\
                      \end{array}
                    \right)
:\bm W^{(22)}\in\mathbb{R}^{(m-r)\times(n-r)}, \;\|\bm W^{(22)}\|_2<1\right\}.
\end{equation}
Using the block form of $\bar{\bm U}$ and $\bar{\bm V}$, we have
\begin{equation*}
\bar{\bm U}\nabla f(\bar{\bm{X}})\bar{\bm V}^\top=\left(
                                                               \begin{array}{cc}
                                                                 \bar{\bm U}^{(1)}\nabla  f(\bar{\bm{X}})(\bar{\bm V}^{(1)} )^\top & \bar{\bm U}^{(1)}\nabla  f(\bar{\bm{X}})(\bar{\bm V}^{(2)} )^\top \\
                                                                 \bar{\bm U}^{(2)}\nabla  f(\bar{\bm{X}})(\bar{\bm V}^{(1)} )^\top & \bar{\bm U}^{(2)}\nabla  f(\bar{\bm{X}})(\bar{\bm V}^{(2)} )^\top \\
                                                               \end{array}
                                                             \right).
                                                            \end{equation*}
This together with \eqref{ri} implies
\begin{equation}\label{u2}\|\bar{\bm U}^{(2)}\nabla  f(\bar{\bm{X}})(\bar{\bm V}^{(2)} )^\top\|_2<\lambda.\end{equation}
Since $\{(\bm{\sigma}_k,\bm{U}_k,\bm{V}_k)\}$  converges to $(\bm{\bar{\sigma}},\bm{\bar{U}},\bm{\bar{V}})$, there exists $N\in\mathbb{N}$ such that the first assertion  in \eqref{uklambda} holds.

To prove the existence of $N$ such that the second assertion in \eqref{uklambda} holds,  we recall from Proposition \ref{prop1} (i) that
$\nabla_{\bm{\sigma}}f_{\bar{\bm{\Omega}}}(\bar{\bm{\sigma}},\bm{0},\bm{0})={\rm diag}\left(\bar {\bm{U}}\nabla f(\bar{\bm{X}})\bar{\bm{V}}^\top\right).$ Then, by Lemma \ref{diag2norm} and \eqref{u2}, one has
\begin{equation}\label{fomegari}|[\nabla_{\bm{\sigma}}f_{\bar{\bm{\Omega}}}(\bar{\bm{\sigma}},\bm{0},\bm{0})]_i|\le\|\bar{\bm U}^{(2)}\nabla  f(\bar{\bm{X}})(\bar{\bm V}^{(2)} )^\top\|_2<\lambda\quad{\rm for \;each}\quad r< i\le n,\end{equation}
Now we suppose on the contrary that  there exist $i_0>r$ and $\{k_i\}$ such that for each $i\in\N$,
$[{\bm \sigma_{k_i}}]_{i_0}\neq0$. By  \eqref{eq:sigma}
and the optimality condition, we have
\begin{equation}\label{knparti}0\in\nabla_{\bm{\sigma}} f_{\bm{\Omega}_{k_i-1}}(\bm{\sigma}_{k_i-1},\bm{0},\bm{0})+\frac{1}{t_{k_i-1}}(\bm{\sigma}_{k_i}-\bm{\sigma}_{k_i-1})+\lambda\partial\|\bm{\sigma}_{k_i}\|_1.\end{equation}
 Without loss of generality, we may assume that $[{\bm\sigma_{k_i}}]_{i_0}>0$. Then, \eqref{knparti} particularly implies
 \begin{equation*}
 \lambda+[\nabla_{\bm{\sigma}} f_{\bm{\Omega}_{k_i-1}}(\bm{\sigma}_{k_i-1},\bm{0},\bm{0})+\frac{1}{t_{k_i-1}}(\bm{\sigma}_{k_i}-\bm{\sigma}_{k_i-1})]_{i_0}=0.
\end{equation*}
Passing to the limit, one checks $$\lambda+[\nabla_{\bm{\sigma}} f_{\bar{\bm{\Omega}}}(\bar{\bm{\sigma}},\bm{0},\bm{0})]_{i_0}=0$$
as $\lim \limits_{i\rightarrow\infty}\nabla_{\bm{\sigma}} f_{\bm{\Omega}_{k_i-1}}(\bm{\sigma}_{k_i-1},\bm{0},\bm{0})=\nabla_{\bm{\sigma}} f_{\bar{\bm{\Omega}}}(\bar{\bm{\sigma}},\bm{0},\bm{0})$ (see \eqref{fparsigma}) and $\lim\limits_{i\rightarrow\infty}\frac{1}{t_{k_i-1}}(\bm{\sigma}_{k_i}-\bm{\sigma}_{k_i-1})=0$ (see Remark \ref{rmk:cluster} and Lemma  \ref{sktkbound}).
It is contradictory  with  \eqref{fomegari} and so there  exists  $N\in\mathbb{N}$ such that for each $k\ge N$,
the second assertion in \eqref{uklambda} holds. The proof is complete.
\end{proof}

\begin{theorem}\label{thm2p=1}
Suppose  that $0\in {\rm ri}\partial \mathcal{F}(\bar{\bm X})$.
Then $\{\bm{X}_k\}$ (resp. $\{\mathcal{F}(\bm{X}_k)\}$) converges linearly to $\bm{\bar{X}}$  (resp. $\mathcal{F}(\bm{\bar{X}}$)).
\end{theorem}

\begin{proof} By Lemma \ref{lemma-supp}, there exists  $N\in\mathbb{N}$ such that for each $k\ge N$, \eqref{uklambda} holds. Note that $\lim\limits_{k\rightarrow\infty}\tau(\bm{\sigma}_{k})=\tau(\bm{\bar{\sigma}})$ by \eqref{taude}, \eqref{sigmabarass} and the fact that  $\lim\limits_{k\rightarrow\infty}\bm{\sigma}_{k}=\bm{\bar{\sigma}}$.
Thus, applying Proposition \ref{DT-v-w} (to $\bm{X}_k$ in place of $\bm{X}$) and using Lemma \ref{prop4-2}, there exists a positive constant $\tau$ such that
\begin{equation}\label{12kl}
{\rm d}(0,\partial \mathcal{F}(\bm{X}_{k}))\leq \tau\Vert (\bm{\sigma}_{k}-\bm{\sigma}_{k-1},\bm{E}_{k-1},\bm{F}_{k-1})\Vert\quad{\rm for\; each}\quad k\ge N.
\end{equation}
Recall that $\{\mathcal{F}(\bm{X}_k)\}$ monotonically decreases and converges to $\mathcal{F}(\bm{\bar{X}})$.
According to \cite[Proposition 12]{zhou2017unified} and \cite[Theorem 4.1]{li2018calculus}, one sees that $\mathcal{F}$ has the   K{\L} property
at $\bm{\bar{X}}$ with the exponent $\frac12$. This, together with \eqref{akdef}, implies that there exists positive constant $\upsilon$ such that $$a_k\le\upsilon{\rm d}^2(0, \partial \mathcal{F}(\bm{X}_k))\quad{\rm for\; each}\quad k\ge0.$$
Combining this  with  \eqref{12kl}  and using Proposition \ref{prop4}, we derive
$$a_k\le\upsilon\tau^2\Vert (\bm{\sigma}_{k}-\bm{\sigma}_{k-1},\bm{E}_{k-1},\bm{F}_{k-1})\Vert^2\le\frac{\upsilon\tau^2}{\alpha}\left(\mathcal{F}(\bm{X}_{k-1})- \mathcal{F}(\bm{X}_{k})\right)\quad{\rm for\; each}\quad k\ge N.$$ Thus, defining
$\omega:=\frac{\upsilon\tau^2}{\alpha+\upsilon\tau^2}$, we deduce
\begin{equation}\label{fxkp=1}a_k\le\omega a_{k-1}\le\cdots\le\omega^{k}a_0\quad{\rm for\; each}\quad k\ge N.\end{equation}
This shows the linear convergence  of $\{\mathcal{F}(\bm{X}_k)\}$.

To show the linear convergence of $\{\bm{X}_{k}\}$, let $k\ge N$. By Lemma \ref {lemma3} (ii), Proposition \ref{prop4}, Remark \ref{rmk:cluster} (ii) and \eqref{fxkp=1}, we have
$$\Vert \bm{X}_{k+1}-\bm{X}_{k}\Vert_F\le\frac{L}{\alpha}(\mathcal{F}(\bm{X}_k)-\mathcal{F}(\bm{X}_{k+1}))^{\frac12}\leq \frac{L}{\alpha}a_k^{\frac12}\le\frac{L}{\alpha}a_0^{\frac12}\omega^{\frac{k}{2}}.$$
Thus,  for all $ l\ge1$, one has
$$\|\bm{X}_{k+1}-\bm{X}_{k+l}\|_F\leq\sum_{j=1}^{l-1}\Vert \bm{X}_{k+j}-\bm{X}_{k+j+1}\Vert_F\leq \frac{L}{\alpha}a_0^{\frac12}\sum_{j=1}^{l-1}\omega^{\frac{k+j}{2}}=\frac{L}{\alpha}a_0^{\frac12}\frac{\omega^{\frac{k+1}{2}}(1-\omega^{\frac{l-1}{2}})}{1-\omega^{\frac12}}.$$
Passing to the limit, one sees that $\{\bm{X}_{k}\}$ converges linearly to $\bm{\bar{X}}$ and so,  the proof is completed.
\end{proof}

\section{Numerical Experiments}\label{sec:numerical}

In this section, we present numerical experiments to evaluate the convergence performance of our proposed algorithm (i.e., DPGA)  for solving the Schatten-$p$ quasi-norm regularized least-squares problem \eqref{Problem-Lp}. We also compare our  algorithm with the fixed point iterative algorithm \cite{Ma2011,peng2017s,Peng2018} (rewritten as FPIA for
convenience), which is based on the classical proximal gradient method and performs an SVD at each iteration.
The numerical experiments are implemented in MATLAB R2013a and the hardware environment is 12th Gen Intel Core i7-12700, @2.10 GHz, 32GB of RAM.

The simulation data are generated as follows. The sensing matrix $\mathcal{A}\in \mathbb{R}^{m \times m}$ and the observation vector $\bm{b}\in\mathbb{R}^{m\times n}$ in problem \eqref{Problem-Lp} are randomly generated 
as sparse matrices by MATLAB's function \texttt{sprand} and possess sparsity levels of $1\%$ and $10\%$, respectively, i.e.,
$$\mathcal{A}:=\texttt{sprand}(m,m,0.01)\quad{\rm and}\quad\bm{b}:=\texttt{sprand}(m,n,0.1).$$
The regularization parameter  $\lambda$ is set to $\lambda=0.01*(m+n).$ We investigate the following cases for $p$: $p=0,\; 1/2,\;2/3,\; 1$ respectively.

In the implementation of DPGA,  the stepsize strategies (i) and (ii) discussed in Subsection \ref{stepsizes} are employed, and the associated DPGAs are denoted by DPGA-I and DPGA-II, respectively. The subproblem \eqref{eq:sigma}  is  explicitly computed for $p = 0,\; 1/2,\; 2/3,\;1$ by using the closed-form solution as described in \cite{Daubechies04,Blumensath08,XuZB12, XuZB23}, and the equations \eqref{eq:U} and \eqref{eq:V} are solved using MATLAB's matrix division function \texttt{mldivide}. For FPIA,  the compact singular value decomposition function \texttt{svd($\cdot$, 'econ')} is utilized to perform the SVD.
The stopping criterion for all algorithms is set as
\begin{description}
  \item[-] the maximum number of iterations is greater than 50000, or
  \item[-] $\operatorname{d}(0,\partial \mathcal{F}(\bm{X}_k))\leq\epsilon$, where $\epsilon$  is a positive number related to $m$, $n$ and $p$ (see Tables \ref{DPGAI-II}--\ref{DPGAI-I-FPIA-P=1}).
\end{description}
Finally, in all experiments, the initial point $\bm{X}_0$ is selected as $\bm{X}_0:=\bm{U}_0^\top \bm{D}(\bm{\sigma}_0)\bm{V}_0=\texttt{sprand}(m,n,0.01)$, and  each experiment consists of 10 independent runs. The performance of algorithms is evaluated by:
\begin{description}
  \item[-]
  Num.: the number of runs out of 10 in which the condition $\operatorname{d}(0,\partial \mathcal{F}(\bm{X}_k))\leq\epsilon$ is satisfied within $50000$ steps and no interruption is encountered due to SVD non-convergence
(specifically, the MATLAB error \texttt{"SVD did not converge"}, which may typically arise from numerical instability and the ill-conditioning of the sparse sensing matrix).
  \item[-] Tim.: the averaged CPU time (in seconds)  over Num. runs;
  \item[-] Rel.: the averaged relative error $\|\bm{X}_k-\bm{X}_{k-1}\|_F/\max\{\|\bm{X}_{k-1}\|_F,1\}$  over Num. runs;
  \item[-] Dis.: the averaged optimality measure $\operatorname{d}(0,\partial \mathcal{F}(\bm{X}_k))$ over Num. runs;
  \item[-] Fva.: the averaged function value $\mathcal{F}(\bm{X}_k)$ over Num. runs.
\end{description}

We first evaluate the numerical performance of DPGA-I and DPGA-II for solving problem \eqref{Problem-Lp} with $m=200$ and $n=100$. For regularization orders $p = 0,\; 1/2,\; 2/3,\;1$,  we  record in Table \ref{DPGAI-II} the
numerical results  of DPGA-I and DPGA-II.
Figure \ref{fig:random_con_all} plots the function value $\mathcal{F}(\bm{X}_k)$ and optimality measure $\operatorname{d}(0,\partial \mathcal{F}(\bm{X}_k))$  along with the iterations generated in one of the 10 runs conducted for Table \ref{DPGAI-II}. Several observations are shown in Table \ref{DPGAI-II} and Figure \ref{fig:random_con_all}:

(i) Both  DPGA-I and DPGA-II converge to an approximate stationary point of problem \eqref{Problem-Lp}, which is consistent with Theorem \ref{thm2} and demonstrates the effectiveness of DPGAs. Moreover, better computational accuracy for stationarity is achieved in terms of $\epsilon$  as the $p$ value increases.

(ii) In the case  when $p=0$, DPGA-I  has a higher success rate than DPGA-II in terms of Num.. 

(iii) DPGA-I requires significantly fewer iterations compared to DPGA-II particularly when  $p=0,\;2/3,\;1$. This is mainly because DPGA-II adopts a conservative stepsize. For example, $\|\bm{A}^\top \bm{A}\|_{\rm op}$ will be very large when the problem size is large, leading to a small step  size and causing DPGA-II to progress slowly.

(iv)  For all $p$'s, except for $p =1/2$, DPGA-I consumes less CPU time than DPGA-II.  A possible reason is that each iteration of DPGA-I may be computationally more expensive due to the backtracking search process in the case  when $p =1/2$.

\begin{table}[H]
\centering
\small
\caption{Numerical results of DPGA-I and DPGA-II for problem \eqref{Problem-Lp} with $m=200$, $n=100$}
\label{DPGAI-II}
\setlength{\tabcolsep}{1.2mm}{
\begin{tabular}{ccccccccccccc}
    \toprule
   \multirow{2}{*}{$p $}& \multirow{2}{*}{$\epsilon $}& \multicolumn{5}{c}{DPGA-I}  & \multicolumn{5}{c}{DPGA-II} \\
 \cmidrule(lr){3-7} \cmidrule(l){8-12}
& &Num. & Dis. & Rel. & Fva. & Tim. &Num.& Dis. & Rel. & Fva. & Tim. \\[0.3em]
   \midrule
    \multirow{1}{*}{$0$}
&5e$-$1   & 10    &5.00e$-$1 &  1.38e$-$3 &  2.57e$+$2 &  2.36e$+$1  &2 & 5.00e$-$1 &  1.01e$-$5&  3.04e$+$2 &     4.82e$+$2\\[0.3em]
    \multirow{1}{*}{$1/2$}
&5e$-$2     &10  &  5.00e$-$2   &1.31e$-$3&  3.02e$+$2 &   3.06e$+$1 &10  &   5.00e$-$2 & 1.83e$-$4&  3.04e$+$2 & 7.20e$+$0\\[0.3em]
         \multirow{1}{*}{$2/3$}
&5e$-$3      &10    &   4.49e$-$3& 9.96e$-$4&   3.08e$+$2 &   1.30e$+$0  &10  &         5.00e$-$3 &   3.82e$-$5 &3.08e$+$2&     4.26e$+$0  \\[0.3em]
    \multirow{1}{*}{$1$}
&5e$-$6   &10    &    4.52e$-$6&     3.65e$-$7&      3.10e$+$2 &3.53e$-$1 & 10   &  4.96e$-$6&        2.88e$-$8&   3.10e$+$2&   2.28e$+$0\\[0.3em]
    \bottomrule
\end{tabular}}
\label{table:random1}
\end{table}

\begin{figure}[h]
\centering
\begin{subfigure}{0.255\linewidth}
  \centering
  \includegraphics[width=\linewidth]{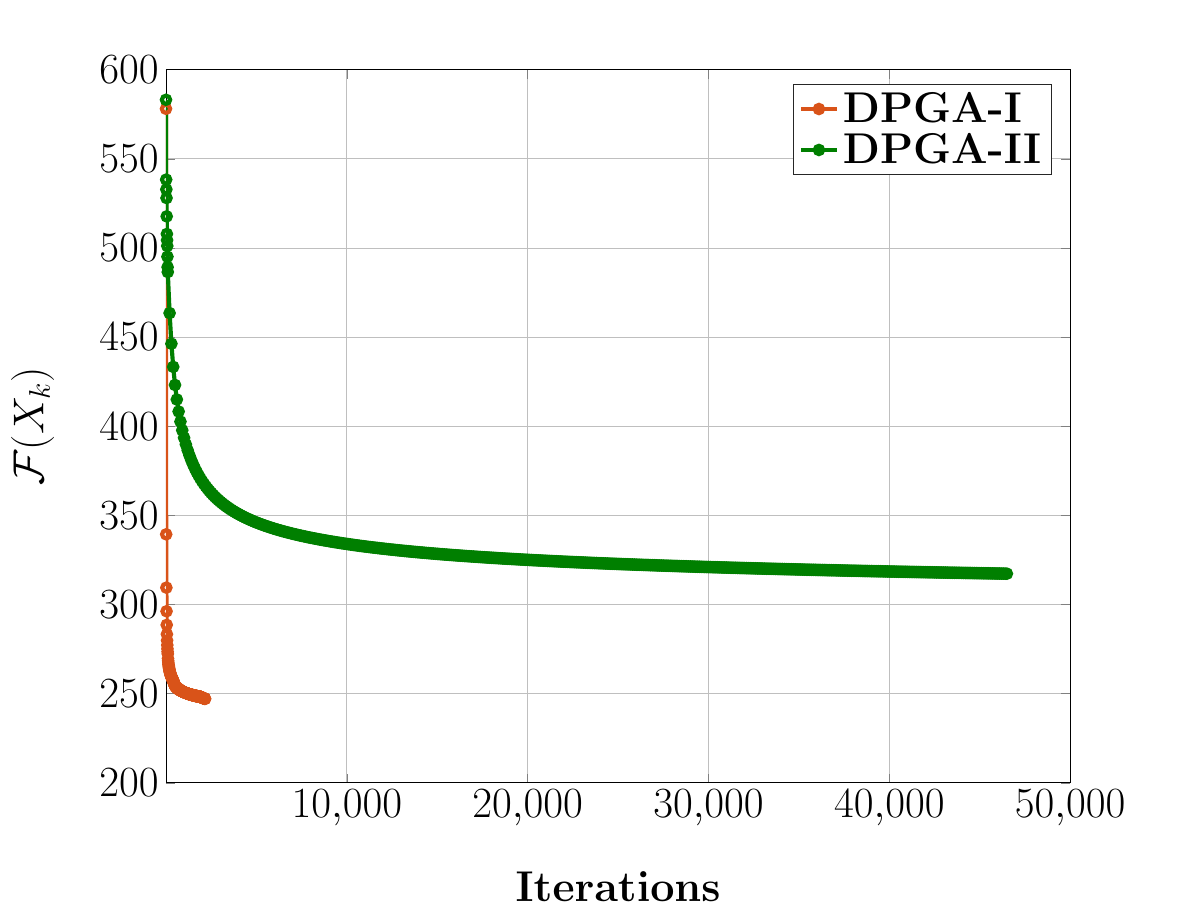}
  \vspace{-0.35cm}
\end{subfigure}\hspace{-0.6em}%
\begin{subfigure}{0.255\linewidth}
  \centering
  \includegraphics[width=\linewidth]{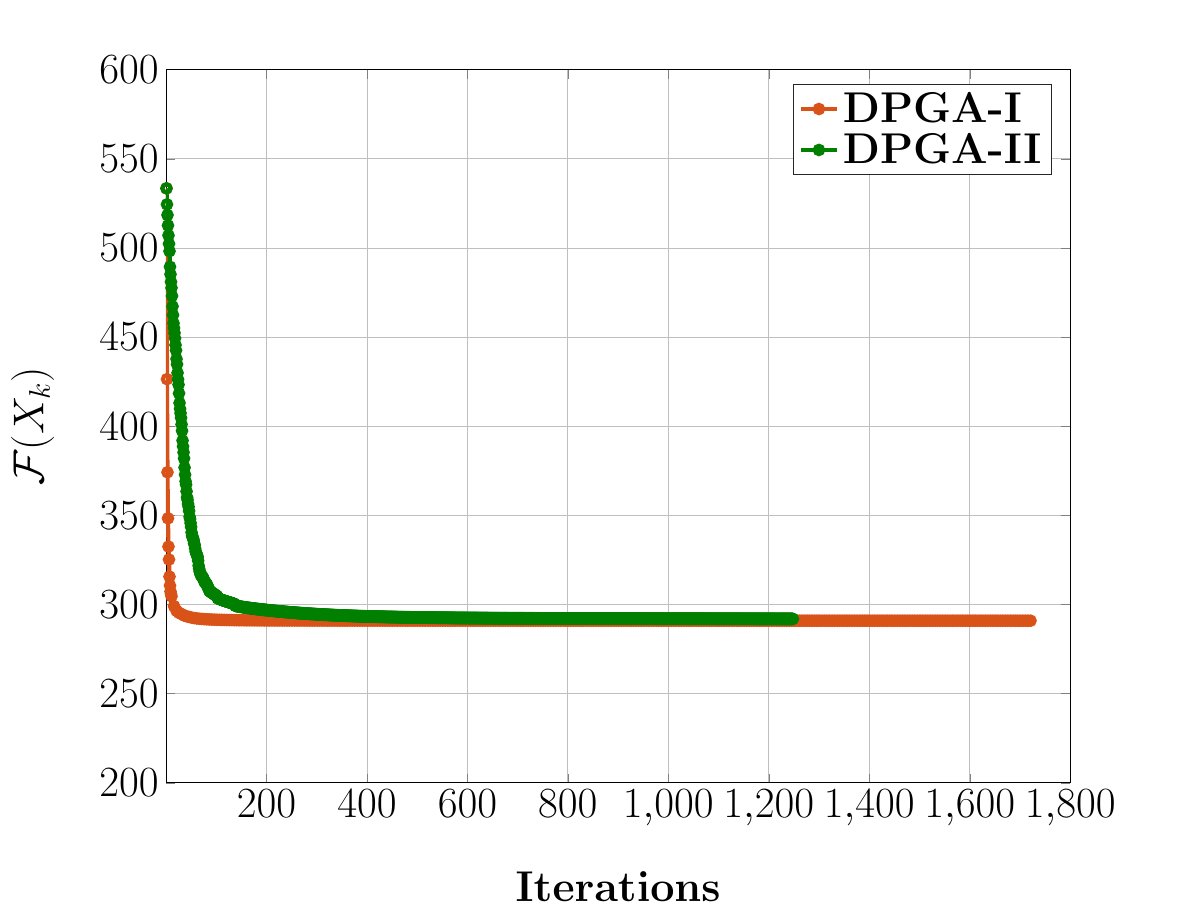}
  \vspace{-0.35cm}
\end{subfigure}\hspace{-0.6em}%
\begin{subfigure}{0.255\linewidth}
  \centering
  \includegraphics[width=\linewidth]{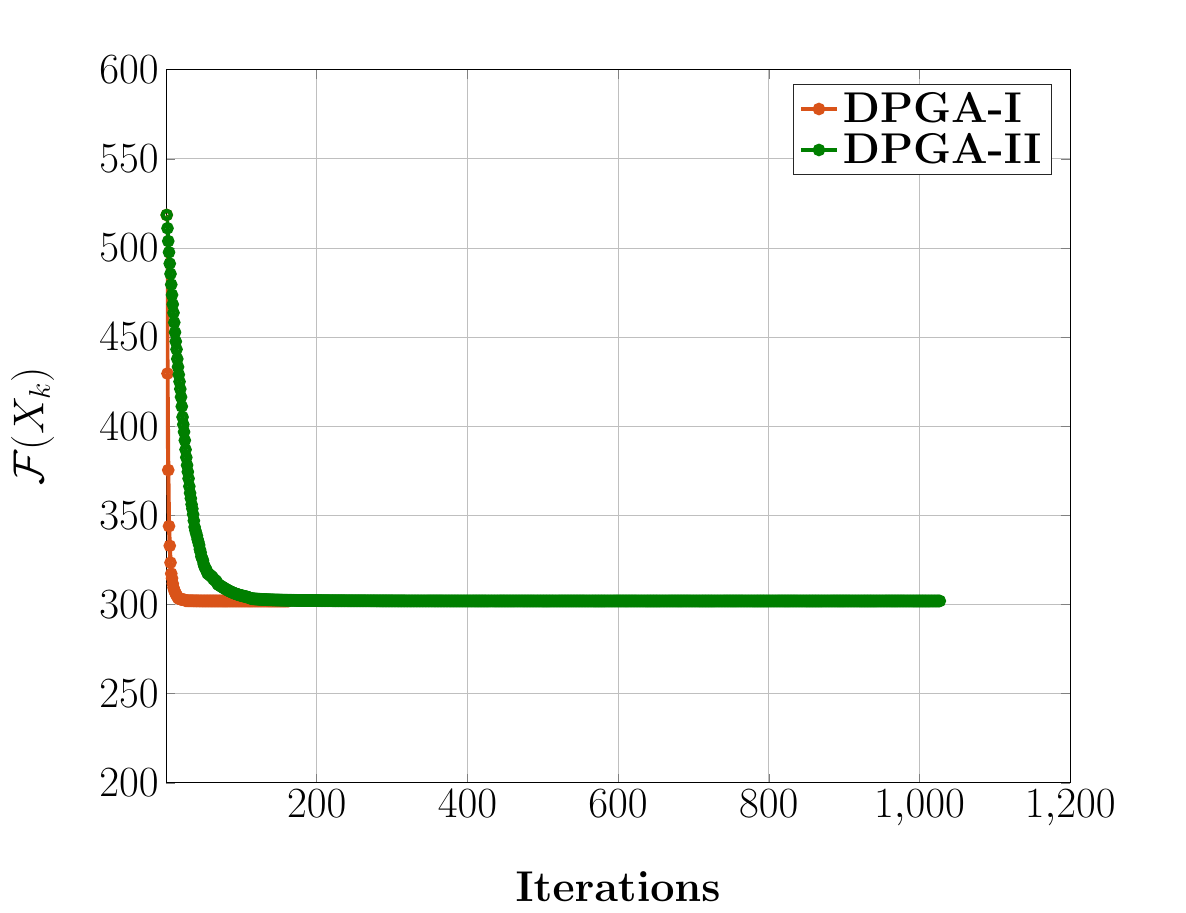}
  \vspace{-0.35cm}
\end{subfigure}\hspace{-0.6em}%
\begin{subfigure}{0.255\linewidth}
  \centering
  \includegraphics[width=\linewidth]{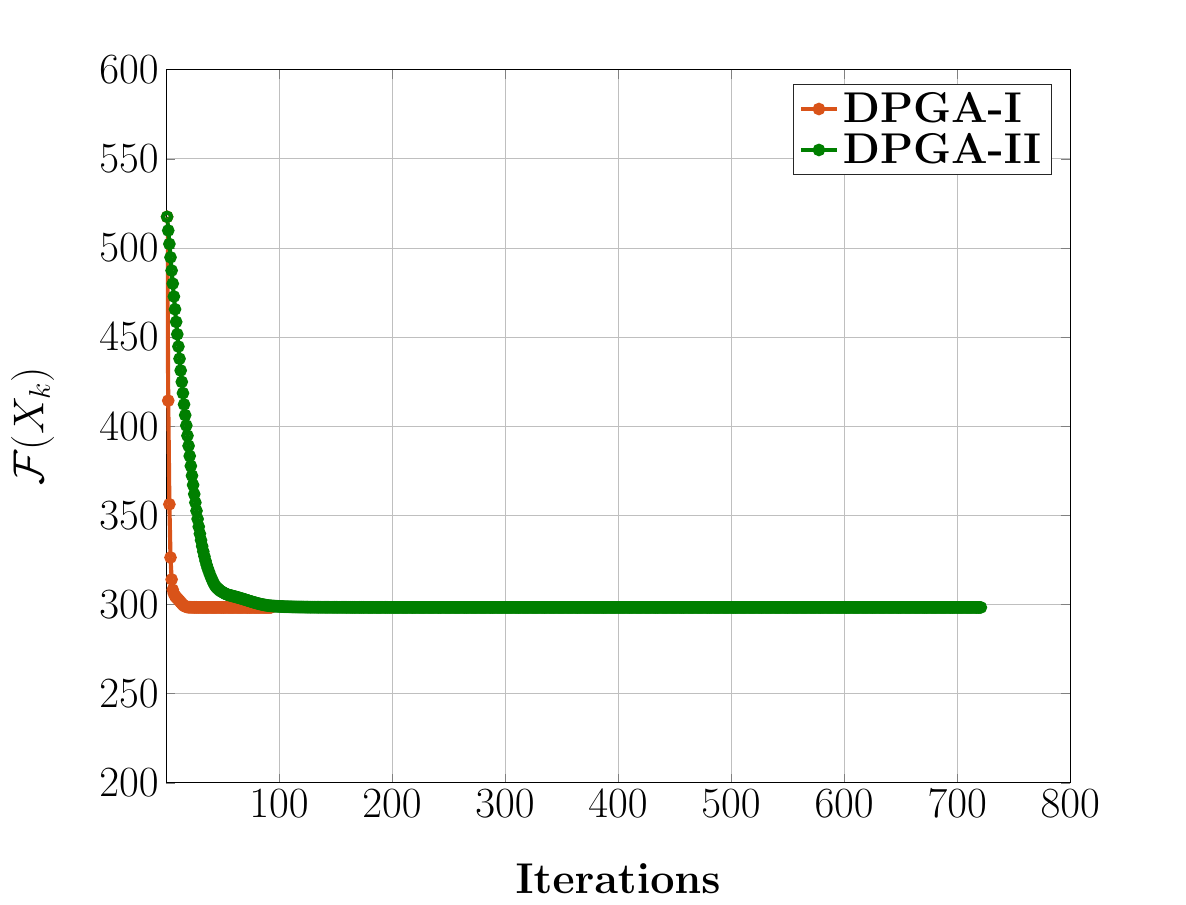}
  \vspace{-0.35cm}
\end{subfigure}

\vspace{-0.15cm}

\begin{subfigure}{0.255\linewidth}
  \centering
  \includegraphics[width=\linewidth]{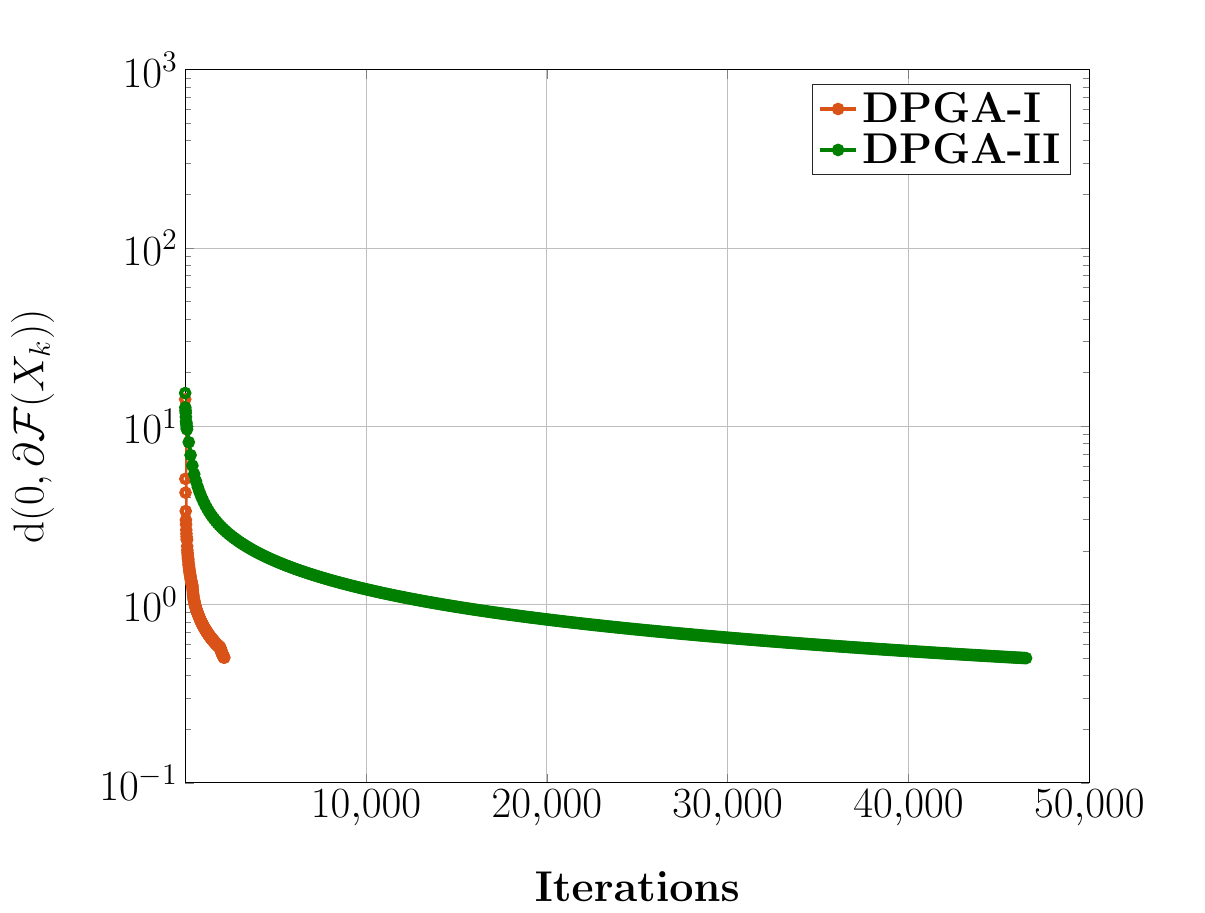}
  \subcaption*{$p=0,\;\epsilon=5e-1$}
  \vspace{-0.35cm}
\end{subfigure}\hspace{-0.6em}%
\begin{subfigure}{0.255\linewidth}
  \centering
  \includegraphics[width=\linewidth]{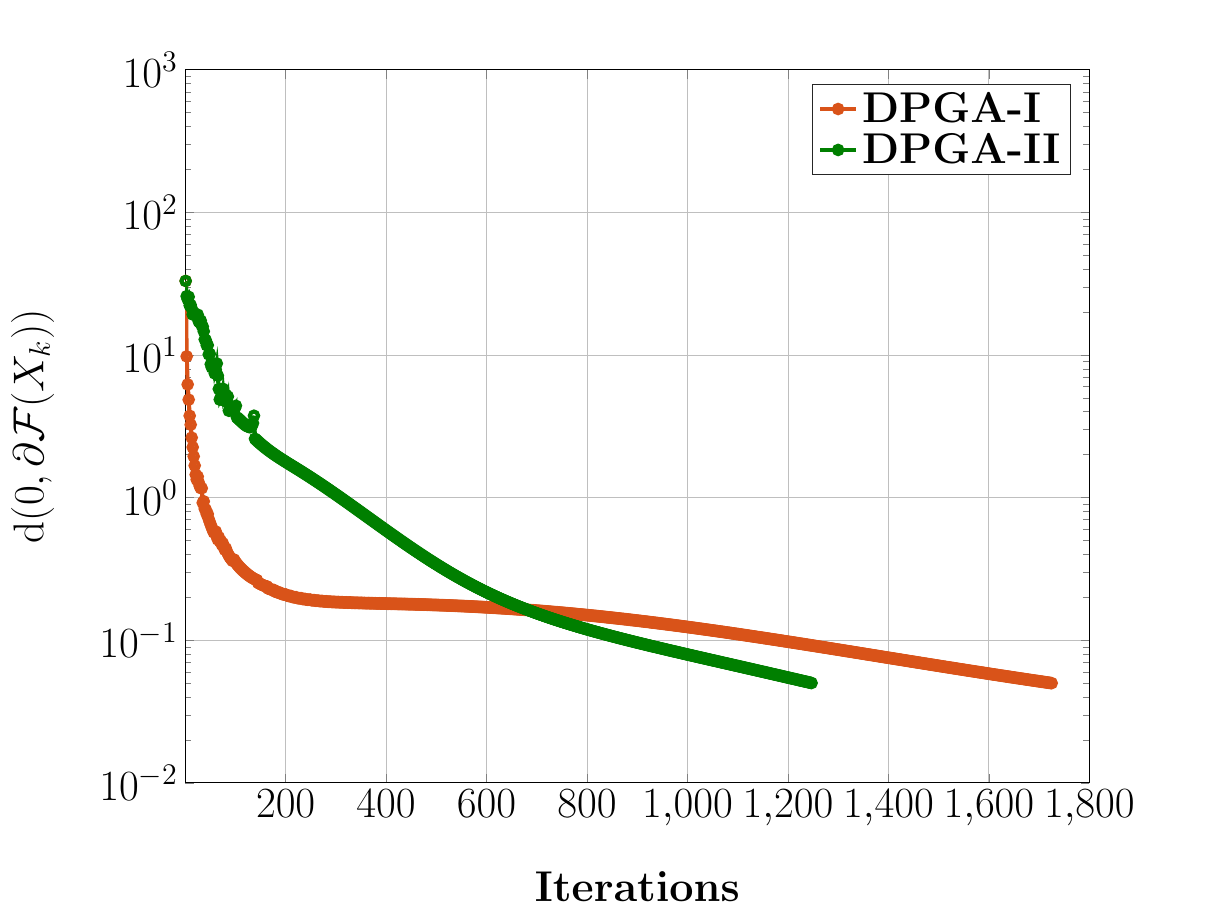}
  \subcaption*{$p=1/2,\;\epsilon=5e-2$}
  \vspace{-0.35cm}
\end{subfigure}\hspace{-0.6em}%
\begin{subfigure}{0.255\linewidth}
  \centering
  \includegraphics[width=\linewidth]{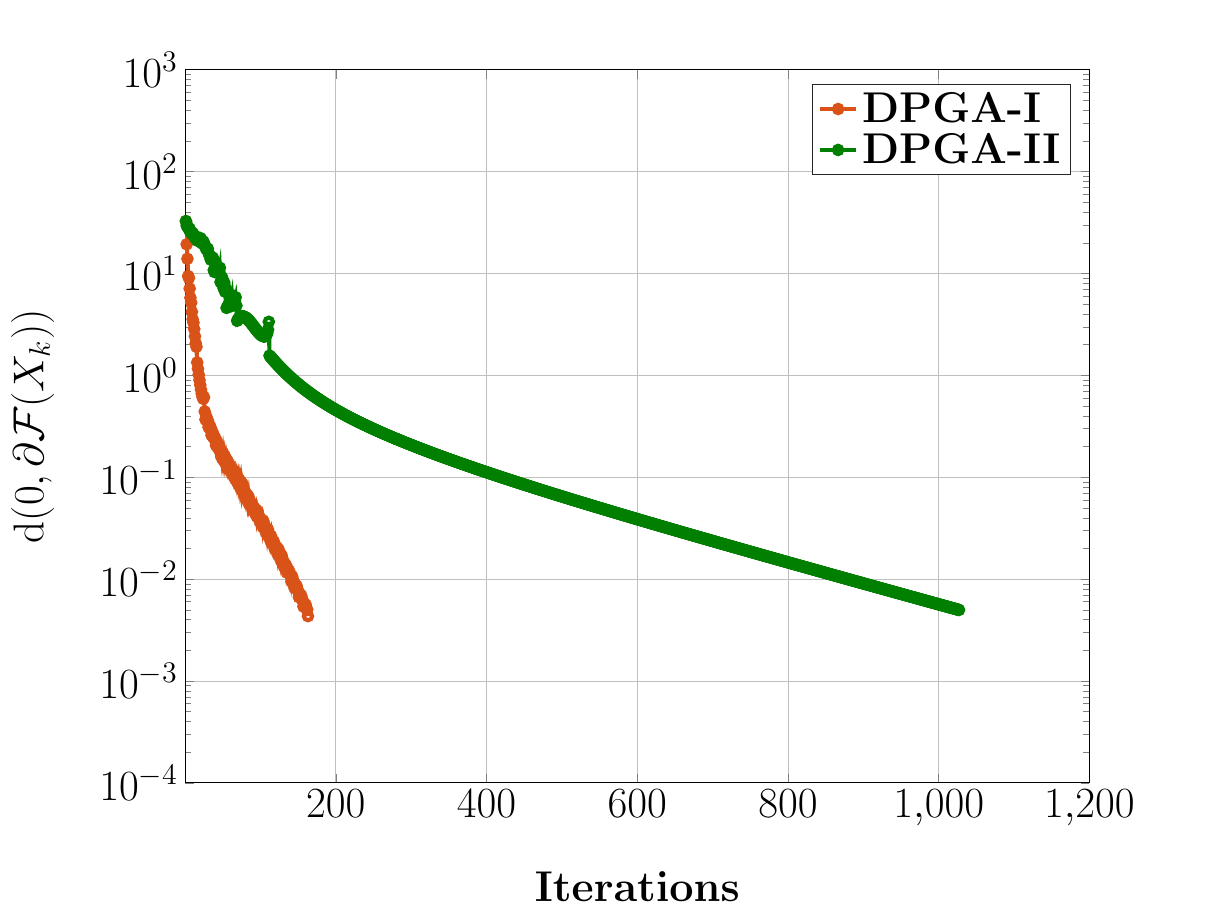}
  \subcaption*{$p=2/3,\;\epsilon=5e-3$}
  \vspace{-0.35cm}
\end{subfigure}\hspace{-0.6em}%
\begin{subfigure}{0.255\linewidth}
  \centering
  \includegraphics[width=\linewidth]{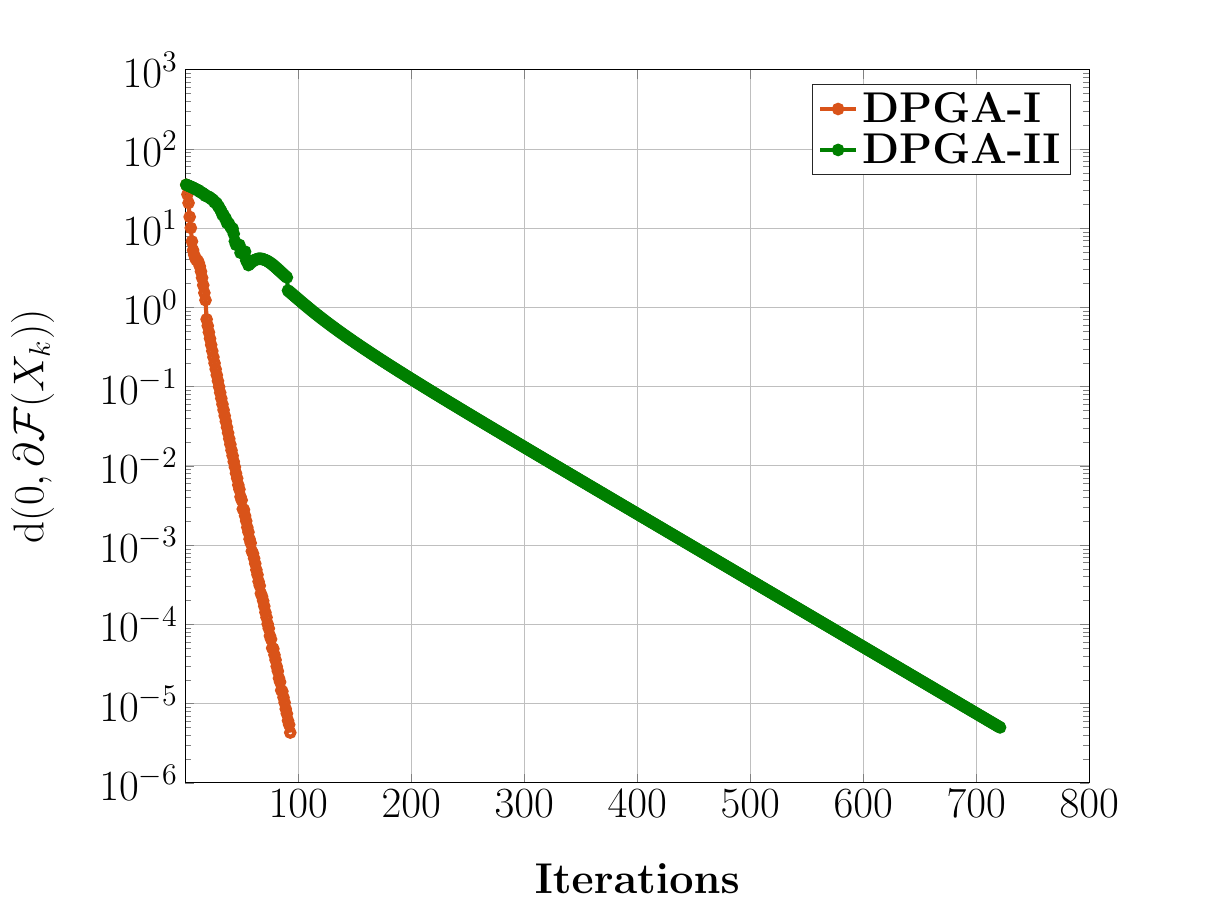}
  \subcaption*{$p=1,\;\epsilon=5e-6$}
  \vspace{-0.35cm}
\end{subfigure}

\vspace{0.35cm}
\caption{Function value and optimality measure for DPGA-I and DPGA-II.}
\vspace{-0.3cm}
\label{fig:random_con_all}
\end{figure}

Next, we compare DPGA with the classical FPIA for solving problem \eqref{Problem-Lp} with different problem sizes $m\times n$. Our evaluation focuses on
$p=1/2,\;2/3,\; 1$. Moreover, inspired by the previous experiments, we implement DPGA-II  for $p=1/2$ and DPGA-I  for $p=2/3,\; 1$. The numerical results of FPIA, DPGA-I and DPGA-II are presented  in Tables \ref{DPGAI-II-FPIA}--\ref{DPGAI-I-FPIA-P=1}. It is revealed in these tables:

(i) DPGA-I and DPGA-II perform better on CPU time than FPIA in all cases. They require at most one tenth of the CPU time consumed by FPIA to approach an approximate stationary point of problem \eqref{Problem-Lp}. Particularly, in the case when $p=1$ and $m\times n=200\times 100$, it requires as little as one two-hundredth of FPIA's time. Consequently, DPGA-I and DPGA-II converge faster than FPIA.

(ii) It is frequently for FPIA to encounter interruptions due to the SVD or to fail to meet accuracy requirement within the prescribed number of iterations, particularly  for large-scale problems $800\times400$ and $1600\times800$ when $p=1/2,\;2/3$. In this sense,
DPGA-I and DPGA-II demonstrate a much higher success rate than FPIA in terms of Num., indicating its superior numerical stability.

\begin{table}[H]
\centering
\small
\caption{Numerical results of FPIA and DPGA-II for different problem sizes with $p=1/2$}
\label{DPGAI-II-FPIA}
\setlength{\tabcolsep}{1.2mm}{
\begin{tabular}{ccccccccccccc}
    \toprule
    \multirow{2}{*}{$m\times n$} &\multirow{2}{*}{$\epsilon$}& \multicolumn{5}{c}{FPIA}  & \multicolumn{5}{c}{DPGA-II} \\
   \cmidrule(lr){3-7} \cmidrule(l){8-12}
 & &Num. & Dis. & Rel.& Fva. & Tim. &Num.& Dis. & Rel. & Fva. & Tim. \\[0.3em]
      \midrule
    \multirow{1}{*}{$200\times 100$}
   &5e$-$2 &10   & 4.99e$-$2  &   3.43e$-$6&   3.07e$+$2&    1.73e$+$2& 10  &5.00e$-$2 &1.27e$-$4 &  3.04$+$2&  1.98e$+$1\\[0.3em]
    \multirow{1}{*}{$400\times 200$}
   &5e$-$2  &8      &  5.00e$-$2 &  2.37e$-$6&  1.22e$+$3&    6.67e$+$2  &10 &  5.00e$-$2 &      4.08e$-$5 &      1.20e$+$3  &     5.92e$+$1\\[0.3em]
    \multirow{1}{*}{$800\times 400$}
   & 5e$-$2 &0     & - & - & - & -  &10 &  5.00e$-$2  & 7.58e$-$6  &  4.75e$+$3 & 1.00e$+$3 \\[0.3em]

            \multirow{1}{*}{$1600\times 800$}
  &5e$-$1 &0     & - & - & - & -  &10 &  5.00e$-$1  & 2.46e$-$5   &  1.89e$+$4 &  3.27e$+$3  \\[0.3em]

    \bottomrule
\end{tabular}}
\label{table:random1}
\end{table}

\begin{table}[H]
\centering
\small
\caption{Numerical results of FPIA and DPGA-I for different problem sizes with $p=2/3$}
\label{DPGAI-I-FPIA}
\setlength{\tabcolsep}{1.2mm}{
\begin{tabular}{ccccccccccccc}
    \toprule
    \multirow{2}{*}{$m\times n$}&\multirow{2}{*}{$\epsilon$} & \multicolumn{5}{c}{FPIA}  & \multicolumn{5}{c}{DPGA-I} \\[0.3em]
   \cmidrule(lr){3-7} \cmidrule(l){8-12}
&  &Num. & Dis. & Rel. &Fva. & Tim.  &Num. & Dis. & Rel. &Fva. & Tim. \\[0.3em]
      \midrule
    \multirow{1}{*}{$200\times 100$}
  &5e$-$3 & 9     & 4.99e$-$3    &   5.98e$-$7     &     3.06e$+$2   &   2.29e$+$2& 10 &   4.89e$-$3 &     9.41e$-$4 &      3.04e$+$2&      4.00e$+$0\\[0.3em]
    \multirow{1}{*}{$400\times 200$}
&5e$-$3 &  5      & 4.98e$-$3  &  4.68e$-$7  &    1.20e$+$3  &       6.02e$+$2 & 10  & 4.84e$-$3 &      3.96e$-$4 &   1.20e$+$3  &    2.06e$+$1  \\[0.3em]
    \multirow{1}{*}{$800\times 400$}
&5e$-$3  &  0     & - & - & - & -     & 10 &    4.77e$-$3  &    2.96e$-$5    &   4.77e$+$3   &        3.28e$+$2  \\[0.3em]
            \multirow{1}{*}{$1600\times 800$}
&5e$-$2 &  0      & - & - & - & -      & 10 &   4.82e$-$2    &   8.33e$-$5   &  1.89e$+$4  &  6.62e$+$3  \\[0.3em]

    \bottomrule
\end{tabular}}
\label{table:random1}
\end{table}

\begin{table}[H]
\centering
\small
\caption{Numerical results of FPIA and DPGA-I for different problem sizes with $p=1$}
\label{DPGAI-I-FPIA-P=1}
\setlength{\tabcolsep}{1.2mm}{
\begin{tabular}{ccccccccccccc}
    \toprule
    \multirow{2}{*}{$m\times n$}&\multirow{2}{*}{$\epsilon$} & \multicolumn{5}{c}{FPIA}  & \multicolumn{5}{c}{DPGA-I} \\[0.3em]
   \cmidrule(lr){3-7} \cmidrule(l){8-12}
& & Num.  & Dis. & Rel.&Fva. & Tim.& Num. & Dis. & Rel. &Fva. &Tim. \\[0.3em]
      \midrule
    \multirow{1}{*}{$200\times 100$}
  &5e$-$6 & 10      &     4.99e$-$6 &      8.67e$-$10&       3.08e$+$2&       6.79e$+$1  & 10   &          4.52e$-$6 &    3.65e$-$7&       3.10e$+$2&        3.49e$-$1\\[0.3em]
    \multirow{1}{*}{$400\times 200$}
&5e$-$5  & 10   &        4.97e$-$5       &     8.32e$-$9    &         1.21e$+$3  &       3.78e$+$2     & 10 &    4.47e$-$5    &      3.81e$-$6     &     1.21e$+$3    &     5.29e$+$1   \\[0.3em]
    \multirow{1}{*}{$800\times 400$}
&5e$-$4  & 10         &      4.96e$-$4   &    7.85e$-$8   &    4.78e$+$3    &    5.45e$+$2      & 10 &     4.52e$-$4   &       6.76e$-$6    &    4.80e$+$3   &    2.76e$+$1 \\[0.3em]

            \multirow{1}{*}{$1600\times 800$}
&5e$-$3   & 10       &      4.97e$-$3  &    6.99e$-$7   &    1.89e$+$4    &      3.60e$+$3  & 10   &    4.53e$-$3  &      1.55e$-$5  &      1.90e$+$4 &   3.89e$+$2\\[0.3em]

    \bottomrule
\end{tabular}}
\label{table:random1}
\end{table}

\section{Conclusions}\label{sec:conclusions}

In this paper, we proposed effective dynamic proximal gradient algorithms for solving Schatten-$p$ quasi-norm regularized problems by mitigating the high computational costs caused by  SVD in each iteration by virtue of Cayley transformations.
The convergence results including convergence rates for the proposed algorithm were established. Numerical tests were also given to illustrate the efficiency of
our algorithm. Future research directions include leveraging the partial smoothness of the $\ell_p$ norm and utilizing the active manifold identified by our method to accelerate the algorithm, such as by incorporating a two-stage second-order approach within our reparameterization framework. Additionally, from an application standpoint, it is crucial to evaluate the performance of DPGA in real-world scenarios and make further adjustments tailored to the specific structures of engineering problems.

\section*{Acknowledgment}

The authors thank Professor Anthony Man-Cho So for helpful discussions and providing useful references.


\section*{\LARGE Appendix}
\begin{appendices}

\section{Basic Facts on Matrices}
In the appendix, we will present some auxiliary facts on matrices. For $\bm{X}\in \mathbb{R}^{n\times n}$, let $e^{\bm{X}}$ and ${\rm log}(\bm{X})$ denote the exponential and logarithmic of $\bm{X}$ which are respectively defined by
\begin{equation*}
e^{\bm{X}}:=\sum_{k=0}^{\infty}\frac{1}{k!}\bm{X}^k\quad\text{and}\quad {\rm log}(\bm{X}):=\sum_{k=1}^{\infty}\frac{(-1)^{k+1}}{k}(\bm{X}-\bm{I})^k.
\end{equation*}
Then,  by   \cite[formular (4)]{Tetsuji}, we have
\begin{align}\label{eX}
e^{\bm{X}+\Delta\bm {X}}=\left(\sum\limits_{k=0}^{\infty}\frac{1}{k!}\mathcal{X}_{k}\right)e^{\bm{X}}=e^{\bm{X}}\left(\sum\limits_{k=0}^{\infty}\frac{1}{k!}\mathcal{Z}_{k}\right),
\end{align}
where $\mathcal{X}_{0}=\mathcal{Z}_{0}=\bm{I} $, $\mathcal{X}_{1}=\mathcal{Z}_{1}=\Delta {\bm X} $, $\mathcal{X}_{k}:=\bm {X}\mathcal{{\bm X}}_{k-1}-\mathcal{\bm X}_{k-1}\bm {X}+\Delta {\bm X}\mathcal{X}_{k-1}$ and
$\mathcal{Z}_{k}:=\mathcal{{\bm Z}}_{k-1}\bm {X}-\bm {X}\mathcal{\bm Z}_{k-1}+\mathcal{Z}_{k-1}\Delta {\bm X}$ ($k\ge2$).
The following lemma illustrates their basic properties. In particular, assertions (i)--(iii) can be found in \cite[Proposition 2.1, Proposition 2.3, Theorem 2.7]{Hall2015Lie}, while assertion (iv) has been proved in \cite[Lemma 2.1]{zhu2023rotation}.

\begin{lemma}\label{lemmaexplog} Let $\bm{X}$, $\bm{Y}\in\mathbb{R}^{n\times n}$. Then we have the following assertions.
\begin{itemize}
  \item [\rm{(i)}]The matrix exponential $e^{\bm{X}}$ is a well-defined  function of $\bm{X}$.
  \item  [\rm{(ii)}]If $\Vert \bm{X}-\bm{I}\Vert_F<1$, ${\rm log}(\bm{X})$ is a well-defined continuous function of ${\bm X}$ and moreover $e^{{\rm log}(\bm{X})}=\bm{X}$.  Suppose further that $\Vert {\bm X}\Vert_F<{\rm log} 2$, then ${\rm log}(e^{\bm X})={\bm X}$.
  \item  [\rm{(iii)}]If  $\bm{X}\bm{Y}=\bm{Y}\bm{X}$, then $e^{\bm{X}+\bm{Y}}=e^{\bm{X}}e^{\bm{Y}}.$
   \item  [\rm{(iv)}] If $\bm{X}\in \mathcal{S}(n)$, then $e^{\bm{X}}\in\mathcal{O}(n)$ and $(3-e)\Vert \bm{X}\Vert_{F}\leq\Vert e^{\bm{X}}-\bm{I}\Vert_{F}\leq\Vert \bm{X}\Vert_{F}$.
   \item  [\rm{(v)}]If  $\bm{X}\bm{Y}=\bm{Y}\bm{X}$, $\Vert \bm{X}-\bm{I}\Vert_F<1$, $\Vert \bm{Y}-\bm{I}\Vert_F<1$ and $\Vert{\rm log}(\bm{X})+{\rm log}(\bm{Y})\Vert_F<{\rm log} 2$, then ${\rm log}(\bm{X}\bm{Y})={\rm log}(\bm{X})+{\rm log}(\bm{Y}).$
  \item  [\rm{(vi)}]If $\bm{X}\in\mathcal{O}(n)$ satisfying $\Vert \bm{X}-\bm{I}\Vert_F<1$ and $\Vert{\rm log}(\bm{X})\Vert_F<\frac{1}{2}{\rm log} 2$, then ${\rm log}(\bm{X})\in\mathcal{S}(n)$.
\end{itemize}
\end{lemma}

\begin{proof}
We only need to prove assertions (v) and (vi). For the proof of assertion (v), suppose that $\bm{X}\bm{Y}=\bm{Y}\bm{X}$, $\Vert \bm{X}-\bm{I}\Vert_F<1$, $\Vert \bm{Y}-\bm{I}\Vert_F<1$  and  $\Vert{\rm log}(\bm{X})+{\rm log}(\bm{Y})\Vert_F<{\rm log} 2$.
Then, thanks to assertion (ii), we know that ${\rm log}(\bm{X})$, ${\rm log}(\bm{Y})$ are well-defined  and  moreover, by the definition of $\bm{X}$, ${\rm log}(\bm{X})$ and ${\rm log}(\bm{Y})$ commutes.
Thus, using (ii) and (iii), one has
\begin{equation}\label{e-XY}
e^{{\rm log}(\bm{X})+{\rm log}(\bm{Y})}=e^{{\rm log}(\bm{X})}e^{{\rm log}(\bm{Y})}=\bm{XY}.
\end{equation}
Note that $\Vert{\rm log}(\bm{X})+{\rm log}(\bm{Y})\Vert_F<{\rm log} 2$. Then, taking  logarithmic for both side of \eqref{e-XY} and using again  (ii), we deduce ${\rm log}(\bm{X})+{\rm log}(\bm{Y})={\rm log}(\bm{XY})$.

 In order to prove assertion (vi), let $\bm{X}\in\mathcal{O}(n)$ satisfy $\Vert \bm{X}-\bm{I}\Vert_F<1$ and $\Vert{\rm log}(\bm{X})\Vert_F<\frac{1}{2}{\rm log} 2$. Then,
 ${\rm log}(\bm{X})$ and ${\rm log}(\bm{X}^\top)$ are well defined. Moreover,  by definition, one has
 $\Vert{\rm log}(\bm{X})+{\rm log}(\bm{X}^\top)\Vert_F\leq 2\Vert{\rm log}(\bm{X})\Vert_F <{\rm log} 2.$ Hence, thanks to assertion (iii) and the orthogonality of $\bm{X}$, it follows that ${\rm log}(\bm{X})+{\rm log}(\bm{X})^\top={\rm log}(\bm{XX}^\top)={\rm log}(\bm{I})=0$, which means ${\rm log}(\bm{X})\in\mathcal{S}(n)$.
\end{proof}

The following lemma demonstrates the Lipschitz continuity  of the exponential and logarithmic mapping.

\begin{lemma}\label{explogcon} Let $\bm{X}\in \mathbb{R}^{n\times n}$ and $\Delta\bm{X}\in \mathbb{R}^{n\times n}$. Then the following two implications hold:
\begin{itemize}
  \item [\rm{(i)}]$2\Vert \bm{X}\Vert_F+\Vert\Delta \bm{X}\Vert_F\le1\Longrightarrow\Vert e^{\bm{X}+\Delta \bm{X}}-e^{{\bm X}}\Vert_F\le2\Vert e^{\bm X}\Vert_F\Vert \Delta \bm {X}\Vert_F;$
  \item  [\rm{(ii)}]$3\Vert \bm{X}-{\bm I}\Vert_F+\Vert\Delta \bm{X}\Vert_F\le\frac12\Longrightarrow\Vert {\rm log}({\bm X}+\Delta {\bm X})-{\rm log}({\bm X})\Vert_F\le2\Vert\Delta \bm{X}\Vert_F.$
\end{itemize}

\end{lemma}
\begin{proof} To prove assertion (i), assume that $2\Vert \bm{X}\Vert_F+\Vert\Delta \bm{X}\Vert_F\le1$. Then, by
mathematical induction, it is easy to check
\begin{equation*}
\Vert\mathcal{X}_{k}\Vert_F\le\Vert\Delta \bm{X}\Vert_F\quad{\rm for \; each}\quad k\ge1.
\end{equation*}
Thus, it follows from \eqref{eX} that
$$\Vert e^{\bm{X}+\Delta\bm {X}}-e^{\bm{X}}\Vert_F=\left\Vert\left (\sum\limits_{k=1}^{\infty}\frac{1}{k!}\mathcal{X}_{k}\right) e^{\bm{X}}\right\Vert_F\le\sum\limits_{k=1}^{\infty}\frac{1}{k!}\Vert e^{\bm{X}}\Vert_F\Vert\Delta\bm {X}\Vert_F\le2\Vert e^{\bm X}\Vert_F\Vert\Delta\bm {X}\Vert_F.$$ This shows  implication (i).

  For the proof of  implication (ii), we  notice  that
\begin{equation}\label{impii}
 {\rm log}({\bm X}+\Delta {\bm X})-{\rm log}({\bm X})=\sum_{k=1}^{\infty}\frac{(-1)^{k+1}}{k}[(\bm{X}+\Delta {\bm X}-\bm{I})^k-(\bm{X}-\bm{I})^k].
 \end{equation}
  Moreover, we set $\mathcal{Y}_{0}:=\bm{I} $, $\mathcal{Y}_{1}:=\Delta {\bm X} $ and $\mathcal{Y}_{l}:=(\bm{X}-\bm{I})\mathcal{{\bm Y}}_{l-1}-\mathcal{\bm Y}_{l-1}(\bm{X}-\bm{I})+\Delta {\bm X}\mathcal{Y}_{l-1}$ ($l\ge2$).
Obviously, for each $ l\ge1$, it holds that
 \begin{equation}\label{ym} \Vert\mathcal{Y}_{l}\Vert_F\le (2\Vert\bm{X}-\bm{I}\Vert_F+\Vert\Delta {\bm X}\Vert_F)\Vert\mathcal{Y}_{l-1}\Vert_F\leq\cdots\leq(2\Vert\bm{X}-\bm{I}\Vert_F+\Vert\Delta {\bm X}\Vert_F)^{l-1}\Vert\Delta {\bm X}\Vert_F.\end{equation}
Fix $k\in\mathbb{N}$ and then,  by \cite[formulae (3)]{Tetsuji}, we have
$$(\bm{X}+\Delta {\bm X}-\bm{I})^k=\sum_{l=0}^{k}\frac{k!}{l!(k-l)!}\mathcal{Y}_l(\bm{X}-\bm{I})^{k-l},$$
which, together with \eqref{ym},  gives
\begin{align}\Vert(\bm{X}+\Delta {\bm X}-\bm{I})^k-(\bm{X}-\bm{I})^k\Vert_F&=\left\Vert\sum\limits_{l=1}^{k}\frac{k!}{l!(k-l)!}\mathcal{Y}_l(\bm{X}-\bm{I})^{k-l}\right\Vert_F\notag\\
&\le\sum\limits_{l=1}^{k}\frac{k!}{l!(k-l)!}(2\Vert\bm{X}-\bm{I}\Vert_F+\Vert\Delta {\bm X}\Vert_F)^{l-1}\Vert\bm{X}-\bm{I}\Vert_F^{k-l}\Vert\Delta {\bm X}\Vert_F.\notag
\end{align}
It follows that
\begin{align*}\Vert(\bm{X}+\Delta {\bm X}-\bm{I})^k-(\bm{X}-\bm{I})^k\Vert_F
\le k(3\Vert\bm{X}-\bm{I}\Vert_F+\Vert\Delta {\bm X}\Vert_F)^{k-1}\Vert\Delta {\bm X}\Vert_F.
\end{align*}
Combining this with \eqref{impii} and noting the fact that $3\Vert \bm{X}-{\bm I}\Vert_F+\Vert\Delta \bm{X}\Vert_F\le\frac12$, we have
$$\Vert {\rm log}({\bm X}+\Delta {\bm X})-{\rm log}({\bm X})\Vert_F\le\sum_{k=1}^{\infty}(3\Vert\bm{X}-\bm{I}\Vert_F+\Vert\Delta {\bm X}\Vert_F)^{k-1}\Vert\Delta {\bm X}\Vert_F\le2\Vert\Delta {\bm X}\Vert_F.$$
Hence we complete the proof.
\end{proof}

Recall that $\C(\cdot)$ is defined by \eqref{mapC}. Then we have the following lemma.
\begin{lemma}\label{fact3}Let $\bm{X}\in\mathcal{S}(n)$. Then $\mathcal{C}(\bm{X})\in\mathcal{O}(n)$ and, if $\|\bm{X}\|_F\le1$, we have
$$\|\mathcal{C}(\bm{X})-e^{\bm{X}}\|_F\le\|\bm{X}\|_F^3,\quad\frac23\|\bm{X}\|_F\le\|\mathcal{C}(\bm{X})-\bm{I}\|_F\le2\|\bm{X}\|_F.$$
\end{lemma}

\begin{proof}We only need to prove that \begin{equation}\label{cx-I}\|\mathcal{C}(\bm{X})-\bm{I}\|_F\ge\frac23\|\bm{X}\|_F\end{equation} as the other conclusions are
 known in \cite[Lemma 2.4]{Wang}. In fact, by \eqref{mapC}, one checks that $(\mathcal{C}(\bm{X})-\bm{I})(\bm{I}-\frac12\bm{X})=\bm{X}$, which gives
$$\Vert\mathcal{C}(\bm{X})-\bm{I}\Vert_F\ge\frac{\Vert\bm{X}\Vert_F}{\Vert\bm{I}-\frac12\bm{X}\Vert_F}.$$
Then  \eqref{cx-I} is seen to hold as $\Vert\bm{X}\Vert_F\le1$ and so, the proof is complete.
\end{proof}

We end the appendix with the following lemma which demonstrates the relationship between the diagonal elements and 2-norm of a matrix.
\begin{lemma}\label{diag2norm}Let $\bm{X}:=(x_{ij})\in\mathbb{R}^{n\times n}$, then
$$|x_{ii}|\le\|\bm{ X}\|_2\quad{\rm for\; each}\quad 1\le i\le n.$$
\end{lemma}
\begin{proof}Let $1\le i\le n$ and   $e_i$ be the $i$th column of the identity matrix $\bm{I}$. Then, we have
 $$|x_{ii}|\le\sqrt{\sum\limits_{l=1}^nx_{li}^2}=\|\bm{X}e_i\|\le\|\bm{X}\|_2,$$
 where the last inequality holds because of the definition of the 2-norm for matrices.
 \end{proof}

\end{appendices}

\addcontentsline{toc}{chapter}{Bibliography}
%
\nocite{}
\bibliography{references}
\bibliographystyle{abbrv}

\end{document}